\documentclass{amsart}

\usepackage{amsmath, amssymb, amsthm, amsfonts, color}

\input xy
\xyoption{all}

\PassOptionsToPackage{linktocpage}{hyperref}
\usepackage{hyperref}

\swapnumbers
\numberwithin{equation}{section}

\theoremstyle{plain}
\newtheorem{theorem}[subsubsection]{Theorem}
\newtheorem{lemma}[subsubsection]{Lemma}
\newtheorem{prop}[subsubsection]{Proposition}
\newtheorem{cor}[subsubsection]{Corollary}

\newtheorem{fact}[subsubsection]{Fact}

\newtheorem*{claim}{Claim}

\theoremstyle{definition}
\newtheorem{defn}[subsubsection]{Definition}

\newtheorem{remark}[subsubsection]{Remark}
\newtheorem{exam}[subsubsection]{Example}

\newtheorem{warning}[subsubsection]{Warning}

\def\AA{\mathbb{A}}
\def\BB{\mathbb{B}}
\def\CC{\mathbb{C}}
\def\DD{\mathbb{D}}

\def\FF{\mathbb{F}}
\def\GG{\mathbb{G}}

\def\PP{\mathbb{P}}
\def\QQ{\mathbb{Q}}
\def\RR{\mathbb{R}}

\def\ZZ{\mathbb{Z}}

\newcommand\cA{\mathcal{A}}
\newcommand\cB{\mathcal{B}}
\newcommand\cC{\mathcal{C}}
\newcommand\cD{\mathcal{D}}
\newcommand\cE{\mathcal{E}}
\newcommand\cF{\mathcal{F}}
\newcommand\cG{\mathcal{G}}

\newcommand\cK{\mathcal{K}}
\newcommand\cL{\mathcal{L}}
\newcommand\cM{\mathcal{M}}
\newcommand\cN{\mathcal{N}}
\newcommand\cO{\mathcal{O}}

\newcommand\cU{\mathcal{U}}
\newcommand\cV{\mathcal{V}}

\def\bI{\mathbf{I}}

\def\bP{\mathbf{P}}
\def\bQ{\mathbf{Q}}
\def\bR{\mathbf{R}}

\newcommand\frA{\mathfrak{A}}

\newcommand\frC{\mathfrak{C}}

\newcommand\frM{\mathfrak{M}}

\newcommand\frR{\mathfrak{R}}

\newcommand\fra{\mathfrak{a}}
\newcommand\frb{\mathfrak{b}}
\newcommand\frc{\mathfrak{c}}

\newcommand\frf{\mathfrak{f}}
\newcommand\frg{\mathfrak{g}}
\newcommand\frh{\mathfrak{h}}

\newcommand\frl{\mathfrak{l}}
\newcommand\fm{\mathfrak{m}}
\newcommand\frn{\mathfrak{n}}
\newcommand\frp{\mathfrak{p}}

\newcommand\frt{\mathfrak{t}}

\newcommand\frz{\mathfrak{z}}

\newcommand\tilW{\widetilde{W}}

\newcommand\aff{\textup{aff}}

\newcommand\AS{\textup{AS}}

\newcommand\aut{\mathfrak{aut}}

\newcommand{\Bun}{\textup{Bun}}

\newcommand{\ch}{\textup{char}}

\newcommand{\codim}{\textup{codim}}
\newcommand{\Coh}{\textup{Coh}}
\newcommand{\coker}{\textup{coker}}
\newcommand{\cont}{\textup{cont}}

\newcommand\ev{\textup{ev}}

\newcommand\Fr{\textup{Fr}}
\newcommand\Frac{\textup{Frac}}
\newcommand\Frob{\textup{Frob}}
\newcommand\Gal{\textup{Gal}}

\newcommand{\Gr}{\textup{Gr}}

\newcommand\IC{\textup{IC}}
\newcommand\id{\textup{id}}
\renewcommand{\Im}{\textup{Im}}

\newcommand\Lie{\textup{Lie}\ }

\newcommand\nil{\textup{nil}}

\newcommand{\ord}{\textup{ord}}

\newcommand\pr{\textup{pr}}

\newcommand\pt{\textup{pt}}

\newcommand{\red}{\textup{red}}

\newcommand\Rep{\textup{Rep}}

\newcommand\res{\textup{res}}
\newcommand\rk{\textup{rk}}
\newcommand\rs{\textup{rs}}

\newcommand\sgn{\textup{sgn}}

\newcommand\Span{\textup{Span}}
\newcommand\Spec{\textup{Spec}\ }

\newcommand\st{\textup{st}}

\newcommand{\tors}{\textup{tors}}

\newcommand{\Tr}{\textup{Tr}}

\newcommand\uni{\textup{uni}}

\newcommand{\univ}{\textup{univ}}

\newcommand{\Vect}{\textup{Vect}}

\newcommand\Aut{\textup{Aut}}
\newcommand\Hom{\textup{Hom}}
\newcommand\End{\textup{End}}

\newcommand\GL{\textup{GL}}
\newcommand\gl{\mathfrak{gl}}
\newcommand\PGL{\textup{PGL}}

\newcommand\SL{\textup{SL}}

\newcommand\SO{\textup{SO}}

\newcommand\Sp{\textup{Sp}}
\newcommand\Spin{\textup{Spin}}

\newcommand{\Gm}{\GG_m}
\def\Ga{\GG_a}

\newcommand{\ad}{\textup{ad}}
\newcommand{\Ad}{\textup{Ad}}
\renewcommand\sc{\textup{sc}}
\newcommand{\der}{\textup{der}}

\newcommand\xch{\mathbb{X}^*}
\newcommand\xcoch{\mathbb{X}_*}

\newcommand{\incl}{\hookrightarrow}
\newcommand{\isom}{\stackrel{\sim}{\to}}
\newcommand{\bij}{\leftrightarrow}
\newcommand{\surj}{\twoheadrightarrow}

\newcommand{\Qlbar}{\overline{\QQ}_\ell}
\newcommand{\Qlb}[1]{\overline{\QQ}_{\ell,{#1}}}

\newcommand{\const}[1]{\overline{\QQ}_{\ell,#1}}
\newcommand{\twtimes}[1]{\stackrel{#1}{\times}}

\renewcommand{\j}[1]{\langle{#1}\rangle}
\newcommand{\wt}[1]{\widetilde{#1}}
\newcommand{\wh}[1]{\widehat{#1}}
\newcommand\quash[1]{}
\newcommand\un{\underline}
\newcommand{\bu}{\bullet}
\newcommand{\ov}{\overline}
\newcommand{\bs}{\backslash}
\newcommand{\tl}[1]{[\![#1]\!]}
\newcommand{\lr}[1]{(\!(#1)\!)}
\newcommand\sss{\subsubsection}
\newcommand\xr{\xrightarrow}
\newcommand\op{\oplus}
\newcommand\ot{\otimes}

\newcommand\one{\mathbf{1}}
\newcommand{\sslash}{\mathbin{/\mkern-6mu/}}
\renewcommand\c{\circ}
\newcommand\sne{\subsetneq}
\newcommand\remove{\setminus}

\newcommand{\homog}[2]{\textup{H}_{#1}({#2})}  
\newcommand{\cohog}[2]{\textup{H}^{#1}({#2})}     
\newcommand{\cohoc}[2]{\textup{H}_{c}^{#1}({#2})}     

\newcommand\upH{\textup{H}}

\renewcommand\a\alpha
\renewcommand\b\beta
\newcommand\G\Gamma
\newcommand\g\gamma
\renewcommand\d\delta
\newcommand\D\Delta
\newcommand{\e}{\epsilon}
\newcommand{\io}{\iota}
\renewcommand{\k}{\kappa}
\renewcommand{\th}{\theta}

\newcommand{\ph}{\varphi}
\renewcommand{\r}{\rho}
\newcommand{\s}{\sigma}
\renewcommand{\t}{\tau}

\newcommand{\y}{\eta}
\newcommand{\z}{\zeta}

\newcommand{\vep}{\varepsilon}
\newcommand{\vp}{\varpi}
\renewcommand{\l}{\lambda}
\renewcommand{\L}{\Lambda}
\newcommand{\om}{\omega}

\newcommand{\Sig}{\Sigma}

\newcommand\na{\natural}
\newcommand\sh{\sharp}

\newcommand\vn{\varnothing}

\newcommand{\AI}{\textup{AI}}

\newcommand{\Wa}{W_{\aff}}

\newcommand\colim{\textup{colim}}

\newcommand\FT{\textup{FT}}

\newcommand\frev{\mathfrak{ev}}

\newcommand\sst{\textup{sst}}
\newcommand\tst{\theta\textup{-st}}
\newcommand\Fun{\textup{Fun}}
\newcommand\gp{\textup{gp}}
\newcommand\Spf{\textup{Spf}\ }
\newcommand\et{\textup{\'et}}
\newcommand\Part{\textup{Part}}
\newcommand\goodp{pretty good\ }
\newcommand\isu{\textup{isu}}

\newcommand\coreg{\textup{coreg}}

\title{Counting absolutely indecomposable $G$-bundles}

\author{Konstantin Jakob}
\thanks{K.J. acknowledges support (through Timo Richarz) by the European Research Council (ERC) under Horizon Europe (grant agreement no 101040935), by the Deutsche Forschungsgemeinschaft (DFG, German Research Foundation) TRR 326 \textit{Geometry and Arithmetic of Uniformized Structures}, project number 444845124 and the LOEWE professorship in Algebra, project number LOEWE/4b//519/05/01.002(0004)/87}

\author{Zhiwei Yun}
\thanks{Z.Y. is supported by the Simons Investigator grant.}

\begin{document}

\begin{abstract}
For a reductive group $G$ over a finite field $k$, and a smooth projective curve $X/k$, we give a motivic counting formula for the number of absolutely indecomposable $G$-bundles on $X$. We prove that the counting can be expressed via the cohomology of the moduli stack of stable parabolic $G$-Higgs bundles on $X$. This result generalizes work of Schiffmann and work of Dobrovolska, Ginzburg, and Travkin from $\GL_n$ to a general reductive group. Along the way we prove some structural results on automorphism groups of $G$-torsors, and we study certain Lie-theoretic counting problems related to the case when $X$ is an elliptic curve -- a case which we investigate more carefully following Fratila, Gunningham and P. Li. 
\end{abstract}

\maketitle

\tableofcontents

\section{Introduction}

\subsection{The counting problem}\label{ss:intro counting}

Let $G$ be a connected reductive (not necessarily split) group over a finite field $k=\FF_{q}$ of characteristic $p$. Let  $X$ be a smooth projective and geometrically connected curve over $k$. Let $\ov k$ be an algebraic closure of $k$.

A vector bundle $\cV$ over $X$ is called {\em absolutely indecomposable} if, base changed to $X_{\ov k}$, $\cV$ cannot be written as a direct sum of two nonzero vector bundles over $X_{\ov k}$. This notion of absolute indecomposability can be generalized to principal $G$-bundles. For a principal $G$-bundle $\cE$ over $X$, $\cE$ is called {\em absolutely indecomposable}, if the neutral component of the algebraic group $\Aut(\cE)$ is unipotent mod center. The goal of this paper is to express the number of absolutely indecomposable $G$-bundles over $X$ in motivic terms.

More precisely, we are concerned with the number
\begin{equation*}
\AI^{\om}_{G,X}(k)
\end{equation*}
of isomorphism classes of absolutely indecomposable $G$-bundles $\cE$ over $X_{\ov k}$ ``with degree $\om$'' that can be defined over $X$. Here $\om$ indexes a connected component of the moduli stack $\Bun_{G}$ of $G$-bundles on $X$. Being able to descend to $X$ is a property of $\cE$ and not extra structure in this counting problem. 

Before stating the result, we remark why this counting problem is not obviously motivic. First, we are not weighting our counting problem by the reciprocal of the cardinality of relevant automorphism groups, as is usually the case in point-counting on a stack. Second, in the moduli stack $\Bun_{G}$,  the locus of absolutely indecomposable bundles is merely a constructible subset, and not a locally closed substack in general. 

Previously known results focus on the case $G=\GL_{n}$, and these results suggest that counting $\AI^{\om}_{G,X}(k)$ is related to counting stable Higgs bundles on $X$. In \cite{Sch} Schiffmann proves that the number of absolutely indecomposable rank $n$ vector bundles on $X$ with degree $d$ (assumed coprime to $n$) is the same as the number of stable Higgs bundles on $X$ with the same rank and degree.  In \cite{DGT}, Dobrovolska, Ginzburg and Travkin extend Schiffmann's result to rank $n$ absolutely indecomposable vector bundles of any degree $d$, where Higgs bundles need to be replaced by parabolic Higgs bundles when $d$ is not coprime to $n$. 

In this paper we extend the results of Schiffmann and Dobrovolska, Ginzburg, and Travkin to an arbitrary reductive group $G$ over $k$, using a different method. 

Schiffmann actually explicitly calculates the number of absolutely indecomposable vector bundles, and uses the relation to Higgs bundles to compute the Betti numbers of the moduli space of stable Higgs bundles via the Weil conjectures. Originally, this idea was used by Hausel and Rodriguez-Villegas \cite{HRV} to compute $E$-polynomials of character varieties through point counting. Counting absolutely indecomposable vector bundles can be understood as a global analogue of counting indecomposable quiver representations, which were studied by Kac \cite{Kac80}, \cite{Kac82}, \cite{Kac83}. Kac proved that this number is given by a polynomial in $q$, and conjectured that its coefficients are positive integers. This is often referred to as Kac' positivity conjecture. The first step towards a proof was given in work of Crawley-Boevey and van den Bergh \cite{CBvdB}, who use deformed preprojective algebras to count indecomposable quiver representations. Their work serves as inspiration for the general counting strategy used in this paper.


\subsection{The main results}
To simplify notation, we assume that $G$ is semisimple in this subsection.

Let $\AI_{G,X}(k)$ be the set of isomorphism classes of absolutely indecomposable $G$-bundles over $X_{\ov k}$ that can be defined over $X$.

Our first main result expresses $\#\AI_{G,X}(k)$ in terms of point-counting on the moduli stack of $G$
-Higgs bundles with a simple pole at a point. More precisely, we choose a $k$-point $x\in X(k)$, and let $\cM_{G,x}$ be the moduli stack of pairs $(\cE, \ph)$ where $\cE$ is a principal $G$-bundle on $X$ and $\ph$ is a rational section of the $\om_{X}$-twisted coadjoint bundle $\Ad^{*}(\cE)\ot \om_{X}$ with a simple pole at $x$ and no poles elsewhere. Let $\frc^*=\frg^*\sslash G$. For $c\in \frc^*$, let $\cM_{G,x}(c)$ be the substack of $\cM_{G,x}$ where the residue of $\ph$ at $x$ (which defines a coadjoint orbit in $\frg^*$) maps to $c$.

To state the first result, we need the notion of an {\em admissible collection}, which is a collection of linear functions $\l_w:\frt_w\to k$ on the various Cartan subalgebras $\frt_w\subset \frg$ of splitting type $w\in W$ (the Weyl group of $G$), satisfying certain compatibility conditions (see Definition \ref{def:adm}). We also impose some genericity conditions on each $\l_w$ called $W$-coregular and $W$-regular (see Definitions \ref{d:W coreg} and \ref{def:reg}). When $q=\#k$ is large enough (depending only on $W$), such an admissible collection $\l=(\l_w)_{w\in W}$ exists.

Let $g$ be the genus of $X$, and let $\cB$ be the flag variety of $G$. Set
\begin{equation*}
    D=(g-1)\dim G+ \dim \cB.
\end{equation*}

\begin{theorem}\label{th:intro c} Suppose $p$ is good for $G$, and neither $\xch(ZG)$ nor $\pi_{1}(G)$ has $p$-torsion. Suppose $q=\#k$ is large enough with respect to $W$ so that a $W$-coregular and $W$-regular admissible collection $\l=\{\l_{w}\}_{w\in W}$ exists. Let $c_w\in \frc^*(k)$ be the image of $\l_w$ under the map $\frt_w^*\to \frc^*$. Then for each $\g\in \pi_{0}(\Bun_{G})(k)$ we have
\begin{equation*}
\#\AI_{G,X}(k)= q^{-D}\frac{1}{\#W}\sum_{w\in W} \sgn(w)\int_{\cM_{G,x}(c_w)(k)}1.
\end{equation*}
The integration notation on the right side means the number of  $k$ points in the stack $\cM_{G,x}(c_w)$, weighted by the reciprocals of automorphism groups (see \eqref{stack pt count}). In particular, each $\cM_{G,x}(c_w)$ has finitely many $k$-points.
\end{theorem}

There is a nicer formula for $\#\AI_{G,X}(k)$ in terms of the cohomology of the moduli stack of parabolic $G$-Higgs bundles. 

Let $\cM_{G}(\bI_{x};0)$ be the moduli stack of triples  $(\cE,\cE^{B}_{x}, \ph)$ where $(\cE,\ph)$ is a $G$-Higgs bundle on $X$ with a simple pole at $x$, $\cE^{B}_{x}$ is a Borel reduction of $\cE$ at $x$ such that the residue $\res_{x}\ph$ strictly preserves $\cE^{B}_{x}$ (i.e. it is orthogonal to $\Ad(\cE^{B}_{x})$; see \S\ref{ss:stable Higgs}). The stack $\cM_{G}(\bI_{x};0)$ is the cotangent bundle of the moduli stack $\Bun_{G}(\bI_{x})$ of pairs $(\cE,\cE^{B}_{x})$ as above. For a generic choice of stability condition $\th$ (an element in $\xcoch(T)_{\RR}$, where $T$ is the Cartan quotient of any Borel subgroup of $G$), one can define the open substack $\cM_{G}(\bI_{x};0)^{\tst}$ of $\th$-stable parabolic Higgs bundles.

In the appendix we prove that when $p$ is sufficiently large (with respect to $G$ only), $\cM_{G}(\bI_{x};0)^{\tst}$ is a smooth Deligne-Mumford stack of finite type over $k$.  Moreover, for $p$ larger than an unspecified number depending on both $G$ and the genus $g$ of $X$, the parabolic Hitchin map
\begin{equation}\label{intro h}
    h:  \cM_{G}(\bI_{x};0)^{\tst}\to \cA_G(\bI_x)
\end{equation}
is proper. It turns out that these properties imply that there is a canonical $W$-action on the $\ell$-adic cohomology or compact support cohomology of $\cM_{G}(\bI_{x};0)^{\tst}$ (base changed to $\ov k$). Let $\cohoc{*}{\cM_{G}(\bI_{x};0)^{\tst}}\j{\sgn}$ be the sign-isotypic summand of $\cohoc{*}{\cM^{\tst}_{G}(\bI_{x};0), \Qlbar}$ under the $W$-action.

\begin{theorem}\label{th:intro} Suppose $p=\ch(k)$ is sufficiently large (depending only on the genus $g$ and the Dynkin diagram of $G$), and $q=\#k$ is sufficiently large with respect to the Weyl group of $G$, then we have an equality
\begin{equation*}
\#\AI_{G,X}(k)=q^{-D}\Tr(\Frob, \cohoc{*}{\cM_{G}(\bI_{x};0)^{\tst}}\j{\sgn}).
\end{equation*}
Here, as usual, $\Tr(\Frob, \upH^{*}_{c})$ means alternating trace.
\end{theorem}

Let  $\cN_{G}(\bI_{x})^{\tst}$ be the zero fiber of $h$ in \eqref{intro h}; this is the parabolic version of the global nilpotent cone. Assuming the properness of $h$, it can be shown that the inclusion $\cN_{G}(\bI_{x})^{\tst}\incl \cM_{G}(\bI_{x};0)^{\tst}$ induces an isomorphism on cohomology. In particular, there is a canonical action of $W$ on the $\ell$-adic (co)homology of $\cN_{G}(\bI_{x})^{\tst}$.

\begin{cor}\label{c:intro} Under the same assumption as in Theorem \ref{th:intro}, we have an equality
\begin{equation*}
\#\AI_{G,X}(k)=q^{D}\Tr(\Frob, \homog{*}{\cN_{G}(\bI_{x})^{\tst}}\j{\sgn}).
\end{equation*}
\end{cor}

\begin{remark} In the main body of the paper, we state and prove a version of the above theorems for $G$ connected reductive. See Corollaries \ref{c:main} and \ref{c:main N}. In particular, it recovers the case $G=\GL_{n}$ for large characteristics. We also generalize these results to the counting of indecomposable bundles with parahoric level structures in \S\ref{s:par}.
\end{remark}

\begin{remark} 
The lower bound on $p$ for which Theorem \ref{th:intro} applies is not explicit at the moment. The issue is that we do not know at the moment how large $p$ needs to be in order for the parabolic Hitchin map to satisfy the existence part of the valuative criterion. We make a general argument for the properness in Appendix \S\ref{a:proper} by reducing to the known case of characteristic zero, which unfortunately leaves the bound for $p$ unspecified. 
In \cite[\S 6.3.]{AHLH}, the authors give an explicit bound on $p$ for the usual (non-parabolic) Hitchin map to satisfy the existence part of the valuative criterion. 
We hope that in the parabolic case, the bound on $p$ can similarly be made explicit.

The lower bound of $q$ for which Theorems \ref{th:intro c} and \ref{th:intro} apply can be made explicit; see Remark \ref{r:q bound}.
\end{remark}

\subsection{Other results}

We prove several results that are of independent interest. In this subsection, our convention is that $K$ is an algebraically closed field, and $G$ is a connected reductive group over $K$. 

The first side result concerns the structure of the automorphism group of a $G$-bundle over a general base $X$. The reducedness of such automorphism groups and the constraint on its maximal tori are crucial for the counting of absolutely indecomposable $G$-bundles.

\begin{theorem}\label{th:intro aut} Let $X$ be a stack over an algebraically closed field $K$ and $G$ be a connected reductive group over $K$. Let $\cE$ be a $G$-bundle over $X$ whose automorphism group $\Aut(\cE)$ is affine of finite type over $K$.

Suppose either $\ch(K)=0$ or $p=\ch(K)$ is good for $G$, and neither $\xch(ZG)$ nor $\pi_{1}(G)$ has $p$-torsion (when $p$ satisfies these conditions, $p$ is said to be \goodp for $G$). Then
\begin{enumerate}
\item The automorphism group $\Aut(\cE)$ is reduced.
\item If $X$ is connected, then any maximal torus $A$ of $\Aut(\cE)$ is canonically isomorphic to the connected center of a Levi subgroup of $G$, up to $G$-conjugacy.
\end{enumerate}
\end{theorem}

This general result has an immediate Lie theoretic consequence:
\begin{cor} Suppose either $\ch(K)=0$ or $p=\ch(K)$ is \goodp for $G$. Let $H$ be a $K$-group acting on $G$ by group automorphisms. Then the fixed point subgroup scheme $G^{H}$ is reduced (hence smooth over $K$). Moreover, any maximal torus of $G^H$ is the connected center of a Levi subgroup of $G$.
\end{cor}
When $H$ acts on $G$ by inner automorphisms (i.e., $G^H$ is the centralizer of a $K$-subgroup of $G$), the reducedness of $G^H$ is proved by Herpel \cite{Herpel} using different methods.

Other results concern counting of indecomposable bundles in low genus. When $X\cong \PP^{1}$, both sides in Theorem \ref{th:intro} are zero: the left side is zero because all $G$-bundles on $\PP^{1}$ has reduction to $T$ by Grothendieck's theorem; the right side is zero because $\cM_{G}(\bI_{x};0)^{\tst}$ is empty. The case when $X$ has genus one is particularly interesting. In this case, assuming $G$ semisimple, there is only a finite number of indecomposable $G$-bundles over $X$ over the algebraically closed field $K$. The set  $\AI_{G,X}(K)$ admits a Lie-theoretic description as follows.

\begin{theorem}\label{th:intro ell} Let $G$ be a semisimple group over an algebraically closed field $K$, and let $X$ be an elliptic curve over $K$. When $\ch(K)=p>0$, we assume $X$ is ordinary and $p=\ch(K)$ is \goodp for $G$. 

Then upon making certain choices (such as a homology basis for $X$), we have a bijection
\begin{eqnarray}
    \AI_{G,X}(K)\bij N_2(G).
\end{eqnarray}
Here $N_2(G)$ is the set of $G$-conjugacy classes of commuting triples $(s,t,u)$ in $G$, such that $s,t$ are semisimple, $u$ is unipotent, and their simultaneous centralizer $G_{s,t,u}$ contains no nontrivial torus.
\end{theorem}
Combined with Corollary \ref{c:intro}, we conclude the following non-obvious geometric fact: the dimension of the sign-isotypic part of the top homology $\homog{top}{\cN_{G}(\bI_{x})^{\tst}}\j{\sgn}$ is equal to the cardinality of $N_2(G)$, defined purely Lie-theoretically.

Computing the cardinality of $N_2(G)$ is the third in a sequence of Lie-theoretic counting problems that we discuss in \S\ref{s:Lie}, preceded by the counting of distinguished nilpotent classes and strongly isolated conjugacy classes in $G$. This triple of counting problems follow the familiar rational--trigonometric--elliptic trilogy. In \S\ref{s:Lie}, which is independent of the rest of the paper and only serves as motivation for readers more on the Lie groups side than the geometric side, we give explicit results on these counting problems, at least for $G$ simply-connected.

\subsection{Sketch of proof of the main result}

\sss{General strategy}\label{sss:gen st} The general idea is to do counting instead in the cotangent bundle (or variant thereof) of $\Bun_{G}$. In other words, we shall count Higgs bundles on $X$ with specific local behavior, and then argue that:
\begin{itemize}
\item Only absolutely indecomposable $G$-bundles on $X$ admit Higgs fields of the specific local behavior,
\item each absolutely indecomposable $G$-bundle carries the same amount of Higgs fields with the specific local behavior.
\end{itemize} 
As mentioned in the beginning, the strategy is similar to the counting strategy for indecomposable quiver representations of Crawley-Boevey and van den Bergh in \cite{CBvdB}, who use the deformed preprojective algebra of a quiver in order to count its indecomposable representations. 

This general strategy works as stated for counting  absolutely indecomposable $G$-bundles with a Borel reduction (i.e. Iwahori level structure) at a point $x\in X(k)$. We denote the moduli of $G$-bundles with a Borel reduction at $x$ by $\Bun_{G}(\bI_{x})$. In this case, we shall consider a twisted version of the cotangent bundle of $\Bun_{G}(\bI_{x})$, namely Higgs bundles $(\cE, \cE^{B}_{x}, \ph)$ with Borel reduction $\cE^{B}_{x}$ at $x$, such that the Higgs field $\ph$ has a simple pole at $x$ whose residue preserves the Borel reduction $\cE^{B}_{x}$, and its image in $\frt^{*}$ (think how the Higgs field acts on the associated graded of the flag at $x$ when $G=\GL_{n}$) is equal to a prescribed $\l\in \frt^{*}$. Denote such moduli stack by $\cM_{G}(\bI_{x};\l)$. One can show that for generic $\l$, the image of $\cM_{G}(\bI_{x};\l)(k)\to \Bun_{G}(\bI_{x})(k)$ consists of absolutely indecomposable pairs $(\cE,\cE^{B}_{x})$, and the non-empty fibers all have cardinality $q^{N}$ for a constant $N$, up to dividing by the size of $\pi_{0}(\Aut(\cE,\cE^{B}_{x}))(k)$. 

\sss{Selection functions} To carry out the general strategy for the counting of indecomposable $G$-bundles without Iwahori level structure, we need a new ingredient: the selection function. Choose a point $x\in X(k)$. We consider the moduli stack $\cM_{G,x}$ of Higgs bundles $(\cE, \ph)$ on $X$ with a simple pole at $x$. Taking residue of $\ph$ at $x$ gives a map $\res_{x}: \cM_{G,x}\to \frg^{*}/G$. We would like to choose a $G$-invariant function $f: \frg^{*}\to \CC$ and consider the weighted counting (see \S\ref{sss:gpoid integral} for notations)
\begin{equation}\label{intro int MGx}
\int_{\cM_{G,x}(k)}\res^{*}_{x}f.
\end{equation}
We would like this weighted counting to satisfy desiderata similar to those in \S\ref{sss:gen st}, namely:
\begin{itemize}
\item When breaking up the sum \eqref{intro int MGx} according to the underlying $G$-bundle, only absolutely indecomposable $G$-bundles on $X$ have nonzero contribution;
\item Each absolutely indecomposable $G$-bundle contributes the same nonzero amount to the sum \eqref{intro int MGx}.
\end{itemize}
We single out a class of functions on $\frg^{*}$ satisfying these conditions. They are Fourier transforms of {\em selection functions} on $\frg$. A selection function $\xi$ on $\frg$ is one that is pulled back from the GIT quotient $(\frg\sslash G)(k)$ and which satisfies $\int_{\frh}\xi=0$ for a class of subalgebras $\frh\subset \frg$. These subalgebras are Lie algebras of {\em relevant} subgroups $H\subset G$ that are ``shadows'' of decomposability of $G$-bundles. More precisely, for any absolutely decomposable $G$-bundle $\cE$, restricting automorphisms of $\cE$ to $x$ gives a homomorphism $\ev_{x}: \Aut(\cE)\to G$ (well-defined up to conjugacy). Then the image of $\ev_{x}$ is a relevant subgroup.

The question then becomes how to construct a selection function $\xi$ on $\frg$, and eventually we would like to construct one such that the resulting weighted counting \eqref{intro int MGx} has a clean geometric meaning. All this is done in \S\ref{s:sel}. 

Here we give an example of a selection function for $G=\GL_{2}$. 
Choose a nontrivial character $\psi_{0}: k\to \CC^{\times}$. Let $k'/k$ be the quadratic extension of $k$. Choose a nontrivial additive character $\psi_{1}: k'/k\to \CC^{\times}$. Let $A\in \gl_{2}(k)$ with eigenvalues $\l_{1}, \l_{2}\in k'$. Define 
\begin{equation*}
\xi(A)=\begin{cases}\frac{1}{2}(\psi_{0}(\l_{1}-\l_{2})+\psi_{0}(\l_{2}-\l_{1})), & \mbox{if }\l_{1},\l_{2}\in k;\\
\frac{1}{2}(\psi_{1}(\l_{1})+\psi_{1}(\l_{2})), & \mbox{if }\l_{1}, \l_{2}\notin k.
\end{cases}
\end{equation*}
This function $\xi: \gl_{2}\to \CC$ is an example of a selection function for $\frg=\gl_{2}$.
The main property of this function is that, for any subalgebra $\frh\subset \frg$ that contains a Cartan subalgebra (so $\frh$ is either a Cartan subalgebra, split or not, or a Borel subalgebra, or $\frg$), we have
\begin{equation*}
\int_{\frh}\xi=0.
\end{equation*}

In general, a selection function on $\frg$ is pulled back from a function on $(\frg\sslash G)(k)$, which can be constructed by summing over pushforwards along $\frt\to \frg\sslash G$ of characters on various Cartan subalgebras $\frt\subset \frg$. In the $\GL_{2}$-example the characters are $\psi_{0}$ on the split Cartan and $\psi_{1}$ on the nonsplit one.  With an extra condition on these characters, one can express  the selection function $\xi$, as well as its Fourier transform, in terms of Frobenius trace functions of certain perverse sheaves coming from the Grothendieck-Springer resolution (\S\ref{ss:Spr}-\S\ref{ss:adm}).   

\sss{} Finally we relate the weighted counting \eqref{intro int MGx}, where $\xi$ is the selection function mentioned above, to counting stable parabolic Higgs bundles with a generic stability parameter. The geometric idea here, involving the deformation from generic residue to zero residue,  is not new,  and we use results from \cite{DGT}.

\subsection{Convention}

\sss{} Fix a prime number $\ell$. All constructible sheaves in this paper will have $\Qlbar$-coefficients.

\sss{Groupoids} 
For a groupoid $V$, let $|V|$ be the set of isomorphism classes in $V$. 
Let $\Qlbar[V]$ be the vector space of $\Qlbar$-valued functions with finite support on $|V|$.

When $V$ is a set and $V'\subset V$ is a finite subset, let $\one_{V'}\in \Qlbar[V]$ denote the characteristic function on $V'$.

\sss{Weighted summation over a groupoid}\label{sss:gpoid integral} 
Here we only consider groupoids whose automorphism groups are finite. 

For a map $\ph: U\to V$ with finite fibers, we have pullback of functions $\ph^{*}: \Qlbar[V]\to \Qlbar[U]$ and pushforward $\ph_{!}: \Qlbar[U]\to \Qlbar[V]$ by summing along the fibers. More precisely, for $f\in \Qlbar[U]$ and $v\in V$,
\begin{equation*}
(\ph_{!}f)(v)=\sum_{(u,\a: \ph(u)\isom v)/\cong }\frac{1}{\#\Aut(u,\a)}f(u).
\end{equation*}
The sum is over the isomorphism classes of the fiber groupoid of $\ph$ over $v$, i.e., the groupoid of pairs $(u,\a)$ where $u\in U$ and $\a$ is an isomorphism between $\ph(u)$ and $v$. 

When $\ph: U\to V=\{\star\}$ is the unique map to the singleton set, we denote $\ph_{!}f$ by 
\begin{equation*}
\int_{U} f=\sum_{u\in |U|}\frac{1}{\#\Aut(u)}f(u).
\end{equation*}
In particular, for a stack $\cM$ of finite type over $k$, applying the above notation to the groupoid $\cM(k)$ and the constant function $1$, we have
\begin{equation}\label{stack pt count}
\int_{\cM(k)}1=\sum_{m\in \cM(k)}\frac{1}{\#\Aut(m)(k)}.
\end{equation}

For $f,g\in \Qlbar[V]$, let 
\begin{equation*}
\j{f,g}_{V}=\int_{V}fg.
\end{equation*}

\sss{Invariants of $G$}\label{sss:intro G} Let $G$ be a connected reductive group over a field $k$. 

We denote by $\pi_{1}(G)$ the {\em algebraic fundamental group} of $G$. Namely, for any choice of maximal torus $T_{0}\subset G_{\ov k}$, $\pi_{1}(G)$ is canonically isomorphic to the quotient $\xcoch(T_{0})/(\mbox{coroot lattice})$.

Let $ZG$ be the center of $G$, and $(ZG)^{\c}$ be its neutral component, and let $C_{G}$ be the maximal torus in $ZG$. It is always the case that $C_{G}$ is the reduced structure of $(ZG)^{\c}$. The character group $\xch(ZG)$ of the diagonalizable group $(ZG)_{\ov k}$ is well-defined. For a maximal torus $T_{0}\subset G_{\ov k}$, $\xch(ZG)$ is canonically isomorphic to the quotient $\xch(T_{0})/(\mbox{root lattice})$. If $\ch(k)=p$ and $\xch(ZG)$ has $p$-torsion, then $(ZG)^{\c}$ is not reduced.

Let $f_{G}$ be the determinant of the Cartan matrix of $G_{\ov k}$. When $G_{\ov k}$ is almost simple, we have
\begin{equation*}
f_{G}=\begin{cases} 1 & \mbox{$G$ is of type $E_{8}, F_{4}$ or $G_{2}$};\\
2 & \mbox{$G$ is of type $B,C$ or $E_{7}$};\\
3 & \mbox{$G$ is of type $E_{6}$};\\
4 & \mbox{$G$ is of type $D$};\\
n & \mbox{$G$ is of type $A_{n-1}$}. \end{cases}
\end{equation*}
In general, $f_{G}$ is the product of $f_{G_{i}}$ for simple isogeny factors $G_{i}$ of $G_{\ov k}$.

When $G$ is semisimple, we have
\begin{equation*}
f_{G}=\#\pi_{1}(G)\cdot \#\xch(ZG).
\end{equation*}
In general, the prime factors of $\#\pi_{1}(G)_{\tors}$ and $\#\xch(ZG)_{\tors}$ are necessarily prime factors of $f_{G}$.

Suppose $L\subset G$ is a twisted Levi subgroup (i.e. the centralizer of a torus). We have natural surjections $\pi_{1}(L)\surj \pi_{1}(G)$ and $\xch(ZH)\surj \xch(ZG)$. When restricted to the torsion parts, we get injections
\begin{eqnarray}
\label{pi1 tors}\pi_{1}(L)_{\tors}\incl\pi_{1}(G)_{\tors}, \\
\label{xz tors} \xch(ZH)_{\tors}\incl \xch(ZG)_{\tors}.
\end{eqnarray}
To see these injections,  we may assume $k=\ov{k}$. Choose a maximal torus $T\subset L$,  a basis $\{\a_{i}\}_{i\in I}$ of $\Phi(G,T)$ and a subset $J\subset I$ such that $\Phi(L,T)$ has basis $\{\a_{j}\}_{j\in J}$. Then $\xch(ZL)=\xch(T)/\Span_{\ZZ}\{\a_{j};j\in J\}$ and $\xch(ZG)=\xch(T)/\Span_{\ZZ}\{\a_{j};j\in I\}$, so that the kernel of $\xch(ZL)\to \xch(ZG)$ is torsion-free. This implies \eqref{xz tors}. The argument for \eqref{pi1 tors} is similar.

\sss{Good, pretty good and very good primes}\label{sss:goodp}
Let $G$ be a connected reductive group over a field $k$ of characteristic $p$. Then $p$ is called
\begin{enumerate}
\item {\em good} for $G$ if for the abstract based root system of $G_{\ov k}$, $p$ does not divide any coefficient of the highest root $\theta$ written as a $\ZZ$-combination of simple roots; see \cite[Ch.I, \S4]{SS},
\item {\em \goodp} for $G$ if $p$ is good for $G$, and neither $\xch(ZG)$ nor $\pi_{1}(G)$ has $p$-torsion; see \cite{Herpel},
\item {\em very good} for $G$ if $p$ is good for $G$ and $p\nmid f_{G}$.  
\end{enumerate}
When $G$ does not have simple factors of type $A$, all three notions are equivalent. In general, as the terminology suggests, very good implies pretty good, and pretty good implies good. 

When the root system of $G$ is irreducible, the following holds.
\begin{itemize}
\item If $G$ has type $A$, all primes are good;
\item If $G$ has type $B, C$ or $D$ (but not of type $A$), all primes except $p=2$ are good;
\item If $G$ has type $E_{6}, E_{7}, F_{4}$ or $G_{2}$, all primes except $p=2$ and $3$ are good;
\item If $G$ has type $E_{8}$, all primes except $p=2, 3$ and $5$ are good.
\end{itemize} 

For $G$ semisimple, pretty good is equivalent to very good. In general, it may be weaker than very good. For example, when $G=\GL_{n}$, all primes are pretty good, while very good primes are those not dividing $n$.


\subsection*{Acknowledgement} The authors would like to thank Victor Ginzburg and Jochen Heinloth for helpful comments.

\section{Some Lie-theoretic counting problems}\label{s:Lie}
This section serves as motivation for the main counting problem concerned in this paper. In the spirit of the trilogy: additive group, multiplicative group and elliptic curves, we formulate three Lie-theoretic counting problems for a semisimple group $G$ over $\CC$. We also give answers to these counting problems.

There is a formal link between the last of these Lie-theoretic counting problems and the counting of indecomposable $G$-bundles on elliptic curves, which we will elaborate on in \S\ref{s:ell}. 

\subsection{Three counting problems}\label{ss:n2}
Let $G$ be a semisimple algebraic group over an algebraically closed field $k$ with Lie algebra $\frg$. We assume either $\ch(k)=0$ or $\ch(k)=p$ is good for $G$.

\sss{$N_0(G)$} Following Bala and Carter, we recall that a nilpotent element $X\in\frg$ (resp. a unipotent element $u\in G$) is {\em distinguished} if its centralizer $G_{X}$ (resp. $G_{u}$) does not contain a nontrivial torus (equivalently, $X$ or $u$ does not lie in a proper Levi subgroup of $G$). 

Let $N_{0}(\frg)$ be the set of distinguished nilpotent orbits in $\frg$, whose cardinality we denote by $n_{0}(\frg)$. Similarly, let $N_0(G)$ be the set of distinguished unipotent conjugacy classes in $G$, whose cardinality we denote by $n_0(G)$.

By our assumption on $\ch(k)$, there exists a $G$-equivariant isomorphism between the unipotent variety $\cU_G$ of $G$ and the nilpotent variety $\cN_G$ of $\frg$ (see \S\ref{ss:unip}, especially Lemma \ref{l:unip log}), which gives a bijection between $N_0(\frg)$ and $N_0(G)$.

\sss{$N_{1}(G)$}Following Lusztig,  a semisimple element $g\in G$ (or its conjugacy class) is called {\em isolated} if the reduced structure of the connected centralizer $G^{\c}_{g}$ is semisimple. Let $\Sig(G)$ be the set of isolated semisimple conjugacy classes.

We call an element $g\in G$ (or its conjugacy class) {\em strongly isolated} if its centralizer $G_{g}$ contains no nontrivial torus. If $g=su$ is the Jordan decomposition of $g$, then $g$ is strongly isolated if and only if $s$ is isolated in the sense of Lusztig, and $u$ is a distinguished unipotent element in $G_{s}$, necessarily lying in the neutral component $G^\c_s$, a semisimple group after taking reduced structure.

Let $N_{1}(G)$ be the set of strongly isolated conjugacy classes in $G$, and let $n_{1}(G)$ be its cardinality.  By the above discussion, we have a bijection
\begin{equation*}
N_{1}(G)\bij \coprod_{[s]\in \Sig(G)}N_{0}(G^\c_s).
\end{equation*}
Therefore
\begin{equation*}
n_{1}(G)=\sum_{[s]\in \Sig(G)}n_{0}(G^\c_s).
\end{equation*}

When $G$ is simply-connected, we have a bijection between $\Sig(G)$ and vertices of the affine Dynkin diagram of $G$ (which is the disjoint union of affine Dynkin diagrams of simple factors of $G$). When $G$ is semisimple but not necessarily simply-connected, we have $\Sig(G)=\Sig(G^{\sc})/\pi_{1}(G)$ where $G^{\sc}$ is the simply-connected cover of $G$. Note that $\pi_{1}(G)$ acts on $\Sig(G^{\sc})$ in the same way as it acts on the affine Dynkin diagram of $G$.

\sss{$N_{2}(G)$}\label{sss:n2} Taking this one step further, let $N_{2}(G)$ be the set of conjugacy classes of commuting triples $(s,t,u)$ in $G$ such that $s,t$ are semisimple, $u$ is unipotent, and the simultaneous centralizer $G_{s,t,u}$ does not contain a nontrivial torus.    Let $n_{2}(G)$ be the cardinality of $N_{2}(G)$.

In \S\ref{s:ell} we will show that, when $\ch(k)$ is \goodp for $G$, there is a bijection between $N_{2}(G)$ and the set of isomorphism classes of indecomposable $G$-bundles on an ordinary elliptic curve $X$ over $k$. The bijection will depend on auxiliary choices such as a homology basis for $X$.

Let $\pi: G^\sc\to G$ be the simply-connected cover, with kernel $\ker(\pi)\cong \pi_1(G)(1)$ (Tate twist of the algebraic fundamental group). From a commuting pair $s,t\in G$, we may attach $\g_{s,t}\in \pi_{1}(G)(1)$ by $\g_{s,t}=[\wt s,\wt t]\in  \ker(G^\sc\to G)$, where $\wt s,\wt t$ are liftings of $s$ and $t$ to $G^\sc$ . Now fix $\g\in \pi_{1}(G)$ and consider the subset $N_2(G;\g)\subset N_2(G)$ consisting of conjugacy classes of commuting triples $(s,t,u)$ such that $\g_{s,t}=\g$. Then we have a decomposition
\begin{equation*}
N_{2}(G)=\coprod_{\g\in \pi_{1}(G)(1)}N_{2}(G;\g).
\end{equation*}
Under the interpretation of $N_2(G)$ as indecomposable $G$-bundles on an ordinary elliptic curve, the above decomposition of $N_2(G)$ corresponds to the decomposition according to the connected components of $\Bun_{G}$.

When $\g=1$, a triple $[(s,t,u)]\in N_{2}(G)$ lies in $N_{2}(G;1)$ if and only if $s$ and $t$ both lie in the same torus. We have a bijection
\begin{equation*}
N_{2}(G;1)=\coprod_{[s]\in \Sig(G)}N_{1}(G_{s}^{\c}).
\end{equation*}
Let $n_{2}(G;1)$ be the cardinality of $N_{2}(G;1)$, then we have 
\begin{equation*}
n_{2}(G;1)=\sum_{[s]\in \Sig(G)}n_{1}(G_{s}^{\c}).
\end{equation*}
\begin{remark}
    When $G$ is simply-connected, centralizers of semisimple elements are connected. The above formulas therefore simplify to 
    \begin{align*}
        N_{2}(G) &= \coprod_{[s] \in \Sigma(G)} N_{1}(G_{s}), \textup{ and}\\
        n_{2}(G) &= \sum_{[s] \in \Sigma(G)} n_{1}(G_{s}). 
    \end{align*}
    This allows us to carry out calculations in the simply-connected case. 
\end{remark}

For general $\g\in \pi_{1}(G)$, it is likely that there is a connected semisimple group $H_{\g}$ such that $N_{2}(G;\g)$ is in bijection with $N_{2}(H_{\g};1)$. We have the following expectation about $H_{\g}$. Choose a  commuting pair of semisimple elements $s,t\in G$ such that $[\wt s,\wt t]=\g$. Let $S$ be a maximal torus in $G_{s,t}$. It is known \cite[Prop. 4.2.1]{BFM} that the $G$-conjugacy class of $S$ depends only on $\g$ and not on $(s,t)$. Let $L=C_{G}(S)$ be a Levi subgroup of $G$, well-defined up to conjugacy. Let $T_{\g}=L/L^{\der}$. Then $H_{\g}$ should have maximal torus $T_{\g}$ and Weyl group $W_{\g}=N_{G}(L)/L$. The group $H_{\g}$ should be closely related to $G_{s_0,t_0}^\c$ for a pair of semisimple elements  $s_0,t_0\in L^\der$ such that $[\wt s_0,\wt t_0]=\g$.

\subsection{Calculations}
In the following, we work over the base field $k=\CC$. We compute the numbers $n_0(G),n_1(G)$ and $ n_2(G)$ for many $G$, including all almost simple simply-connected groups.

For the calculations in classical types, recall the following facts (see for example \cite[Theorem 8.2.14.]{Collingwood-McGovern}).
\begin{fact} Let $\frg$ be a simple Lie algebra of classical type. For an integer $n$, denote by $\Part(n)$ the set of  partitions of $n$. Then, nilpotent orbits in $\frg$ are in bijection with the following sets by taking the sizes of Jordan blocks in the standard representation:
\begin{enumerate} 
\item Type $A_{n-1}$: $\Part(n)$, 
\item Type $B_{n}$: $\{\l \in \Part(2n+1)  \mid \textup{even parts have even multiplicity}\}$,
\item Type $C_{n}$: $\{\l \in \Part(2n)  \mid \textup{odd parts have even multiplicity}\}$,
\item Type $D_{n}$: $\{\l \in \Part(2n)  \mid \textup{even parts have even multiplicity}\}$, with the exception that each very even partition (having only even parts) corresponds to two nilpotent orbits.
\end{enumerate}
In type $A$, the only distinguished nilpotent orbit is the regular nilpotent orbit. In types $B,C,D$ the distinguished orbits are precisely those which have no repeated parts. 
\end{fact}

\sss{Type $A$}
When $G$ is a semisimple group isogenous to a product of type $A$ groups, then $N_0(G)$ is a singleton consisting of the regular unipotent class. There is a bijection
\begin{equation*}
N_{1}(G)\bij ZG.
\end{equation*}
sending $[su]$ to $s$ which is necessarily central.


We next describe $N_2(G;\g)$ for $G$ almost simple of type $A$. There is an action of $(ZG)^2$ on $N_2(G;\g)$ by
\begin{equation*}
    (z_1,z_2)\cdot (s,t,u)=(z_1s, z_2t, u).
\end{equation*}

\begin{lemma}
    Suppose $G=\SL_n/\mu_m$ where $m|n$. Then for any $\g\in \mu_m$, $N_2(G;\g)$ is a $(ZG)^2$-torsor. In particular, $n_2(G;\g)=(n/m)^2$ for all $\g\in \mu_m$, and $n_2(G)=n^2/m$.
\end{lemma}
\begin{proof}
    Let $d$ be the order of $\g$ (so that $d|m$). 
    Let $\G$ be the 2-step nilpotent group generated by a central subgroup $\mu_d$ and two other elements $x$ and $y$ subject to the relation $[x,y]=\g\in \mu_d$. Note that its center $Z\G$ is generated by $\mu_d, x^d$ and $y^d$, and $\G/Z\G\cong(\ZZ/d\ZZ)^2$.

    Let $\wt N_\G$ be the set of isomorphism classes of $n$-dimensional representations $\r$ of $\G$ satisfying the following conditions:
    \begin{enumerate}
        \item $\det(\r)$ is the trivial character on $\G$.
        \item $\r|_{\mu_d}$ is the identity map onto the subgroup of scalar matrices $\mu_d\subset \SL_n$.
        \item As a representation of $\G$, $\r$ is isotypical (i.e., direct sum of isomorphic irreducible representations).
    \end{enumerate}

    By the Stone-von Neumann theorem, for each character $\chi: Z\G\to \CC^\times$ that restricts to the tautological embedding $\mu_d\incl \CC^\times$ on $\mu_d\subset \G_0$, there is a unique irreducible representation $\r_\chi$ with central character $\chi$. Moreover, $\dim \r_\chi=d$. 
    Therefore each $\r\in \wt N_\G$ is isomorphic to $\r_\chi^{\op n/d}$ for a unique $\chi\in \Hom(Z\G,\mu_n)$. We thus have an injective map
    \begin{equation*}
        c: \wt N_\G\incl \Hom(Z\G, \mu_n)'
    \end{equation*}
    sending $\r$ to its central character, where the right side is the subset of characters $\chi: Z\G\to \mu_n$ extending the tautological one on $\mu_d$.

    If $\r,\r'\in \wt N_\G$, then $\r\cong \r'\ot \om$ for some character $\om: \G\to \CC^\times$. Since both $\r$ and $\r'$ have trivial determinant, the image of $\om$ must lie in $\mu_n$, i.e. it is a map $\om: \G\to \mu_n$. The central characters $\chi$ of $\r$ and $\chi'$ of $\r'$ are related by $\chi'=\chi\cdot \om^d$. Now $\om^d|_{Z\G}$ can be any character $Z\G/\mu_d\to \mu_n/d$. We conclude that $c(\wt N_\G)$ is a torsor under $\Hom(Z\G/\mu_d, \mu_{n/d})\cong (\mu_{n/d})^2$. 

    From $\r\in \wt N_\G$, we get a pair of semisimple elements $\xi=\r(x)$ and $\y=\r(y)$ in $\SL_n$ with $[\xi,\y]=\g$. Let $s,t$ be the images of $\xi,\y$ in $G$. Then $G_{s,t}^\c$ is the image of $C_{\SL_n}(\xi,\y)^\c=C_{\SL_n}(\r)^\c\cong \SL_{n/d}$. For a regular unipotent element $u$ in $G_{s,t}^\c$, the group $G_{s,t,u}^\c$ is unipotent. This way we get a map
    \begin{equation*}
        \t: \wt N_\G\to N_2(G;\g).
    \end{equation*}
    This map is surjective: for any $(s,t,u)\in N_2(G;\g)$, lifting $s,t$ to $\xi,\y\in \SL_n$ (so that $[\xi,\y]=\g$), we get a homomorphism $\r:\G\to \SL_n$ by setting $\r(x)=\xi$ and $\r(y)=\y$. Since $[\xi,\y]=\g$, $\r|_{\mu_d}$ is the identity map onto $\mu_d\subset \SL_n$. Since $G_{s,t}^\c$ has to be semisimple, $C_{\SL_n}(\xi,\y)^\c=C_{\SL_n}(\r)^\c$ is semisimple, which forces $\r$ to be isotypical. This shows $\r\in \wt N_\G$. Given $(s,t)$, $u$ has only one choice up to conjugation in $G_{s,t}$, therefore $\t$ is surjective.

    Consider the action of $\Hom(\G,\mu_m)$ on $\wt N_\G$ by tensoring. Since $\mu_m=\ker(\SL_n\to G)$, the action preserves each fiber of the map $\t$ and acts transitively on each fiber. Moreover, this action is intertwined under $c$ with the action of $\Hom(Z\G/\mu_d, \mu_{m/d})=(\mu_{m/d})^2$ on the image $c(\wt N_\G)$ via the $d$-th power map $\Hom(\G,\mu_m)\surj \Hom(Z\G/\mu_d, \mu_{m/d})$. Since $c$ is injective, the action of $\Hom(\G,\mu_m)$ on $\wt N_\G$ factors through  $\Hom(Z\G/\mu_d, \mu_{m/d})$. We conclude that $\t$ is a $\Hom(Z\G/\mu_d, \mu_{m/d})\cong (\mu_{m/d})^2$-torsor. 

Since  $\wt N_\G$ is a $\Hom(Z\G/\mu_d, \mu_{n/d})\cong (\mu_{n/d})^2$-torsor, we conclude that $N_2(G;\g)$ is a $(\mu_{n/d})^2/(\mu_{m/d})^2\cong (ZG)^2$-torsor. 
\end{proof}

Before calculating the numbers for other classical types, we need some preparatory results. 

\sss{Conjugacy classes under central isogeny} Let $\pi : G_{1} \to G_{2}$ be a central isogeny of reductive groups. For an element $\ov{s}\in G_{2}$ the centralizer $C_{G_{1}}(\ov{s})$ is well-defined: it is the centralizer of  any lift of $\ov{s}$ in $G_1$. Since any two lifts differ by a central element in $G_{1}$, this does not depend on the choice. We have a natural map $\delta : \pi_{0}(C_{G_{2}}(\ov{s})) \to \ker(\pi)$ constructed as follows. 

First, we write down a map $C_{G_{2}}(\ov{s}) \to \ker(\pi)$. Since $\ker(\pi)$ is finite, this map will factor through the component group. Let $g\in C_{G_{2}}(\ov{s})$, and choose lifts $h\in G_{1}$ of $g$ and $s\in G_{1}$ of $\ov{s}$. By definition, we have $\pi(hsh^{-1}s^{-1} ) =1$, so $hsh^{-1} = z_{g}s$ for an element $z_{g}\in \ker(\pi)$. It is easy to check that this is independent of the lifts, so we can define $\delta(g) = z_{g}$. Moreover, by construction, the sequence
\begin{equation}\label{eqn:exactness-delta}
\pi_{0}(C_{G_{1}}(\ov{s})) \to \pi_{0}(C_{G_{2}}(\ov{s})) \xrightarrow{\delta} \ker(\pi)
\end{equation}
is exact in the middle (i.e. the image of the first map is equal to $\ker(\delta)$). Indeed, if $h \in C_{G_{1}}(\ov{s})$, then $hsh^{-1} = s$, so $z_{\pi(h)} = 1$.

\begin{lemma}\label{lem:conj-classes-isogeny} In the above situation, let $\ov{s}\in G_{2}$ be any element. Then, the set of $G_{1}$-conjugacy classes lying over the conjugacy class of $\ov{s}$ is a torsor for the group $\ker(\pi) / \Im(\delta)$.
\end{lemma}
\begin{proof}
Let $\cC_{\ov s}$ be the $G_2$-conjugacy class of $\ov s$. Then $\th:\pi^{-1}(\cC_{\ov s})\to \cC_{\ov s}$ is a $\ker(\pi)$-torsor. By the centrality of $\ker(\pi)$, 
$\pi^{-1}(\cC_{\ov s})$ is a union of $G_1$-conjugacy classes.

We have to check that the action of $\ker(\pi)$ on the set  $\pi^{-1}(\cC_{\ov s})/G_1$ is transitive, and that the stabilizer is precisely $\Im(\delta)$.

For transitivity, observe that for two classes $\cC_{1}, \cC_{2}$ lying over the conjugacy class $\cC_{\ov{s}}$ in $G_2$ you can find representatives $s_1,s_2$ of $\cC_{1},\cC_{2}$ respectively such that $\pi(s_i) = \ov{s}$. Thus $s_2 = zs_1$ for some $z\in \ker(\pi)$, proving transitivity. 

To compute the stabilizer, let $s$ be a lift of $\ov{s}$. Then $z\in \ker(\pi)$ stabilizes the $G_1$-conjugacy class $\cC_{s}$ if $zs = hsh^{-1}$ for some $h\in G_{1}$. This means that for $g=\pi(h)$ we have $\delta(g) = z$, so $z\in \Im(\delta)$. It is clear that the converse also holds. 
\end{proof}
A particular instance is when $G_2$ is semisimple and $G_1$ is its simply connected cover. In that case, we obtain the following result. 
\begin{cor}\label{cor:sc-conj-classes} Let $G$ be a complex semisimple group with simply-connected cover $\pi : G^{\sc} \to G$. Let $s \in G^{\sc}$ with image $\ov{s}\in G$. Then, the number of conjugacy classes in $G^{\sc}$ mapping to the conjugacy class of $\ov{s}$ is equal to ${|\pi_{1}(G)|}/{|\pi_{0}(C_{G}(\ov{s}))|}$.
\end{cor}

Below we will give the generating functions for the numbers to be calculated in the $B,D$ and $C$ series. We consider only the simply-connected cases here (in which case $n_{2}(G;1))= n_2(G)$).

\sss{Types $B$ and $D$}

Let $G=\Spin_{n}$. To find $n_0(\Spin_{n})$, we have to count partitions of $n$ with distinct odd parts. This number is given by the generating function
\begin{equation*}
f_{BD}(t) = \prod_{j \ge 0} (1+t^{2j+1}), 
\end{equation*}
so we get 
\begin{equation*}
\sum_{n\ge0}n_{0}(\Spin_{n})t^{n}=f_{BD}(t).
\end{equation*}
Moreover, we write 
\begin{align*}
f_B(t) &= \frac{1}{2}(f_{BD}(t)-f_{BD}(-t)), \textup{and} \\
f_D(t) &= \frac{1}{2}(f_{BD}(t)+f_{BD}(-t)). 
\end{align*}
for the odd and even parts.

In the following we understand $\Spin_{0} = \{1\}$, and $\Spin_{1} = \mu_{2}$. In order to find $n_1(\Spin_{n})$ and $n_2(\Spin_{n})$, we need to list semisimple conjugacy classes with semisimple centralizers in $\Spin_{n}$. We can do that by using Lemma \ref{lem:conj-classes-isogeny}, describing the conjugacy classes in $G^{\sc} = \Spin_{n}$ in terms of those in $G = \SO_{n}$. 

The elements in $\SO_{n}$ with semisimple centralizers are represented by $\ov{s} = (1^{a}, (-1)^{b})$ where the powers mean repeated entries, and $b$ is even. The centralizer $C_{G}(\ov{s}) \cong S(O_a\times O_b)$ has two connected components, unless $a=0$ or $b=0$, in which case it is connected. This implies that either $a=0$ or $b=0$ and there are two classes above $\ov{s}$, or $a\neq 0$ and $b \neq 0$ (and both are also not equal to $2$), and there is a unique class above $\ov{s}$. 

Therefore, when $n$ is odd, these centralizers are either of the form $\Spin_{n}$, counted twice, corresponding to the two non-conjugate central elements, or of the form
\begin{equation*}
G_{a,b} =(\Spin_{a} \times \Spin_{b} )/\mu_{2}
\end{equation*}
with $a+b=n$, $a$ odd, $b$ even, and $b\neq 2$, counted once, corresponding to the unique conjugacy class lying over $(1^{a}, (-1)^{b})$. The quotient is by the diagonally embedded central $\mu_2$. 

When $n$ is even, semisimple centralizers are of the form $\Spin_{n}$, counted four times, corresponding to the four central elements, or of the form 
\begin{equation*}
G_{a,b} =(\Spin_{a} \times \Spin_{b} )/\mu_{2}
\end{equation*}
with $a+b=n$, $a, b\neq 2$, and $a,b$ even, corresponding to the unique class lying over $(1^{a}, (-1)^{b})$.

Note that the degree $2$ coefficient of $f_{BD}$ is the number of partitions of $2$ with distinct odd parts, hence it vanishes. Moreover, $n_{0}((\Spin_{a} \times \Spin_{b} )/\mu_{2}) = n_{0}(\Spin_{a} \times \Spin_{b} )$. Therefore we get 

\begin{equation*}
n_{1}(\Spin_{n}) = \sum_{\substack{a+b=n \\ b \textup{ even}}} n_0(\Spin_{a})n_0(\Spin_{b})
+\begin{cases}
n_0(\Spin_{n}), n \textup{ odd} \\
2n_0(\Spin_{n}), n \textup{ even}
\end{cases}
\end{equation*}

For generating functions, this implies
\begin{equation*}
\sum_{n\ge 1}n_{1}(\Spin_{n})t^{n}=f_{BD}(t)f_{D}(t)+f_{B}(t)+2f_{D}(t) - 3 = (1+f_{BD})(1+f_{D})-4
\end{equation*}

We now compute the generating function for $n_{2}(\Spin_{n})$. For that, we need to describe the isolated semisimple conjugacy classes in $G_1 = G_{a,b} =(\Spin_{a} \times \Spin_{b} )/\mu_{2}$. When $a=0$ or $b=0$ this is described above. So we assume $a>0$ and $b>0$. We apply Lemma \ref{lem:conj-classes-isogeny} to the natural projection
\begin{equation*} 
\pi : (\Spin_{a} \times \Spin_{b} )/\mu_{2} \to \SO_{a} \times \SO_{b}
\end{equation*}
which is a central isogeny with kernel $\mu_{2}$. The isolated conjugacy classes in $G_2 = \SO_{a} \times \SO_{b}$ are represented by elements of the form
\begin{equation*}
\ov{s} = (1^{a_{+}},(-1)^{a_{-}}, 1^{b_{+}},(-1)^{b_{-}}) 
\end{equation*}
with $a_{-}$ and $b_{-}$ even. The centralizer is $C_{G_{2}}(\ov{s}) \cong S(O_{a_{+}} \times O_{a_{-}}) \times S(O_{b_{+}} \times O_{b_{-}})$. Consider the map $\delta : C_{G_2}(\ov{s}) \to \ker(\pi)$, and let $\delta_{a}$ (resp. $\delta_{b}$) be the restriction of $\delta$ to $S(O_{a_{+}} \times O_{a_{-}})$ (resp. to $S(O_{b_{+}} \times O_{b_{-}})$). 

We have a Cartesian diagram 
\begin{equation*}
\xymatrix{
\Spin_{a} \ar[d] \ar[r] & G_{a,b} \ar[d]\\
\SO_{a}\ar[r] & \SO_{a}\times \SO_{b} 
}
\end{equation*}
and likewise for $b$. Therefore, the map $\delta_{a}$ is the same as the map $\delta$ for the projection $\Spin_{a}\to \SO_{a}$ and the element $\ov{s}_{a} = (1^{a_{+}},(-1)^{a_{-}})$. Note that by exactness of the sequence (\ref{eqn:exactness-delta}), $\delta_{a}$ is injective (since $\pi_0(C_{\Spin_{a}}(\ov{s}_{a})) = 1$ ). Assuming $a_{\pm} > 0$, the group $S(O_{a_{+}} \times O_{a_{-}})$ has two connected components. This implies $\delta_{a}$ is an isomorphism on component groups, and moreover that $\delta : C_{G_2}(\ov{s}) \to \ker(\pi)$ is surjective. The same argument applies when $b_{\pm}> 0$ to show surjectivity of $\delta$. Therefore, in this situation, there is a unique conjugacy class in $G_{a,b}$ above $\ov{s}$. 

In the case that $a_+ a_-=0$ and $b_+b_- =0$ (i.e. $\ov{s}$ is central in $\SO_{a}\times \SO_{b}$), the map $\delta$ is trivial, so there are two conjugacy classes above $\ov{s}$.

Summarizing, we can compute $n_{1}(G_{a,b})$ as follows. When $a=0$ or $b=0$, we have $G_{a,b}=\Spin_{n}$, so we assume $a,b>0$. For an integer $k$ abbreviate $G_k =\Spin_k$. We have
\begin{align*}
n_{1}(G_{a,b}) &= \sum_{\substack{a_{+}a_{-}\neq 0 \textup{ or} \\ b_{+}b_{-} \neq 0 }} n_0(G_{a_{\pm},b_{\pm}}) + 2 \sum_{\substack{a_{+}a_{-}= 0 \textup{ and} \\ b_{+}b_{-} = 0 }} n_0(G_{a_{\pm},b_{\pm}}) \\
&=  \sum_{\substack{a_{+}+a_{-} = a,\\ b_{+}+b_{-} =b }} n_0(G_{a_{+}})n_0(G_{a_{-}})n_0(G_{b_{+}})n_0(G_{b_{-}}) + \begin{cases} 
2n_0(G_{a})n_{0}(G_{b}), a \textup{ odd} \\
4n_0(G_{a})n_{0}(G_{b}), a \textup{ even}
\end{cases}
\end{align*}
where $G_{a_{\pm},b_{\pm}} \subset G_{a,b}$ is the centralizer of any element in $G_{a,b}$ lying above $\ov{s} = (1^{a_{+}},(-1)^{a_{-}},1^{b_{+}},(-1)^{b_{-}}) $. 
Moreover, by the same reasoning as for the computation of $n_1(\Spin_n)$ we get
\begin{equation*}
n_{2}(\Spin_{n}) = \sum_{\substack{a+b=n \\ b \textup{ even}}} n_1(G_{a,b})
+\begin{cases}
n_1(\Spin_{n}), n \textup{ odd} \\
2n_1(\Spin_{n}), n \textup{ even}.
\end{cases}
\end{equation*}
Let $g_{BD}(t) = (1+f_{BD}(t))(1+f_{D}(t))-3$ be the generating function for $n_1(\Spin_n)$, and let $g_{D}(t) = \frac{1}{2}(g_{BD}(t)+g_{BD}(-t))$ be its even part. It is easy to check that $g_{D}(t) = f_{D}(t)^2+2f_{D}(t) -2$, and that for the generating function we now get
\begin{align*}
\sum_{n\ge 1} n_{2}(\Spin_n) t^n &=  f_{BD}(t)f_{D}(t)^{3}+2f_{BD}(t)f_{D}(t)+2f_{D}(t)^2+g_{BD}(t)+g_{D}(t) -  7\\
& = f_{BD}(t)f_{D}(t)^{3}+3f_{D}(t)^{2}+3(f_{BD}(t)+1)f_{D}(t)+f_{BD}(t) - 11
\end{align*}

\sss{Type $C$}
Let $G=\Sp_{2n}$. The set of partitions of $2n$ with distinct even parts is in bijection with partitions of $n$ with distinct parts. Therefore, the number of distinguished orbits is equal to $p'(n)$,  the number of partitions of $n$ into distinct parts. Its generating function is
\begin{equation*}
\sum_{n\ge0}p'(n)t^{n}=\prod_{n\ge1}(1+t^{n}).
\end{equation*}
For the generating function $f_C$ of distinguished nilpotent orbits for $\Sp_{2n}$ we therefore have
\begin{eqnarray*}
\sum_{n\ge0}n_{0}(\Sp_{2n})t^{n}=&f_C=\prod_{n\ge1}(1+t^{n}),\\
\sum_{n\ge0}n_{1}(\Sp_{2n})t^{n}=&f_{C}^2=\prod_{n\ge1}(1+t^{n})^{2},\\
\sum_{n\ge0}n_{2}(\Sp_{2n})t^{n}=&f_{C}^4=\prod_{n\ge1}(1+t^{n})^{4}.\\
\end{eqnarray*}
Indeed, the centralizers that can occur are the subgroups $\Sp_{2i} \times \Sp_{2(n-i)}\subset \Sp_{2n}$. This means that
\begin{eqnarray*}
n_{1}(\Sp_{2n})=\sum_{i=0}^{n}n_{0}(\Sp_{2i})n_{0}(\Sp_{2(n-i)}),\\
n_{2}(\Sp_{2n})=\sum_{i=0}^{n}n_{1}(\Sp_{2i})n_{1}(\Sp_{2(n-i)}),
\end{eqnarray*}
from which the above formulas follow.

\sss{Exceptional groups}

Partial results are listed in the table below. We omit details of the calculation.

\begin{center}

\begin{tabular}{|c|c|c|c|c|}

\hline

$G$ & $n_{0}(\frg)$ & $n_{1}(G)$ & $n_{2}(G;1)$ & $n_{2}(G)$ \\

\hline

$G_{2}$ & 2 & 4 & 9 & 9\\
\hline

$F_{4}$ & 4 & 10 & 29 & 29 \\\hline

$E^{\sc}_{6}$ & 3 & 13 & 66 & 66 \\\hline

$E^{\ad}_{6}$ & 3 & 5 & 10 & \\\hline

$E^{\sc}_{7}$ & 6 & 22 & 102 & 102 \\
\hline
$E^{\ad}_{7}$ & 6 & 12 &  &\\
\hline
$E_{8}$ & 11 & 31 & 113 & 113 \\
\hline
\end{tabular}

\end{center}

\section{Selection functions}\label{s:sel}

This section is entirely Lie-theoretic and does not involve any curve. We will construct certain functions on the Lie algebra $\frg$ that will later help us single out indecomposable bundles on curves. 

\subsection{General remarks on reductive groups $G$}

Let $G$ be a (not necessarily split) connected reductive group over $k=\FF_{q}$ with Lie algebra $\frg$. Recall that every such $G$ is quasi-split. Let $\cB$ be the flag variety of $G$. Let
\begin{equation*}
\chi: \frg\to \frc=\frg\sslash G
\end{equation*}
be the Chevalley morphism.

\sss{Universal Cartan}\label{sss:univ Cartan}
For any two Borel subgroups $B_{1}$ and $B_{2}$ of $G$ with reductive quotients $T_{1}$ and $T_{2}$, choose any $g\in G(k)$ such that $\Ad(g)B_{1}=B_{2}$, then the isomorphism $T_{1}\isom T_{2}$ induced by $\Ad(g)$ is independent of the choice of $g$. Therefore the reductive quotients of all Borel subgroups of $G$ are canonically identified with a single torus, the {\em universal Cartan} of $G$, denoted by $T$. Let $\frt=\Lie T$.

Let $W$ be the {\em abstract Weyl group} of $G$, i.e., its underlying set is the set of $G_{\ov k}$-orbits on $\cB^{2}_{\ov k}$. The simple reflections in $W$ correspond to minimal $G_{\ov k}$-orbits other than the diagonal orbit. The $\FF_{q}$-form $G$ determines an automorphism $\d$ of $W$ as a Coxeter group, hence defining a finite \'etale group scheme $\un W$ over $k$. The group scheme $\un W$ acts on $T$.

If we choose a Borel subgroup $B$ defined over $k$ and a maximal torus $T_{0}\subset B$, the composition $\io_{T_{0},B}: T_{0}\subset B\to T$ gives an isomorphism between $T_{0}$ and $T$. Let $W_{0}=N_{G}(T_{0})/T_{0}$ be the Weyl group scheme defined by $T_{0}$ (a finite \'etale group scheme  over $k$). Then there is a canonical isomorphism of group schemes $\io_{W_{0}, B}: W_{0}\cong \un W$ sending simple reflections of $W_{0}(\ov k)$ defined by $B$ to simple reflections of $W=\un W(\ov k)$. The isomorphism $\io_{T_{0}, B}$ is equivariant under the isomorphism $\io_{W_{0},B}$. 

Applying the above discussions after base change to $\ov k$, we see that for any maximal torus $T_{0}\subset G$, upon choosing a Borel subgroup $B\subset G_{\ov k}$ containing $T_{0,\ov k}$, we get an isomorphism $\io_{T_{0,\ov k}, B}: T_{0,\ov k}\isom T_{\ov k}$. Different choices of $B$ only changes $\io_{T_{0,\ov k}, B}$ by post-composing with the $W$-action.  The root system $\Phi(G_{\ov k}, T_{0,\ov k})\subset \xch(T_{0,\ov k})$ together with its simple roots determined by $B$ induces a based root system $\Phi$ on $\xch(T_{\ov k})$ by $\io_{T_{0,\ov k}, B}$. This based root system  $\Phi$ is independent of the choices of $T_{0}$ and $B$, and is called the {\em abstract based root system} of $G$. The simple reflection in $W$ are reflections defined by the simple roots in $\Phi$.

\sss{Maximal tori}\label{sss:tori} Conjugacy classes of maximal tori in $G$ are in bijection with Frobenius-twisted conjugacy classes in $W$. When $p$ is not too small, the same statement holds for Cartan subalgebras in $\frg$. In any case, for each $w\in W$ one can define a Cartan subalgebra of $\frg$ as follows. First, for each $w \in W$, define the subspace 
\begin{equation}\label{tw t} 
\frt_{w} = \{ x \in \frt_{\ov{k}} \mid x = \Ad_{w}\Fr(x) \}.
\end{equation}
in the universal Cartan $\frt$. Now we realize $\frt_{w}$ as a Cartan subalgebra in $\frg$.

Let $T'$ be a maximally split torus in $G$ with Lie algebra $\frt'$, contained in a Borel subgroup $B$ defined over $k$. Over $\ov{k}$, the Weyl group of $T'_{\ov{k}}$ is identified with $W$ via $B_{\ov{k}}$. 
Let $w \in W$, and lift it to $\dot{w}\in N_{G_{\ov{k}}}(T'_{\ov{k}})$. By the Lang-Steinberg theorem, the morphism $g \mapsto g^{-1}\Fr(g)$ is surjective, and we may choose $g$ such that $\dot{w} = g^{-1}\Fr(g)$. 

By construction, $\frt'' := \Ad_{g}(\frt')$ is $\Fr$-stable, hence defined over $k$. Since
\begin{equation*}
    \frt''(k) = (\frt''(\ov{k}))^{\Fr} = (\Ad_{g}\frt'(\ov{k}))^{\Fr}, 
\end{equation*}
the map $\Ad_{g^{-1}}$ induces an isomorphism of $\frt''(k)$ with the $w\circ \Fr$-fixed points 
\begin{equation*}
(\frt')^{w\circ\Fr} = \{x \in \frt'(\ov{k}) \mid x = \Ad_{w}\Fr(x) \}.
\end{equation*}
Finally, the subspace $\frt_{w} \subset \frt_{\ov{k}}$ in the abstract Cartan is identified with $\frt''$. This construction determines an embedding $\frt_{w}\incl \frg$, which is well-defined up to $G(k)$-conjugacy. For each $w$ we moreover obtain a map 
\begin{equation*}
\chi_{w}: \frt_{w}\to \frc
\end{equation*}
over $k$, which is descended from $\chi_{\frt}: \frt_{\ov k}\to \frc_{\ov k}$.

A similar argument can be made for maximal tori in $G$, and this construction induces the bijection between Frobenius-twisted conjugacy classes in $W$ and $G$-conjugacy classes of maximal tori.

\sss{Further notations}
Since $G$ is defined over $k$, its base change $G_{\ov{k}}$ is equipped with a $\ov{k}$-linear Frobenius automorphism $\Fr_{\ov{k}} : G_{\ov{k}}\to G_{\ov{k}}$. This automorphism induces a Frobenius structure on the flag variety $\cB_{\ov{k}}$, on the universal Cartan $\frt_{\ov{k}}$, and on the Chevalley quotient $\chi : \frg_{\ov{k}} \to \frc_{\ov{k}}$. Moreover, the Grothendieck alteration $\pi : \wt{\frg}_{\ov{k}} \to \frg_{\ov{k}}$ carries an induced Frobenius structure, and it makes sense to consider the commutative diagram 
\begin{equation*}
\xymatrix{ 
    \wt{\frg} \ar[r]^{\pi} \ar[d]^{\varepsilon} & \frg \ar[d]^{\chi} \\
    \frt \ar[r]^{\chi_{\frt}} & \frc. 
    }
\end{equation*}
over $k$. The cohomology groups are usually taken for varieties over $\ov k$, even though we may drop the subscript ${\ov{k}}$ from the notation: e.g., $\cohog{*}{\cB}$ refers to the $\Qlbar$-cohomology of $\cB_{\ov k}$. 

\begin{lemma}\label{l:van sum} Suppose $G$ is a connected reductive group over $k$ (not necessarily split). Let $N=\dim \cB$.
\begin{enumerate} 
\item As functions on $\frc(k)$, we have
\begin{equation}\label{fibers of chi}
\chi_{!}\one_{\frg}=q^{N}\frac{1}{\#W}\sum_{w\in W}\sgn(w)\Tr(\Fr^{*}\c w, \cohog{*}{\cB})\chi_{w!}\one_{\frt_{w}}.
\end{equation}
Here $\sgn: W\to \{\pm1\}$ is the sign character.
The $W$-action on $\cohog{*}{\cB}$ is normalized so that its action on $\cohog{0}{\cB}$ is trivial, and its action on the top cohomology is by $\sgn$.
\item Let $f$ be a function on $\frc(k)$. Suppose $\j{\one_{\frt}, \chi^{*}f}_{\frg}=0$ for all Cartan subalgebras $\frt\subset \frg$, then $\int_{\frg}\chi^{*}f=0$.
\end{enumerate} 
\end{lemma}
\begin{proof}  
(1): Abbreviate $p=\chi_{\frt}$, and recall the commutative diagram 
\begin{equation*}
\xymatrix{ 
    \wt{\frg} \ar[r]^{\pi} \ar[d]^{\varepsilon} \ar[dr]^{\wt{\chi}}& \frg \ar[d]^{\chi} \\
    \frt \ar[r]^{p} & \frc. 
    }
\end{equation*}
Take a point $a\in \frc(k)$. The equality \eqref{fibers of chi} at $a$ will come from two ways of calculating the Frobenius trace on the $W$-invariants on $\cohoc{*}{\wt\chi^{-1}(a),\Qlbar}$, one using $\wt\chi=\chi\c\pi$, the other one using $\wt\chi=p\c \vep$. We have a canonical isomorphism of Weil sheaves on $\frc$
\begin{equation}\label{two push}
\chi_{!}\pi_{!}\Qlbar\cong p_{!}\vep_{!}\Qlbar.
\end{equation}
We will equip the pullback of both sides to $\frc_{\ov k}$ with $W$-actions. Here $W$ is the abstract Weyl group $W$ as a Coxeter group, identified as a set with the set of $G_{\ov k}$-orbits on $\cB_{\ov k}\times \cB_{\ov k}$. 

For the left side of \eqref{two push}, the $W$-action comes from the $W$-action on the Grothendieck-Springer sheaf $\pi_{!}\Qlbar$ on $\frg_{\ov k}$. On the other hand, $\pi_{!}\Qlbar$  carries a Weil structure 
$$\phi: \Fr^{*}(\pi_{!}\Qlbar)\cong \pi_{!}\Qlbar,$$ 
which is twisted-equivariant under $W$ in the following sense: recall that the $\FF_{q}$-form $G$ determines an automorphism $\d$ of $W$ as a Coxeter group. Then $\phi$ intertwines the action of $w$ on the left side (before applying $\Fr^{*}$) and the action of $\d(w)$  on the right. Taking $W$-invariants, $(\pi_{!}\Qlbar)^{W}$ still inherits a Weil structure from $\phi$, and as such, we have an isomorphism of Weil sheaves on $\frg$
\begin{equation*}
(\pi_{!}\Qlbar)^{W}\cong \const{\frg}.
\end{equation*}
Taking stalk at $x\in \frg(k)$, we get
\begin{equation*} \one_{\frg}(x) = \frac{1}{\#W}\sum_{w\in W} \Tr(\Fr^{*}\c w, \cohoc{*}{\cB_{x},\Qlbar}).
\end{equation*} 
Note that $\Fr^{*}\c w=\d(w)\c \Fr^{*}$ as endomorphisms of $\cohoc{*}{\cB_{x},\Qlbar}$. The above equality implies for any $a\in \frc(k)$,
\begin{equation}\label{Fr w tr 0}
    (\chi_{!}\one_{\frg})(a) =\frac{1}{\#W} \sum_{w\in W} \Tr(\Fr^{*}\c w, \cohoc{*}{\wt{\chi}^{-1}(a), \Qlbar}).
\end{equation}

Now we define the $W$-action on the right side of \eqref{two push}. The complex $\vep_{!}\Qlbar$ is easy to describe. Namely, the diagram
\begin{equation*} 
\xymatrix{ 
\wt{\frg} 	\ar[r]\ar[dr]^{\vep}	& \cB \times \frt \ar[d]^{\pr_{\frt}} \\
										& \frt 
}
\end{equation*}
commutes. The map $\wt{\frg} = G\times^{B} \frb \to G/B \times \frt = \cB \times \frt$ is an affine space bundle with fiber isomorphic to $[\frb, \frb]$ (here $B$ is a Borel subgroup with Lie algebra $\frb)$. Let $N= \dim[\frb,\frb]$, then we have an isomorphism of Weil sheaves
\begin{equation*} 
\varepsilon_{!}\Qlbar = \cohog{*}{\cB}[-2N](-N) \otimes \Qlb{\frt}.
\end{equation*}
The $W$-action on $p_{!}\vep_{!}\Qlbar$ is induced from the equivariant structure on $\vep_{!}\Qlbar$  defined as follows. 
The constant sheaf $\const{\frt}$ carries a canonical $W$-equivariant structure. We equip the $\cohog{*}{\cB}$ with the $W$-action by identifying it with the stalk of $\pi_{!}\Qlbar$ at $0$ (so $W$-acts trivially on $\cohog{0}{\cB}$ by sign on $\cohog{2N}{\cB}$). We equip the $\Qlbar[-2N](-N)$-factor with the sign representation of $W$. 

We claim that the isomorphism \eqref{two push} is equivariant under the $W$-actions defined on both sides. Factoring $\vep$ as the composition 
\begin{equation*}
\vep: \wt\frg\xr{\wt \pi} \frg\times_{\frc}\frt\xr{\pr_{2}}\frt,
\end{equation*}
we reduce to show that the $W$-equivariant structure on $\vep_{!}\Qlbar$ deefined above coincides with the one coming by $\pr_{2,!}$ from the canonical $W$-equivariant structure on $\wt\pi_{!}\Qlbar=\IC_{\frg\times_{\frc}\frt}[-\dim \frg]$ (which induces the Grothendieck-Springer action by pushing forward to $\frg$).  Since $\vep_{!}\Qlbar$ is a graded constant sheaf, any $W$-equivariant structure on it is determined by the action of $W$ on the $0$-stalk $(\vep_{!}\Qlbar)_{0} = \cohoc{*}{\wt{\cN}} \cong \cohog{*}{\cB}[-2N](-N)$. Thus we reduce to checking the following two $W$-actions on $\cohoc{*}{\wt{\cN}} \cong \cohog{*}{\cB}[-2N](-N)$ are the same: one comes by taking $\cohoc{*}{-}$ of the restriction of $\pi_{!}\Qlbar$ to $\cN$, together with the restricted Grothendieck-Springer $W$-action; the other is the Grothendieck-Springer action of $W$ on $\cohog{*}{\cB}$ twisted by the sign character. The first one is Verdier dual to $\cohog{*}{\cN,\pi_{*}\Qlbar}= (\pi_{*}\Qlbar)_{0}\cong \cohog{*}{\cB}$ (again with the $W$-action that acts trivially on $\cohog{0}{\cB}$). Dualizing back, we see it is the same $W$-action as the second action.

Using the $W$-equivariance of \eqref{two push} at $a\in\frc(k)$, we see that 
\begin{eqnarray}
\notag \Tr(\Fr^{*}\c w, \cohoc{*}{\wt{\chi}^{-1}(a), \Qlbar})&=&\Tr(\Fr^{*}\c w, \cohoc{*}{p^{-1}(a), \vep_{!}\Qlbar})\\
\notag &=&\Tr(\Fr^{*}\c w, \cohoc{*}{p^{-1}(a)}\ot \cohog{*}{\cB}[-2N](-N))\\
\label{Fr w tr}&=&q^{N}\sgn(w)\Tr(\Fr^{*}\c w, \cohoc{*}{p^{-1}(a)}\ot \cohog{*}{\cB}).
\end{eqnarray}
Since $p$ is finite, $\cohoc{*}{p^{-1}(a)}$ has a basis given by $\wt a\in \frt(\ov k)$ that map to $a$. Now $w\c \Fr$ permutes the these points $\wt a$, and the fixed points are precisely $\wt a\in \frt_{w}(k)$ that map to $a$. Only fixed points contribute to the trace of $(w\c \Fr)^{*}=\Fr^{*} \c w$. Therefore \eqref{Fr w tr} can be written as
\begin{equation*}
q^{N}\sgn(w) \sum_{\substack{\wt{a} \in \frt_{w}(k) \\ \chi_{w}(\wt{a}) = a }}\Tr(\Fr^{*}\c w, \cohog{*}{\cB}) =q^{N}\sgn(w)\Tr(\Fr^{*}\c w, \cohog{*}{\cB})(\chi_{w,!}\one_{\frt_{w}})(a). 
\end{equation*}
Combining this with \eqref{Fr w tr} and \eqref{Fr w tr 0} proves part (1). 

(2): Taking the pairing $\j{-,f}_{\frc(k)}$ on both sides of \eqref{fibers of chi} we conclude that $\int_{\frg}\chi^{*}f=\j{\chi_{!}\one_{\frg}, f}_{\frc(k)}$ is a linear combination of $\j{\chi_{!}\one_{\frt_{[w]}},f}_{\frc(k)}=\j{\one_{\frt_{[w]}, \chi^{*}f}}_{\frg}$. The statement follows.
\end{proof}

\subsection{Relevant subgroups}
We introduce a class of subgroup of $G$ called $G$-relevant. They are closely related to automorphism groups of $G$-bundles, about which we will elaborate in \S\ref{s:aut}.

Whenever talking about $G$-relevant subgroups/subalgebras, we make the following assumption on $p=\ch(k)$:
\begin{equation}\label{p prime to ZG}
\mbox{$\xch(ZG)$ does not have $p$-torsion; equivalently, $ZG$ is reduced.}
\end{equation}

\begin{defn}\label{def:rel} 
\begin{enumerate}

\item Let $A\subset G$ be a $k$-torus. We say $A$ is {\em $G$-relevant} if there is a Levi subgroup $L\subset G_{\ov k}$ such that $A_{\ov k}=C_{L}$ (maximal torus in the center of $L$). In this case, $L$ is determined by $A$ and descends to $k$, namely $L=C_{G}(A)_{\ov k}$.

\item Let $H\subset G$ be a not-necessarily-reductive $k$-subgroup. We say $H$ is {\em $G$-relevant} if $H$ is reduced, and some (equivalently any) maximal torus $A\subset H$ is $G$-relevant. 

\item A $k$-subalgebra $\fra\subset \frg$ is called {\em $G$-relevant toral subalgebra} if it is the Lie algebra of a $G$-relevant torus $A\subset G$.

\item A $k$-subalgebra $\frh\subset\frg$ is called {\em $G$-relevant} if it is the Lie algebra of a $G$-relevant subgroup  $H\subset G$. 

\end{enumerate}
\end{defn}

Note that all relevant subgroups of $G$ contain $C_{G}$, which is equal to $(ZG)^{\c}$ by our assumption \eqref{p prime to ZG}, and its Lie algebra contains $\frz=\Lie(ZG)$.

\begin{defn}\label{d:ess unip} A relevant subgroup $H\subset G$ is called {\em essentially unipotent} if $H^{\c}/C_{G}$ is unipotent. Equivalently, $C_{G}$ is the maximal torus in $H$.

\end{defn}

\begin{lemma}\label{l:G rel cut by roots}
Suppose $\xch(ZG)$ has no $p$-torsion. Let $\frt_{0}=\Lie T_{0}\subset \frg$ be a Cartan subalgebra. Then a $k$-subspace $\fra\subset \frt_{0}$ is $G$-relevant toral if and only if there exists a set of roots $\{\a_{j}\}_{j\in J}\subset \Phi(G_{\ov k}, T_{0,\ov k})$ that can be extended to a basis of the root system $\Phi(G_{\ov k}, T_{0,\ov k})$, such that $\fra_{\ov k}=\cap_{j\in J}\ker(d\a_{j})$.
\end{lemma}
\begin{proof} We may base change the situation to $\ov k$, so we assume $k=\ov k$. 

First assume $\fra\subset \frt_{0}$ is $G$-relevant toral. Then $\fra=\Lie A$ for a $G$-relevant torus $A\subset T_{0}$. The $G$-relevance of $A$ implies $A=C_{L}$ for the Levi subgroup $L=C_{G}(A)$. Then $T_{0}$ is a maximal torus in $L$. Now $\Phi(G,T_{0})$ has a basis $\{\a_{i}\}_{i\in I}$ and a subset $J\subset I$ such that $\{\a_{j}\}_{j\in J}$ is a basis for $\Phi(L,T_{0})$. We have an exact sequence
\begin{equation}\label{AT0}
0\to \xcoch(A)\to \xcoch(T_{0})\xr{f_{J}}\ZZ^{J}.
\end{equation}
Here $f_{J}(\l)=\{\j{\a_{j},\l}\}_{j\in J}\in \ZZ^{J}$ for $\l\in \xcoch(T_{0})$. Clearly $\coker(f_{J})$ is finite. The adjoint map of $f_{J}$ is the inclusion $\Span_{\ZZ}\{\a_{j}; j\in J\}\incl \xch(T_{0})$. Now $\coker(f_{J})$ has the same cardinality as $(\xch(T_{0})/\Span_{\ZZ}\{\a_{j}; j\in J\})_{\tors}=\xch(ZL)_{\tors}$. The natural map $\xch(ZL)\to \xch(ZG)=\xch(T_{0})/\Span_{\ZZ}\{\a_{i}; i\in I\}$ induces an injection on the torsion part. Since $\xch(ZG)$ has no $p$-torsion, neither does $\xch(ZL)$. Therefore $\coker(f_{J})$ has no $p$-torsion. This implies that \eqref{AT0} becomes a short exact sequence after reduction mod $p$, hence also after $\ot_{\ZZ}k$:
\begin{equation*}
0\to \xcoch(A)\ot_{\ZZ}k\to \xcoch(T_{0})\ot_{\ZZ}k\xr{f_{J}\ot k}k^{J}\to 0.
\end{equation*}
This is the same as the short exact sequence
\begin{equation*}
0\to \fra\to \frt_{0}\xr{\frf_{J}}k^{J}\to 0
\end{equation*}
where $\frf_{J}(v)=(\j{d\a_{j}, v})_{j\in J}$, $v\in \frt_{0}$. Thus we conclude that $\fra=\cap_{j\in J}\ker(d\a_{j})$.

Conversely, suppose $\{\a_{j}\}_{j\in J}$ is part of a basis for $\Phi(G,T_{0})$. Let $L$ be the Levi subgroup of $G$ containing $T_{0}$ with root system spanned by $\{\a_{j}\}_{j\in J}$. Then $\Lie ZL=\cap_{j\in J}\ker(d\a_{j})$. The argument in the previous paragraph shows that $\xch(ZL)$ has no $p$-torsion, hence $ZL$ is reduced, and $A=(ZL)^{\c}$ is a $G$-relevant torus. Therefore $\cap_{j\in J}\ker(d\a_{j})=\Lie (ZL)^{\c}=\Lie A$ is a $G$-relevant toral subalgebra of $\frg$.
\end{proof}

\begin{exam}\label{ex:GLn rel tori} Let $G=\GL_{n}$ and $T_{0}\cong (\Gm)^{n}$ be the diagonal maximal torus. The $G$-relevant tori $A\subset T_{0}$ are in bijections with set-partitions of $\{1,2,\cdots, n\}$: for a partition $\{1,2,\cdots, n\}=I_{1}\coprod I_{2}\coprod \cdots\coprod I_{s}$, the corresponding $A$ consists of $(x_{1},\cdots, x_{n})\in (\Gm)^{n}$ such that $x_{i}=x_{j}$ whenever $i$ and $j$ are in the same part of the partition. The $G$-relevant toral subspaces $\fra\subset \frt_{0}=\Lie T_{0}\cong k^{n}$ are obtained from set-partitions of $\{1,2,\cdots, n\}$ in the same way: $\fra$ consists of elements $(x_{1},\cdots, x_{n})\in k^{n}$ such that $x_{i}=x_{j}$ whenever $i$ and $j$ are in the same part.

For $G=\PGL_{n}$, the $G$-relevant tori (resp. $G$-relevant toral subalgebras) in the diagonal torus (resp. its Lie algebra) are the images of the $\GL_{n}$-relevant tori (resp. $G$-relevant toral subalgebras) described above.
\end{exam}

\begin{exam} Let $G=\SL_{n}$ and $T_{0}$ be the diagonal maximal torus. We identify $T_{0}$ with the subtorus of $(\Gm)^{n}$ with total product one. 

We describe the minimal nontrivial $G$-relevant tori contained in $T_{0}$. Let $1\le a\le [n/2]$ and $b=n-a$. Let $d=\gcd(a,b)$ and write $a=da_{1}, b=db_{1}$. Let $A_{a}\subset T_{0}$ be the subtorus consisting of elements of the form
\begin{equation*}
(\underbrace{x,\cdots, x}_{a}, \underbrace{y,\cdots, y}_{b})\quad \mbox{ satisfying }x^{a_{1}}y^{b_{1}}=1, x,y\in \Gm.
\end{equation*}
Then $A_{a}$ is $G$-relevant, and any minimal nontrivial $G$-relevant subtori in $T_{0}$ is of the form $w(A_{a})$ for a unique $1\le a\le [n/2]$ and some $w\in S_{n}$. 

\end{exam}

\subsection{$W$-relevance}
Recall that $\frt$ is the universal Cartan of $\frg$, so that $\frt_{\ov k}$ carries the action of the abstract Weyl group $W$. 
\begin{defn} 
\begin{enumerate}
\item A $\ov k$-subspace $\fra\subset \frt_{\ov k}$ is called {\em $W$-relevant} if $\fra=(\frt_{\ov k})^{W'}$ for some subgroup $W'\subset W$. 
\item Let $\frt_{0}\subset\frg$ be a Cartan subalgebra. A $k$-subspace $\fra\subset \frt_{0}$ is called {\em $W$-relevant} if, under some (equivalently any) identification $\io: \frt_{0,\ov k}\cong \frt_{\ov k}$ obtained by choosing a Borel over $\ov k$ containing $\frt_{0,\ov k}$ (see \S\ref{sss:univ Cartan}), $\io(\fra_{\ov k})\subset \frt_{\ov k}$ is $W$-relevant in the sense above.
\end{enumerate}
\end{defn}

\begin{exam} For $\GL_{n}$ and the diagonal Cartan $\frt_{0}$,  $W$-relevant subspaces in $\frt_{0}$ are the same as $G$-relevant toral subalgebras in the diagonal Cartan $\frt_{0}$, which are in bijection with set-partitions of $\{1,2,\cdots, n\}$ as described in Example \ref{ex:GLn rel tori}.
\end{exam}

\begin{exam}\label{ex:PGLn W rel}
Let $G=\PGL_{n}$ and assume $p|n$. In the diagonal Cartan $\frt_{0}$, consider the line $\fra$ spanned by $(1,2,\cdots, n)\in k^{n}/\D(k)=\frt_{0}$. Then $\fra$ is $W$-relevant: it is the fixed point locus of the cyclic permutation $(12\cdots n)\in S_{n}$. It is however not $G$-relevant.
\end{exam}

\begin{lemma}\label{l:W rel} Let $\frt_{0}$ be a Cartan subalgebra of $\frg$. 
\begin{enumerate}
\item If $\fra_{1},\fra_{2}\subset \frt_{0}$ are both $W$-relevant, so is $\fra_{1}\cap \fra_{2}$. 
\item Suppose $\xch(ZG)$ has no $p$-torsion, and $p\ne 2$ when $G^{\der}\cong \SO_{2n+1}\times G_{1}$ for some $n\ge1$. If $\fra\subset \frt_{0}$ is a $G$-relevant toral subalgebra, then it is also $W$-relevant.
\item\label{W rel cut by roots} If $p$ is \goodp for $G$ (see \S\ref{sss:goodp}), then a subspace $\fra\subset\frt_{0}$ is $W$-relevant if and only if it is $G$-relevant.
\end{enumerate}
\end{lemma}
\begin{proof}
(1) is clear from the definition.

Below we let $\frt_{0}=\Lie T_{0}$ for a maximal torus $T_{0}\subset G$, and denote the Weyl group $W(G,T_{0})$ by $W$.  

The statements (2) and (3) can both be checked after base change to $\ov k$. Therefore, below we assume $k=\ov k$.

(2) Suppose $\fra\subset \frt_{0}$ is $G$-relevant, then  $\fra=\cap_{\a\in \Phi'} \ker(\a)$ for a subset $\Phi'\subset \Phi:=\Phi(G,T_{0})$. Let $W'\subset W$ be generated by the reflections $r_{\a}$ for $\a\in \Phi'$. We always have $\ker(\a)\subset (\frt_{0})^{r_{\a}}$. Under the assumption on $p$, the differential $d\a^{\vee}\in \frt_{0}$ of every coroot $\a^{\vee}: \GG_{m}\to T_{0}$ is nonzero. Therefore, $(\frt_{0})^{r_{\a}}=\ker(\a)$, hence $\fra_{\ov k}=(\frt_{0})^{W'}$.   

(3) When $G$ has a simple factor of type $B$, $p$ good implies $p\ne 2$. Therefore $p$ satisfies the condition in (2), and $G$-relevance implies $W$-relevance.  

Conversely, let $\fra\subset \frt_{0}$ be $W$-relevant. Then $\fra=\frt_{0}^{W'}$ for some subgroup $W'\subset W$. We may assume $W'$ is the pointwise stabilizer of $\fra$. Then there exists a point $x\in \fra$ such that $W_{x}=W'$. To show that $\fra$ is $G$-relevant, by Lemma \ref{l:G rel cut by roots}, we need to show that $\frt^{W_{x}}_{0}$ is cut out by a subset of (differentials of) simple roots. By Lemma \ref{l:Wy parabolic} below, $W_{x}$ is a parabolic subgroup of $W$. In other words, there exists a basis $\{\a_{i}\}_{i\in I}$ for $\Phi(G,T_{0})$ and a subset $J\subset I$ such that $W_{x}$ is generated by $r_{\a_{j}}$ for $j\in J$. We thus have $\frt_{0}^{W_{x}}=\cap_{j\in J} \ker(d\a_{j})$ (using that $d\a^{\vee}_{j}\ne0$ with our assumption on $p$).
\end{proof}

\begin{lemma}\label{l:centralizer Levi}
Let $N$ be a positive integer all of whose prime factors are good for $G$. Assume further that $\pi_{1}(G)$ has no nonzero $N$-torsion. Then:
\begin{enumerate}
\item For any $y\in \xcoch(T)/N$, the stabilizer $W_{y}$ is a parabolic subgroup of $W$.
\item For any homomorphism $\ph: \mu_{N}\to G$, the centralizer of $\ph(\mu_{N})$ is a Levi subgroup of $G$.
\end{enumerate}
\end{lemma}
\begin{proof}
(1) Recall $\Phi$ is the abstract root system in $\xch(T)$. Let $\Phi_{y}=\{\a\in \Phi|\a(y)=0\in \ZZ/N\ZZ\}$. Let $\wt y\in \xcoch(T)$ be any lifting of $y$, and $x=\frac{1}{N}\wt{y}$. Let $\Phi_{\aff,x}=\{(\a,n)\in \Phi\times \ZZ|\a(x)+n=0\}$, the set of affine roots that vanish at $x$. Then projection to $\Phi$ gives a bijection $\Phi_{\aff,x}\isom \Phi_{y}=\Phi'$.  Let $\tilW_{x}$ be the stabilizer of $x$ under $\tilW=\xcoch(T)\rtimes W$, then projection to $W$ gives an isomorphism $\tilW_{x}\isom W_{y}$. 

We claim that $\tilW_{x}$ is contained in the affine Weyl group $\Wa=(\Span\ZZ\Phi^{\vee})\rtimes W$. In fact, let $(\l,w)\in \tilW_{x}$, where $\l\in \xcoch(T)$ and  $w\in W$. Then $wx+\l=x$. Multiply by $N$ on both sides we see $N\l=w\wt y-\wt y$, which lies in the coroot lattice since $\wt y\in \L^{\vee}$. The image of $(\l,w)$ in $\tilW/\Wa\cong \xcoch(T)/\Span\Phi^{\vee}$ is the coset of $\l$, and the sentence above shows that the coset of $\l$ is $N$-torsion in $\pi_{1}(G)=\xcoch(T)/\Span\Phi^{\vee}$. Since $\pi_{1}(G)$ has no $N$-torsion, $\l\in \Span\Phi^{\vee}$ and $(\l,w)\in \Wa$.

Now $\tilW_{x}\subset \Wa$, it is therefore generated by reflections corresponding to affine roots in $\Phi_{\aff,x}$. Hence $W_{y}$ is generated by reflections corresponding to roots in $\Phi_{y}$. Now, the same argument of Sommers in \cite[p.380, proof of Prop. 4.1 when $W^{J}=W$]{Sommers} applies to show that $\Phi_{y}$ is rationally closed in $\Phi$ when $N$ does not have any bad prime factors for $G$. Therefore $W_{y}$ is a parabolic subgroup of $W$.

(2)  Choose a maximal torus $T_{0}$ of $G$ containing $\ph(\mu_{N})$. Identify $T_{0}$ with the universal Cartan $T$. Let $\Phi'=\{\a\in \Phi|\a\c\ph=1\}$. Then $H:=C_{G}(\ph(\mu_{N}))$ is generated by $T_{0}$, the root subgroups $U_{\a}$ for $\a\in \Phi'$, and (liftings of) $W'=C_{W}(\ph(\mu_{N}))$, see \cite[Ch.II, \S4.1]{SS}. To show that $H$ is the Levi subgroup with root system $\Phi'$, it suffices to show that $W'$ is a parabolic subgroup of $W$, and is generated by the reflections $r_{\a}$ for $\a\in \Phi'$. Now $\ph:\mu_{N}\to T_{0}$ corresponds to an element $y\in \xcoch(T_{0})/N$. In the notation from the proof of (1), we have $W_{y}=W'$ and $\Phi_{y}=\Phi'$. We conclude that $W'$ is the parabolic subgroup corresponding to the rationally closed subsystem $\Phi'$.
\end{proof}

\begin{lemma}\label{l:Wy parabolic}
Let $G$ be a connected and simply-connected and reductive group over $k=\ov k$ of characteristic $p$. Assume $p$ is \goodp for $G$. Then for all $y\in \frt$ (universal Cartan), the stabilizer $W_{y}$ is a parabolic subgroup of $W$.   
\end{lemma}
\begin{proof}
We have $\frt=(\xcoch(T)/p)\ot_{\FF_{p}}k$. Choose an $\FF_{p}$-basis $\{b_{s}\}_{s\in S}$ of $k$, and identify $\frt$ with $\op_{s\in S}(\xcoch(T)/p)b_{s}$, with the action of $W$ on each $\xcoch(T)/p$. Write $y=\sum_{s}y_{s}$ under this decomposition. Then $W_{y}=\cap_{s\in S}W_{y_{s}}$.  By Lemma \ref{l:centralizer Levi}(1), each $W_{y_{s}}$ is a parabolic subgroup. Since the intersection of parabolic subgroups of $W$ is still parabolic, $W_{y}=\cap_{s\in S}W_{y_{s}}$ is also parabolic. 
\end{proof}

\subsection{Selection functions}\label{ss:sel} We assume \eqref{p prime to ZG}. Recall that $\chi: \frg\to \frc=\frg\sslash G$ is the Chevalley morphism.

\begin{defn}A function $\xi: \frc(k)\to \Qlbar$ is called a {\em $G$-selection function} if it satisfies the following conditions:
\begin{itemize}

\item For any $G$-relevant torus $A\subset G$ such that $A\ne C_{G}$, letting $\fra=\Lie A$, we have $\j{\one_{\fra}, \chi^{*}\xi}_{\frg}=0$.
\item $\xi(\chi(z))=1$ for all $z\in \frz(k)$.
\end{itemize}
\end{defn}

\begin{prop}\label{p:sel fun} Let $\xi: \frc(k)\to \Qlbar$ be a selection function. Then for any relevant subgroup $H\subset G$ that is not essentially unipotent (Definition \ref{d:ess unip}), letting $\frh=\Lie H$, we have $\j{\one_{\frh}, \chi^{*}\xi}_{\frg}=0$. 
\end{prop}
\begin{proof} 
Since the statement only depends on the neutral component $H^{\c}$, we may assume $H$ is connected. Let $\pi_{H}: H\to M=H_{\red}$ be the reductive quotient of $H$. Let $\pi_{\frh}: \frh\to \fm=\Lie M$ be the map on Lie algebras. Now $(\chi^{*}\xi)|_{\frh}$ descends to a function $\Xi_{\fm}: \fm\to \Qlbar$, which further descends to its adjoint quotient $\xi_{\fm}: \frc_{M}(k)\to \Qlbar$. To show $\j{\one_{\frh}, \chi^{*}\xi}_{\frg}=0$, it suffices to show $\int_{\fm}\Xi_{\fm}=0$. 

We would like to apply Lemma \ref{l:van sum} to the connected reductive group $M$. In order to do that, we need to check that for any maximal torus $A\subset M$, we have $\j{\one_{\fra}, \Xi_{\fm}}_{\fm}=0$. Now $A$ lifts to a maximal torus $\wt A\subset H$, hence $\wt A$ is a $G$-relevant torus in $G$. By our assumption, $H/C_{G}$ is not unipotent, hence $M/C_{G}$ is nontrivial,  and therefore $\wt A\ne C_{G}$. Let $\wt\fra=\Lie\wt A$. Using that $\xi$ is a selection function, we have $\j{\one_{\wt\fra}, \chi^{*}\xi}_{\frg}=0$, which implies $\j{\one_{\fra}, \Xi_{\fm}}_{\fm}=0$, for any maximal torus  $A$ of $M$.  Lemma \ref{l:van sum} then implies that  $\int_{\fm}\Xi_{\fm}=0$. This finishes the proof.
\end{proof}

We give a construction of $G$-selection functions by considering certain generic linear functions on various Cartan subalgebras of $\frg$.

\begin{defn}\label{d:W coreg}
\begin{enumerate}
    \item An element $\l\in\frt^*_{\ov k}$ (equivalently a $\ov k$-linear function $\frt_{\ov k}\to \ov k$) is called {\em $W$-coregular} if 
    \begin{itemize} 
        \item $\l|_{\frz}=0$.
        \item $\l|_{\fra}\ne0$ for any $W$-relevant  subspace $\fra\subset \frt_{\ov k}$ such that $\fra\ne \frz$.
    \end{itemize}
    
\item Let $\frt_{0}\subset\frg$ be a Cartan subalgebra. An element $\l\in \frt^*_{0}$ (equivalently a $k$-linear map $\l:\frt_0\to k$) is called {\em $W$-coregular} if it is so when viewed as an element in $\frt^*_{\ov k}$ under some (equivalently any) identification $\io: \frt_{0,\ov k}\cong \frt_{\ov k}$ obtained by choosing a Borel over $\ov k$ containing $\frt_{0,\ov k}$ (see \S\ref{sss:univ Cartan}).
\end{enumerate}
\end{defn}

\begin{exam} Let $G=\GL_{n}$ and $T_{0}$ be the diagonal torus. A linear function $\l: \frt_{0}=k^{n}\to k$ can be identified with an $n$-tuple $(\l_{1},\cdots, \l_{n})\in k^{n}$. The condition $\l|_{\frz}=0$ says that $\l_{1}+\cdots+\l_{n}=0$. Such a $\l$ is $W$-coregular if and only if for any proper non-empty subset $J\subset \{1,2,\cdots, n\}$, the partial sum $\sum_{j\in J}\l_{j}\ne0$. This can be seen from the description of $W$-relevant (equivalently $G$-relevant in this case) subspaces of $\frt_{0}$ given in Example \ref{ex:GLn rel tori}. 

On the other hand, when $G=\PGL_{p}$ (where $p=\ch(k)$) and $T_{0}$ the diagonal torus, a $W$-coregular function on $\frt_{0}$ can be identified with $(\l_{1},\cdots, \l_{p})\in k^{p}$ such that
\begin{itemize}
\item $\l_{1}+\cdots +\l_{p}=0$;
\item no partial sum of $\{\l_{1},\cdots, \l_{p}\}$ is zero;
\item For any $\s\in S_{p}$, $\sum_{i=1}^{p}\s(i)\l_{i}\ne0$.
\end{itemize}
This is caused by the extra $W$-relevant subspaces in $\frt_{0}$ described in Example \ref{ex:PGLn W rel} (and their translates under $W$). So for $k=\FF_{3}$, there exists a $W$-coregular function on the diagonal Cartan of $\GL_{3}$ but not for $\PGL_{3}$.
\end{exam}

Next we argue for the existence of $W$-coregular functions, when $q=\#k$ is not too small. Let $\frR_{\ov k}(W, \frt)$ be the set of relevant $\ov k$-subspaces in the universal Cartan $\frt_{\ov k}$ that are not equal to $\frz_{\ov k}$. Let $\frR^{\min}_{\ov k}(W, \frt)\subset \frR_{\ov k}(W,\frt)$ be the set of minimal elements under inclusion. Let $R^{\min}_{\ov k}(W, \frt)=\#\frR^{\min}_{\ov k}(W, \frt)$. 

\begin{lemma}\label{l:W coreg exist}
\begin{enumerate}
\item If $q>R^{\min}_{\ov k}(W, \frt)$, then $W$-coregular linear functions $\l: \frt_{0}/\frz\to k$ exist.
\item Assume $p$ is \goodp (see \S\ref{sss:goodp}). Choose a set of simple roots $\{\a_{i}\}_{i\in I}$ for the abstract root system $\Phi(G_{\ov k}, T_{\ov k})$, and let $\vp_{i}^{\vee}\in \xcoch(T_{\ov k})_{\QQ}$ be the corresponding fundamental coweights. Then
\begin{equation*}
R^{\min}_{\ov k}(W, \frt)\le \#(\bigcup_{i\in I} W\cdot \vp^{\vee}_{i})
\end{equation*}
where the right side is a subset of $\xcoch(T_{\ov k})_{\QQ}$.
\end{enumerate}
\end{lemma}
\begin{proof}
(1) View the dual space $V=(\frt_{0}/\frz)^{*}$ as an affine space over $k$. Choose an identification $\io:  \frt_{0,\ov k}\cong \frt_{\ov k}$ by choosing a Borel over $\ov k$ containing $\frt_{0,\ov k}$. Let $V^{\c}_{\ov k}\subset V_{\ov k}$ be the open subscheme given by the complement of $\fra^{\bot}$ for $W$-relevant subspaces $\fra\subset \frt_{0,\ov k}$ that are minimal among those not equal to $\frz_{\ov k}$. Then $V^{\c}_{\ov k}$ descends to an open subscheme $V^{\c}\subset V$. The complement $V_{\ov k}-V^{\c}_{\ov k}$, being the union of $R^{\min}_{\ov k}(W,\frt)$ proper linear subspaces, contains at most $q^{\dim V-1}R^{\min}_{\ov k}(W,\frt)$ points. Thus when $q>R^{\min}_{\ov k}(W,\frt)$, we have $\#V(k)=q^{\dim V}>q^{\dim V-1}R^{\min}_{\ov k}(W,\frt)\ge \#(V-V^{\c})(k)$, which implies that $V^{\c}(k)$ is non-empty.

(2) By Lemma \ref{l:W rel}\eqref{W rel cut by roots},  $W$-relevant $\ov k$-subspaces $\fra\subset \frt_{\ov k}$ are exactly the $G$-relevant toral $\ov k$-subspaces, which by Lemma \ref{l:G rel cut by roots} are cut out by $\{d\a_{j}\}_{j\in J}$ for a subset $J$ of a choice of simple roots $\{\a_{i}\}_{i\in I}$ in $\Phi=\Phi(G_{\ov k}, \frt_{\ov k})$. The minimal elements in $R_{\ov k}(W,\frt)$ therefore are $W$-translates of $\fra_{i}=\cap_{j\in I-\{i\}}\ker(d\a_{j})$, for all choices of $i\in I$. Note that $\fra_{i}$ can be reconstructed from the fundamental coweight $\vp_{i}^{\vee}$: $\fra_{i}=\cap_{\a\in \Phi, \j{\a,\vp^{\vee}_{i}}=0}\ker(d\a)$.  Therefore, the number of $W$-translates of $\fra_{i}$ for various $i\in I$ is bounded above by the cardinality of $\cup_{i\in I}W\cdot \vp^{\vee}_{i}$, hence the desired inequality.
\end{proof} 

Recall from \S\ref{sss:tori} that for each $w\in W$, $\frt_{w}$ has a class of embeddings $\io:\frt_{w}\incl \frg$ that form a single $G(k)$-orbit.  It therefore makes sense to talk about relevant $k$-subspaces of $\frt_{w}$ and $G$-coregular linear functions on $\frt_{w}$.

Fix a non-trivial additive character $\psi : k \to \Qlbar$. The following proposition gives an explicit construction of selection functions.

\begin{prop}\label{p:cons sel fn from t} Assume that $p\ne2$ when $G^{\der}$ has a direct factor isomorphic to $\SO_{2n+1}$ (in addition to the assumption \eqref{p prime to ZG}). For each $w\in W$, let $\l_{w}: \frt_{w}\to k$ be a $W$-coregular linear function.  Then the following function on $\frc(k)$
\begin{equation*}
\xi=\frac{1}{\#W}\sum_{w\in W}\chi_{w!}(\psi\c\l_{w})
\end{equation*}
is a selection function. 
\end{prop}
\begin{proof}
Replacing $G$ by $G/C_{G}$, we may assume that $G$ is semisimple. By our assumption on $p$, we have $\frz=0$ which is the smallest $W$-relevant subspace (see Lemma \ref{l:W rel}(2)).

It is clear that $\xi(0)=0$. Let $A\subset G$ be a nontrivial $G$-relevant subtorus with Lie algebra $\fra\subset \frg$. We would like to show that $\j{\one_{\fra}, \chi^{*}\xi}_{\frg}=0$. We have 
\begin{equation*}
\j{\one_{\fra}, \chi^{*}\chi_{w!}(\psi\c\l_{w})}_{\frg}
=\j{\chi_{!}\one_{\fra}, \chi_{w!}(\psi\c\l_{w})}_{\frc(k)}
=\int_{\fra\times_{\frc(k)}\frt_{w}}\pi_{2}^{*}(\psi\c\l_{w}).
\end{equation*}
Here $\pi_{2}: \fra\times \frt_{w}\to \frt_{w}$ and $\pr_{1}: \fra\times \frt_{w}\to \fra$ are the projections.

We would like to introduce a family of subspaces $\cF$ in the vector space $\fra\times\frt_{w}$ and apply Lemma \ref{l:int subspace arr} below.  We first consider the case $\fra=\frt_{v}$ for some $v\in W$. We have
\begin{equation*}
\frt_{v}\times_{\frc(k)}\frt_{w}=\{(x,y)\in \frt_{\ov k}\times\frt_{\ov k}|\Fr(x)=vx,\Fr(y)=wy, y=ux\textup{ for some } u\in W\}.
\end{equation*}
We let $\cF_{\ov k}$ be those $\ov k$-subspaces of $\frt_{\ov k}^{2}$ of the form $\G(u|\fra_{1})$, where $\fra_{1}\subset \frt_{\ov k}$ is a $W$-relevant $\ov k$-subspace, $u\in W$, and $\G(u|\fra_{1})$ is the graph of $u:\fra_{1}\subset \frt_{\ov k}\xr{u} \frt_{\ov k}$ (note that different $u$'s may give the same element in $\cF_{\ov k}$). Let $\cF(\frt_{v}, \frt_{w})$ be those elements in $\cF_{\ov k}$ that descend to subspaces of $\frt_{v}\times\frt_{w}$ (using the Frobenius structure on $\frt_{v}$ and $\frt_{w}$).

\begin{claim} The family $\cF(\frt_{v}, \frt_{w})$ satisfies the axioms in Lemma \ref{l:int subspace arr}. 
\end{claim}
\begin{proof}[Proof of Claim] Clearly $\{0\}\in \cF_{\frt_{v},\frt_{w}}$ since $\{0\}\subset \frt_{\ov k}$ is $W$-relevant. Given two elements $\G(u_{1}|\fra_{1})$ and $\G(u_{2}|\fra_{2})$ in $\cF(\frt_{v}, \frt_{w})$, we will construct a decreasing sequence of $W$-relevant subspaces 
\begin{eqnarray*}
\fra_{1}\supset\fra^{(0)}\supset\fra^{(1)}\supset\fra^{(2)}\supset\cdots
\end{eqnarray*}
as follows. Let $\fra^{(0)}=\fra_{1}\cap\fra_{2}$. 
If the $W$-relevant subspace $\fra^{(i)}$ is constructed, let $\fra^{(i+1)}=\{x\in \fra^{(i)}|u_{1}(x)=u_{2}(x)\}$. This is a relevant subspace because it is the intersection of $\fra^{(i)}$ and $\frt^{u_{1}^{-1}u_{2}}$, which is $W$-relevant by Lemma \ref{l:W rel}(1). Since $\dim\fra_{1}<\infty$, this sequence has to stabilize at some point. Let $\fra^{(\infty)}=\fra^{(N)}$ for large $N$. By construction, we must have $u_{1}x=u_{2}x$  for all $x\in \fra^{(\infty)}$. Finally we have
\begin{equation*}
\G(u_{1}|\fra_{1})\cap\G(u_{2}|\fra_{2})=\G(u_{1}|\fra^{(\infty)}).
\end{equation*}
This verifies that $\cF(\frt_{v}, \frt_{w})$ satisfies the axioms in Lemma \ref{l:int subspace arr}.
\end{proof}

Using the notation from Lemma \ref{l:int subspace arr} below, we denote the union of all subspaces in $\cF(\frt_{v}, \frt_{w})$ by $C_{\cF(\frt_{v}, \frt_{w})}$.

\begin{claim} We have
\begin{equation}\label{union form tt}
C_{\cF(\frt_{v}, \frt_{w})}=\frt_{v}\times_{\frc(k)}\frt_{w}.
\end{equation}
\end{claim}
\begin{proof} Any subspace $\G(u|\fra)\in \cF_{\ov k}$ is contained in $\frt_{\ov k}\times_{\frc_{\ov k}}\frt_{\ov k}$. Therefore any element in $\cF(\frt_{v}, \frt_{w})$ is automatically contained in $\frt_{v}\times_{\frc(k)}\frt_{w}$. This shows that the left side is contained in the right side. 

We show the other inclusion. The Frobenius structure on $\frt_{\ov k}\times \frt_{\ov k}$ that we consider here is $\ph(x,y)=(v^{-1}\Fr(x), w^{-1}\Fr(y))$. For $u\in W$, $\ph$ sends the graph $\G(u)\subset \frt_{\ov k}\times \frt_{\ov k}$ to $\G(u')$, where $u'=w^{-1}\d(u)v$. Here recall $\d$ is the Coxeter group automorphism of the abstract Weyl group $W$ given by the $k$-structure on $G$. Let $\fra=\frt^{u^{-1}u'}$, which is $W$-relevant. Then $\G(u|\fra)$ is stable under the Frobenius structure $\ph$, hence $\G(u|\fra)\in \cF(\frt_{v},\frt_{w})$. 

Now let $(x,y)\in \frt_{v}\times_{\frc(k)}\frt_{w}$. Since $x$ and $y$ have the same image in $\frc$, there exists some $u\in W$ such that $y=ux$. Now we have
\begin{equation*}
\Fr(x)=vx, \quad \Fr(y)=wy, \quad y=ux.
\end{equation*}
From the above equations we conclude that 
\begin{equation*}
wux=wy=\Fr(y)=\Fr(ux)=\d(u)\Fr(x)=\d(u)vx.
\end{equation*}
Therefore $x\in \frt^{u^{-1}w^{-1}\d(u)v}_{\ov k}=\fra$ defined above. We have $(x,y)\in \G(u|\fra)\subset C_{\cF(\frt_{v}, \frt_{w})}$. This proves the other inclusion.
\end{proof}

For general $G$-relevant toral $\fra$, we may assume $\fra\subset \frt_{v}$ for some $v\in W$. Define $\cF(\fra, \frt_{w})$ to be the subset of those elements $U\in \cF(\frt_{v}, \frt_{w})$ contained in $\fra\times\frt_{w}$. The two Claims above imply that $\cF(\fra, \frt_{w})$ satisfies the axioms in Lemma \ref{l:int subspace arr}, and that
\begin{equation}\label{union form at}
C_{\cF(\fra, \frt_{w})}=\fra\times_{\frc(k)}\frt_{w}.
\end{equation}

Now let $\mu=\l_{v}|_{\fra}\in \fra^{*}$.  
We claim that both linear functions
\begin{equation*}
\mu\c\pi_{1}, \l_{w}\c\pi_{2}\in (\fra\times\frt_{w})^{*}
\end{equation*}
are $\cF(\fra, \frt_{w})$-coregular (the notion introduced in the statement of Lemma \ref{l:int subspace arr} below). Indeed, any subspace in $\cF(\fra, \frt_{w})$ is of the form $\G(u|\fra_{1})$ for a $W$-relevant subspace $\fra_{1}\subset \fra$ and $u\in W$ such that $u(\fra_{1})\subset \frt_{w}$. Under the identification $\fra_{1}\cong \G(u|\fra_{1})$, the function $\mu\c\pi_{1}|_{\G(u|\fra_{1})}$ becomes the function $\mu|_{\fra_{1}}$, which is nonzero since $\mu$ is $G$-coregular on $\fra$. The function $\l_{w}\c\pi_{2}|_{\G(u|\fra_{1})}$ becomes the function $\l_{w}\c u|_{\fra_{1}}$, which is nonzero since $\l_{w}|_{u(\fra_{1})}$ is nonzero ($\l_{w}$ is $G$-coregular on $\frt_{w}$).

By Lemma \ref{l:int subspace arr}, we have
\begin{equation*}
\int_{C_{\cF(\fra,\frt_{w})}}\pi_{2}^{*}(\psi\c\l_{w})=\int_{C_{\cF(\fra,\frt_{w})}}\pi_{1}^{*}(\psi\c\mu).
\end{equation*}
Using \eqref{union form at}, we have
\begin{equation*}
\int_{\fra\times_{\frc(k)}\frt_{w}}\pi_{2}^{*}(\psi\c\l_{w})=\int_{\fra\times_{\frc(k)}\frt_{w}}\pi_{1}^{*}(\psi\c\mu).
\end{equation*}

Now sum over all $w\in W$ of the right side above, we get
\begin{equation}\label{sum w}
\sum_{w\in W}\j{\one_{\fra}, \chi^{*}\chi_{w!}(\psi\c\l_{w})}_{\frg}=\sum_{w\in W}\int_{\fra\times_{\frc(k)}\frt_{w}}\pi_{1}^{*}(\psi\c\mu)=\j{\psi\c\mu, \sum_{w\in W}\pi_{1!}\one_{\fra\times_{\frc(k)}\frt_{w}}}_{\fra}
\end{equation}

We compute the function $\sum_{w\in W}\pi_{1!}\one_{\fra\times_{\frc(k)}\frt_{w}}$. This is the pullback of the following function on $\frc(k)$ along the map $\chi_{\fra}: \fra\to\frc(k)$
\begin{equation*}
\y=\sum_{w\in W}\chi_{w!}\one_{\frt_{w}}.
\end{equation*}

\begin{claim} The function $\y$ is the constant function on $\frc(k)$ with value $\#W$. 
\end{claim}
\begin{proof}[Proof of Claim] Let $x\in \frc(k)$ and let $\Xi\subset \frt_{\ov k}$ be all the liftings of $x$ (viewed as a $\ov k$-point). This is a $W$-orbit with a Frobenius action. The value of $\chi_{w!}\one_{\frt_{w}}$ at $x$ is the cardinality of the following subset of $\Xi$
\begin{equation*}
\Xi_{w}:=\{y\in\Xi|\Fr(y)=wy\}.
\end{equation*}
Summing over $w\in W$, we get
\begin{equation}\label{Xiw}
\sum_{w\in W}\#\Xi_{w}=\#\{(y,w)\in \Xi\times W|\Fr(y)=wy\}.
\end{equation}
Denote the set on the right by $\wt\Xi$. Then the projection $\wt\Xi\to \Xi$ has fiber over $y\in \Xi$ a torsor under $W_{y}$ (the stabilizer of $y$ under $W$).  Since $W$ acts transitively on $\Xi$, we have  $\#W_{y}=\#W/\#\Xi$ for all $y\in \Xi$, hence $\#\wt\Xi=\#W$. This shows that the left side of \eqref{Xiw} is $\#W$, hence $\y(x)=\#W$. 
\end{proof}

Plugging into \eqref{sum w} we get
\begin{equation*}
\sum_{w\in W}\j{\one_{\fra}, \chi^{*}\chi_{w!}(\psi\c\l_{w})}_{\frg}=\#W\j{\psi\c\mu, \chi_{\fra}^{*}\one_{\frc(k)}}_{\fra}=\#W\int_{\fra}\psi\c\mu.
\end{equation*}
Since $\mu=\l_{v}|_{\fra}$, $\fra\ne0$ and $\l_{v}$ is $W$-coregular, $\mu$ is nonzero and $\psi\c\mu$ is a nontrivial character on $\fra$. Therefore the above sum is zero. This proves the proposition. 
\end{proof}

\begin{lemma}\label{l:int subspace arr}
Let $V$ be a finite-dimensional $k$-vector space, and let $\cF$ be a set of linear subspaces of $V$ satisfying the conditions:
\begin{enumerate}
\item $\{0\}\in \cF$.
\item If $U_{1},U_{2}\in \cF$, so is $U_{1}\cap U_{2}$.
\end{enumerate}
Let $C_{\cF}=\cup_{U\in \cF}U\subset V$, a conic subset.

Let $\l\in V^{*}$ be an element that is {\em $\cF$-coregular}, i.e., $\l|_{V'}\ne0$ for every nonzero member $V'\in \cF$. Then
\begin{equation*}
\int_{C_{\cF}}\psi\c\l=\sum_{\{0\}\sne U_{1}\sne\cdots\sne U_{\ell}, U_{i}\in \cF}(-1)^{\ell}
\end{equation*}
where the sum is over all strictly increasing chains in $\cF$ starting from $\{0\}$. In particular, the left side is independent of the choices of $\cF$-coregular $\l$ and the nontrivial additive character $\psi$.
\end{lemma}
\begin{proof}
For each $U\in\cF$, let $c_{U}$ be the sum
\begin{equation*}
\sum_{U=U_{0}\sne U_{1}\sne\cdots\sne U_{\ell}, U_{i}\in \cF}(-1)^{\ell}.
\end{equation*}
Consider the following function on $V$
\begin{equation*}
c=\sum_{U\in \cF}c_{U}\one_{U}.
\end{equation*}
We claim that $c$ is the constant function on $C_{\cF}$ with value $1$. Clearly $c$ is supported on $C_{\cF}$. Now let $x\in C_{\cF}$, and let $\cF_{x}$ be the subset of $\cF$ consisting of those subspaces containing $x$. We have
\begin{equation*}
c(x)=\sum_{U_{0}\sne U_{1}\sne\cdots\sne U_{\ell}, U_{i}\in \cF_{x}}(-1)^{\ell}.
\end{equation*}
This is the Euler characteristic of the geometric realization $|N(\cF_{x})|$ of the nerve of $\cF_{x}$ (a poset inherited from $\cF$). Since $\cF_{x}$ has an initial object by assumption (the intersection of all elements in $\cF_{x}$), $|N(\cF_{x})|$ is contractible, hence $c(x)=\chi(|N(\cF_{x})|)=1$. 

Using that $c=\one_{C_{\cF}}$, we have
\begin{equation*}
\int_{C_{\cF}}\psi\c\l=\sum_{U\in \cF}c_{U}\int_{U}\psi\c(\l|_{U}). 
\end{equation*}
When $U\ne \{0\}$, $\l|_{U}$ is nonzero, and hence $\int_{U}\psi\c(\l|_{U})=0$. Therefore the above sum reduces to one term $c_{\{0\}}$, which is what we wanted to show.
\end{proof}

\subsection{Springer sheaves and Fourier transform}\label{ss:Spr}
In this subsection we prove a result (well-known to experts) in Proposition \ref{prop:signFT} concerning the $W$-action on the Grothendieck-Springer sheaf and Fourier transform. They will be used in the next subsection. 

In this subsection we assume that $p=\ch(k)$ is good for $G$. We also fix a nontrivial additive character  $\psi:k\to \Qlbar^\times$. It defines an Artin-Schreier local system $\AS_\psi$ on $\AA^1$.

Consider the Grothendieck alterations
\begin{align}\label{GS alteration}
\xymatrix{ & \wt{\frg} \ar[dl]_{\pi} \ar[dr]^{\varepsilon} & & & \wt{\frg}^* \ar[dl]_{\pi'} \ar[dr]^{\varepsilon'} & \\
\frg & & \frt & \frg^{*} & & \frt^{*}. 
}
\end{align}
for $\frg$ and for $\frg^*$. Here, $\wt{\frg}^* = \{(v,\frb) \in \frg^* \times \cB \mid v\in \frb^{\bot} \}$, $\pi'$ is the projection $(v,\frb)\mapsto v$, and $\varepsilon'(v,\frb) = v|_{\frb}$.

We take the product of the first diagram above with $\frt^*$, and we get the diagram
\begin{align}
\xymatrix{ & \widetilde{\frg} \times \frt^* \ar[dl]_{\pi \times \id} \ar[dr]^{\varepsilon \times \id} & & \\
\frg\times \frt^* & & \frt \times \frt^{*} \ar[r]^{\langle \cdot,\cdot\rangle} & \AA^1
}
\end{align}
where $\j{\cdot,\cdot}$ is the evaluation map.  On the other hand, 
letting $\frc^*=\frg^*\sslash G$, we have the natural morphism 
\begin{align*}
\wt{\pi}' : \wt{\frg}^{*} \to \frg^* \times_{\frc^*} \frt^*.
\end{align*}

\begin{defn}  
We call
\begin{align} \cK_{\univ} &= (\pi \times \id)_{!}(\varepsilon \times \id)^*\langle \cdot,\cdot \rangle^* \AS_{\psi}\in D^b
_c(\frg\times\frt^*,\Qlbar)
\end{align} 
the \emph{universal Grothendieck sheaf} and
\begin{align}
\cK_{\univ}' &= \wt{\pi}'_{!}\Qlbar \in D^b
_c(\frg^*\times_{\frc^*}\frt^*,\Qlbar)
\end{align}
the \emph{universal Springer sheaf}. 
\end{defn}

\begin{lemma}
    The map $\wt{\pi}' : \wt{\frg}^{*} \to \frg^* \times_{\frc^*} \frt^*$ is proper, small and birational. In particular, up to a shift, $\cK'_\univ=\wt{\pi}'_{!}(\Qlbar)$ is the intermediate extension of the constant perverse sheaf from $\frg^{*,\rs} \times_{\frc^*} \frt^*$, i.e.,
\begin{align*}
  \cK'_\univ[\dim \frg] = \IC(\frg^{*} \times_{\frc^*} \frt^*).
\end{align*}
\end{lemma}
\begin{proof}
    The map $\wt{\pi}'$ is an isomorphism over $\frg^{*,\rs} \times_{\frc^*} \frt^*$, and is small by \cite[Section 3]{Lusztig}. 
\end{proof}

\begin{remark}[$W$-equivariant structure on $\cK'_\univ$] The Weyl group acts on the $\frt^*$-factor of $\frg^{*,\rs} \times_{\frc^*}\frt^*$ and there is a natural $W$-equivariant structure on the constant sheaf. This induces a $W$-equivariant structure on $\wt{\pi}'_{!}\Qlbar$, denoted by 
\begin{align*} \b_w : \cK'_{\univ} \to w^*\cK'_{\univ}. \end{align*}
\end{remark}

\begin{lemma}[$W$-equivariant structure on $\cK_\univ$]
    The complex $\cK_{\univ}$ carries a canonical $W$-equivariant structure, denoted by $\a_w : \cK_{\univ} \to w^*\cK_{\univ}$.
\end{lemma}
\begin{proof}
The map $\pi \times \id$ is small, so $\cK_{\univ}$ is the intermediate extension of its restriction to the regular locus $\frg^{\rs} \times \frt^*$. It therefore suffices to construct the $W$-equivariant structure on the regular locus (for the action of $W$ only on $\frt^*$). On $\wt{\frg}^{\rs} \times \frt^*$ we have a second $W$-action, acting only on the first factor. Denote this action by 
\begin{align*}
\wt{a}_w : \wt{\frg}^{\rs} \times \frt^* \to \wt{\frg}^{\rs} \times \frt^*. 
\end{align*}
The map $\varepsilon$ to the universal Cartan intertwines the $W$-action on $\wt{\frg}$ with the opposite action on $\frt$, i.e. $\varepsilon(w.(x,\frb)) = w^{-1}\varepsilon(x,\frb)$. Indeed, $w\in W$ acts on $\wt{\frg}$ by $(x,\frb) \mapsto (x,\Ad_{w}\frb)$, and $\varepsilon (x, \Ad_{w}\frb) = x \mod [\Ad_{w}\frb, \Ad_{w}\frb]$. The latter is mapped to $\Ad_{w^{-1}}(x)$ in $\frt = \frb/ [\frb, \frb]$ under the canonical identification. From here, it is easy to check that 
\[a_{w}^{*}(\varepsilon \times \id)^*\langle \cdot,\cdot \rangle^* \AS_{\psi} = \wt{a}_{w^{-1}}^* (\varepsilon \times \id)^*\langle \cdot,\cdot \rangle^* \AS_{\psi}.\]
The restriction $\pi^{\rs} \times \id$ of $\pi \times \id$ to the regular locus is $W$-invariant for the action $\wt{a}_{w}$. Therefore, we obtain a canonical $W$-equivariant structure (for $a_w$) on $\cK_{\univ}^{\rs} = (\pi^{\rs} \times \id)_!(\varepsilon^{\rs}\times \id)^* \langle \cdot,\cdot \rangle^* \AS_{\psi}$. 
\end{proof}

The sheaves $\cK_\univ$ and $\cK'_\univ$ are related by Fourier transform. For an algebraic stack $Y$ over $k$, and a vector bundle $E \to Y$ with dual bundle $E^*\to Y$, we denote by $\FT_{E}$ the Fourier transform
\begin{eqnarray*}
    \FT_E: D^b_c(E,\Qlbar)&\isom& D^b_c(E^*,\Qlbar)\\
    \cF &\mapsto & p_{E^*,!}(p_E^*\cF\ot \j{\cdot,\cdot}^*\AS_\psi).
\end{eqnarray*}
Here $p_E: E\times_YE^*\to E$, $p_{E^*}: E\times_YE^*\to E^*$ are the projections, and $\j{\cdot,\cdot}: E\times_YE^*\to \AA^1$ is the evaluation map.

Recall the following well-known property of Fourier transform. 
\begin{lemma}\label{lem:ASdelta} 
    Let $Y$ be an algebraic stack over $k$, and $E\to Y$ a vector bundle of rank $r$, with dual bundle $E^{*} \to Y$. Let $f:E\to \AA^1_{Y}$ be a linear map over $Y$ (i.e. a global section $s_f: Y\to E^{*}$), and let $\AS_{f} = f^*a^*(\AS_{\psi})$ where $a:\AA^1_Y\to \AA^1$ is the projection. Then, we have a canonical isomorphism
    \begin{equation*} \FT_{E} (\AS_{f} )\cong s_{f,!}(\overline{\QQ}_{\ell,Y})[-2r](-r). 
    \end{equation*}
    \end{lemma}
With this, we can prove the following. 
\begin{prop}\label{p:FT K} Let $\iota : \frg^* \times_{\frc^*}\frt^* \to \frg^* \times \frt^*$ be the canonical closed embedding. We have 
\begin{equation*}\FT_{\frg \times \frt^*}(\cK_{\univ}) = \iota_{!}\cK_{\univ}'[2(\dim\cB-\dim\frg)](\dim\cB-\dim\frg), 
\end{equation*} 
where we understand $\frg \times \frt^*$ as trivial vector bundle over $\frt^*$. 
\end{prop}
\begin{proof}
Consider the embeddings 
\begin{align*}
    \wt{\frg}^* \times \frt^* \xrightarrow{x^*}  \frg^* \times \frt^* \times \cB, \\
    \wt{\frg}\times \frt^* \xrightarrow{x} \frg \times \frt^* \times \cB 
\end{align*}
induced by the embedding of $\wt{\frg}^*$ (resp. $\wt{\frg}$) into the trivial bundle $\frg^*\times\cB$ (resp. $\frg\times\cB$) over $\cB$. We denote the dual of $x$ by 
\begin{align} \frg^* \times \frt^* \times \cB \xrightarrow{{}^tx}  (\wt{\frg})^* \times \frt^*.
\end{align}
Here, by $(\wt{\frg})^*$ we denote the dual vector bundle to $\wt\frg\to \cB$. Let $f = \langle \cdot, \cdot \rangle \circ (\varepsilon\times\id_{\frt^*}): \wt\frg\times\frt^*\to \AA^1$, with corresponding global section $s_f : \cB \times \frt^* \to (\wt{\frg})^* \times \frt^*$. 

It is easy to check that the diagram
\begin{align*}
\xymatrix{
\wt{\frg}^*   \ar[d] \ar[r] & \cB \times \frt^* \ar[d]^{s_f}\\
\frg^* \times \frt^* \times \cB \ar[r]^{{}^tx} & (\wt{\frg})^* \times \frt^*
}
\end{align*} 
is Cartesian, with the top horizontal arrow taking $(x,\frb)$ to $(\frb, \varepsilon(x,\frb))$, and the left vertical arrow mapping $(x,\frb)$ to $(x,\frb, \varepsilon(x,\frb))$. The claim follows from basic properties of Fourier transform with respect to base-change and dualizing, combined with Lemma \ref{lem:ASdelta}, as in \cite[\S 2.8.]{AHJR}. 
\end{proof}

We will prove the compatibility of the $W$-equivariant structures on $\cK_\univ$ and $\cK'_\univ$ under Fourier transform. 

\begin{prop}\label{prop:signFT} We have $\FT_{\frg}(\a_w) = \sgn(w) \b_w$. 
\end{prop}
\begin{proof} Since both $\cK'_\univ$ and $\cK_\univ$ are simple perverse sheaves up to shifts, $\a_w$ and $\b_w$ are each unique up to a nonzero scalar. We write $\FT_\frg(\a_w)=c_w\b_w$ for some $c_w\in \Qlbar^\times$. The assignment $w \mapsto c_w\in \Qlbar^\times$ determines a character of $W$. We want to show that $c_w=\sgn(w)$. It is therefore enough to compute $c_w$ after restricting to any point of $\frt^*$, in particular to $0 \in \frt^*$. Now $W$ acts on both $\cK_{0}= \cK_{\univ}|_{\frg \times \{0\} }$ and $\cK'_0=\cK'_{\univ} |_{\frg^* \times_{\frc^*} \times \{0\}}$. To show that $c_w$ is the sign character of $W$, it suffices to show that $\FT_\frg$ transforms the trivial isotypic component of $\cK_0$ to the sign isotypic component of $\cK'_0$ under the respective $W$-actions.

Note that $\cK_{0}= \cK_{\univ}|_{\frg \times \{0\} }$ is the Grothendieck sheaf, and the induced $W$-action is the Springer action characterized as follows. The isotypic component for the trivial $W$-representation is the constant sheaf, and the isotypic part for the sign-character is the $\delta$-sheaf at $0$ (up to shift). 

On the other hand, $\cK'_{0} = \cK'_{\univ} |_{\frg^* \times_{\frc^*} \times \{0\}}$ is the Springer sheaf, with $W$-action induced by restricting the action on the Grothendieck sheaf (for $\frg^*$). Now Fourier transform exchanges the constant sheaf with the $\delta$-sheaf, hence it exchanges the sign-component with the isotypic component of the trivial $W$-character. This proves the claim. 
\end{proof}

\sss{Specializations to $\l$}
For any $\l \in \frt^*(\ov k)$, let $\AS_{\l} = \l^* \AS_{\psi}$, a local system on $\frt_{\ov k}$. Write
\begin{align*} \cK_{\l} := \cK_{\univ}|_{\frg \times \{ \l\}} =  \pi_{!}\varepsilon^* \AS_{\l} \in D^b_c(\frg_{\ov k}, \Qlbar),\\
\cK'_{\l} := \cK'_{\univ} |_{\frg^* \times_{\frc^{*}} \{ \l\}}\in D^b_c(\frg^* \times_{\frc^{*}} \{ \l\}, \Qlbar).
\end{align*}
Then we have
\begin{align*} i_{\l,!}\cK'_{\l} = \pi_{!}'\varepsilon'^* \delta_{\l}, \end{align*}
where $i_{\l}: \frg^* \times_{\frc^{*}} \{ \l\} \hookrightarrow \frg^*_{\ov k}$ is the embedding, and $\delta_{\l}$ the skyscraper sheaf supported at $\l\in \frt^*_{\ov k}$. 

The $W$-equivariant structure on $\cK_{\univ}$ induces canonical isomorphisms 
\begin{equation*}
\a_{\l,w} : \cK_{\l} \to \cK_{w\l}
\end{equation*}
for all $\l\in \frt^*(\ov k)$. Similarly, from the $W$-equivariant structure on $\cK'_{\univ}$ we obtain canonical isomorphisms
\begin{equation*}
\b_{\l,w} : \cK'_{\l} \to \cK'_{w\l}.
\end{equation*}

Restricting the statements of Proposition \ref{p:FT K} and \ref{prop:signFT} to $\l$, we obtain the following consequence.
\begin{cor}\label{c:signFT lam} For each $\l \in \frt^*_{\ov k}$, we have a canonical isomorphism $\FT_{\frg}(\cK_{\l}) \cong i_{\l,!}\cK'_{\l}[2(\dim\cB-\dim \frg)](\dim \cB-\dim \frg)$. 
Moreover,
\begin{equation*}
\FT_\frg(\a_{\l,w})=\sgn(w)\b_{\l,w}, \quad\forall w\in W.
\end{equation*} 
\end{cor}

\sss{Trace functions of $\cK_{\l}$}
Now suppose $\l\in \frt^*(\ov k)$ satisfies $\Fr(\l)=w\l$ for some $w\in W$. Here $\Fr$ is the Frobenius action on $\frt^{*}\ot\ov k$ that is on the $\ov k$-factor. For a scheme $Y$ defined over $k$, let $\Fr_{Y}:Y_{\ov k}\to Y_{\ov k}$ be the base changes of the $q$-power Frobenius morphism on $Y$ (so $\Fr_{Y}$ is a morphism over $\ov k$). We have a commutative diagram
\begin{equation*}
\xymatrix{\frt_{\ov k}\ar[r]^{\l}\ar[d]^{\Fr_{\frt}} & \AA^{1}_{\ov k}\ar[d]^{\Fr_{\AA^{1}}}\\
\frt_{\ov k}\ar[r]^{\Fr(\l)} & \AA^{1}_{\ov k}
}
\end{equation*}
From this, and the canonical Weil structure $\Fr^{*}\AS_{\psi}\cong \AS_{\psi}$ on $\AS_{\psi}$, we deduce an isomorphism of local systems on $\frt_{\ov k}$
\begin{equation*}
\Fr^{*}_{\frt}\cL_{\Fr(\l)}\cong \cL_{\l}.
\end{equation*}
Therefore we obtain an isomorphism
\begin{equation*}
\Fr^{*}_{\frg}\cK_{\Fr(\l)}\cong \cK_{\l}.
\end{equation*}
The composition
\begin{equation*}
\Fr^{*}_{\frg}\cK_{\l}\xr{\Fr^{*}_{\frg}\a_{\l,w}}\Fr^{*}_{\frg}\cK_{w\l}=\Fr^{*}_{\frg}\cK_{\Fr(\l)}\cong \cK_{\l}
\end{equation*}
gives a Weil structure on $\cK_{\l}$. We denote the resulting Weil sheaf (on $\frg_{\ov k}$) by $\cK_{\l,w}$. 

Taking the associated functions on the $k$-points of $\frg$ using trace of Frobenius, we obtain a function
\begin{equation*}
\k_{\l,w}\in \Qlbar[\frg].
\end{equation*}
Similarly, the canonical isomorphisms $\beta_{\l,w} : \cK_{\l}' \to \cK'_{w\l}$ induce a Weil structure on $i_{\l!}\cK'_{\l}$, and we obtain a Weil sheaf  $\cK'_{\l,w}$ on $\frg^*$ (whose underlying sheaf is $i_{\l!}\cK'_{\l}$), with trace function $\kappa_{\l,w}' \in \Qlbar[\frg]$. 

\sss{Twisted Frobenius action on cohomology of Springer fibers}\label{sss:tw Fr Spr} For any $v\in \frg^*$, let $\cB_{v}$ be the Springer fiber of $v$, and write $\cB_{v}^{\l} = \cB_{v} \cap (\varepsilon'^{*})^{-1}(\l)$. Then the stalk of $\cK'_\l$ at $v$ is $\cohog{*}{\cB^{\l}_{v}}$. To make the function $\kappa_{\l,w}'(v)$  explicit, we explain how Frobenius acts on $\cohog{*}{\cB^{\l}_{v}}$. 

Let $\Xi\subset \frt^*_{\ov k}$ be the $W$-orbit that has the same image in $\frc^* = \frg^* \sslash G$ as $v$. We have a decomposition
\begin{equation*}
\cB_{v}=\coprod_{\l\in \Xi}\cB^{\l}_{v}
\end{equation*}
according to the image of the projection $\varepsilon':\wt\frg^*\to \frt^*$. Recall that we have a natural action of $W$ on $\cohog{*}{\cB_{v}}=\op_{\l\in \Xi}\cohog{*}{\cB^{\l}_{v}}$. The element $w\in W$ sends the summand $\cohog{*}{\cB^{\l}_{v}}$ onto $\cohog{*}{\cB^{w\l}_{v}}$. If $w\l=\Fr(\l)$, then $\Fr^{*}$ maps $\cohog{*}{\cB^{w\l}_{v}}$ back to $\cohog{*}{\cB^{\l}_{v}}$. The Frobenius action on $\cohog{*}{\cB^{\l}_{v}}$ is the composition
\begin{equation*}
    \Fr^{*}\c w: \cohog{*}{\cB^{\l}_{v}}\isom \cohog{*}{\cB^{w\l}_{v}}\isom \cohog{*}{\cB^{\l}_{v}}.
\end{equation*}
We conclude that $ \kappa_{\l,w}'(v) = \Tr(\Fr^{*}\c w, \cohog{*}{\cB^{\l}_{v}})$. Since $\FT_{\frg}(\cK_{\l})$ is isomorphic to $ i_{\l,!}\cK'_{\l} = \pi'_{!}\varepsilon'^* \delta_{\l}$ up to a shift and twist,  together with Proposition \ref{prop:signFT}, we obtain the following. 
\begin{cor} \label{cor:FTkappa}
For any $\l\in \frt^*(\ov k)$ with $\Fr(\l) = w\l$, the value of $\FT(\k_{\l,w})$ at $v \in \frg^{*}$ is 
\begin{equation*}
q^{\dim\frg-\dim\cB}\sgn(w)\Tr(\Fr^{*}\c w, \cohog{*}{\cB^{\l}_{v}}).
\end{equation*}
\end{cor}

\subsection{Admissible collections}\label{ss:adm}
This subsection concerns a refinement of the construction of selection functions in Proposition \ref{p:cons sel fn from t}. Under a certain compatibility condition on the collection $\{\l_{w}\}_{w\in W}$, we will show that the selection function in {\em loc. cit.} can be expressed using the trace functions of Grothendieck-Springer sheaves.

\begin{defn}\label{def:adm}
A collection of linear functions $\{\l_{w}\in\frt^{*}_{w}\}_{w\in W}$ is called {\em admissible}, if it satisfies the following: whenever $y\in \frt_{w_{1}}\cap\frt_{w_{2}}$ for $w_{1},w_{2}\in W$,  we have $\j{\l_{w_{1}},y}=\j{\l_{w_{2}}, y}$. Such a collection is called $W$-coregular if $\l_{w}$ is $W$-coregular for each $w\in W$.
\end{defn}
Here the intersection $\frt_{w_{1}}\cap \frt_{w_{2}}$ takes place in $\frt_{\ov k}$. In other words, consider the following subset of $\frt_{\ov k}\times\frt_{\ov k}$:
\begin{equation*}
\G:=\{(y,\Fr(y))|y\in \frt_{\ov k}, \chi(y)=\chi(\Fr(y))\}=\frt_{\ov k}\times_{\frc(\ov k)}\frt_{\ov k}\cap \G(\Fr|\frt_{\ov k}).
\end{equation*}
This is the union of all $\frt_{w}=\G(w|\frt_{\ov k})\cap \G(\Fr|\frt_{\ov k})$. Then an admissible collection of linear functions $\{\l_{w}\in\frt^{*}_{w}\}_{w\in W}$ is a function
\begin{equation*}
\l: \G\to k
\end{equation*}
that restricts to a $k$-linear function on each $\frt_{w}$.

\begin{lemma}\label{l:adm exist} Recall that we have $\d \in \Aut(W)$ (corresponding to the form of $G$). There is a constant $M$ depending only on the pair $(W,\d)$, such that whenever $q>M$, $W$-coregular admissible collections $(\l_{w})_{w\in W}$ exist.
\end{lemma}
\begin{proof} Replacing $G$ by $G/C_{G}$ we may assume $G$ is semisimple and $\frz=0$. Let $\cA$ be the $k$-vector space of admissible collections $(\l_{w})_{w\in W}$. More precisely, $\cA$ is the kernel of the linear map
\begin{equation}\label{def A}
\op_{w\in W}\frt_{w}^{*}\to \op_{\{w,w'\}\subset W} (\frt_{w}\cap \frt_{w'})^{*}
\end{equation}
where the second term is summing over subsets of $W$ of order $2$. Choose an arbitrary linear order $\prec$ on $W$. The map sends $\l_{w}\in \frt_{w}^{*}$ to the element whose $\{w_{1}\prec w_{2}\}$-coordinate is $\l_{w}|_{\frt_{w}\cap \frt_{w_{2}}}$ if $w=w_{1}$,  $-\l_{w}|_{\frt_{w_{1}}\cap \frt_{w}}$ if $w=w_{2}$, and zero otherwise.

View  $\cA$ as an affine scheme over $k$, and let $\cA^{\c}\subset \cA$ be the open  $k$-subscheme of $W$-coregular admissible collections. 

\begin{lemma}\label{l:adm surj} For each $w\in W$, the projection $p_{w}: \cA\to \frt^{*}_{w}$ is a surjective  linear map of $k$-vector spaces.
\end{lemma}
\begin{proof}
The map $p_{w}$ is $k$-linear. To show it is surjective, it suffices to show $p_{w,\ov k}: \cA\ot\ov k\to \frt^{*}_{w}\ot\ov k$ is surjective. Let $\cA'$ be the $\ov k$-vector space of functions on $\G':=\frt_{\ov k}\times_{\frc_{\ov k}}\frt_{\ov k}=\cup_{w\in W}\G(w|\frt_{\ov k})$ that are $\ov k$-linear on each graph $\G(w|\frt_{\ov k})$. Tensoring \eqref{def A} with $\ov k$, and noting that $(\frt_{w}\cap \frt_{w'})\ot_{k}\ov k\incl \G(w|\frt_{\ov k})\cap\G(w'|\frt_{\ov k})$, we see that $\cA'\subset \cA\ot\ov k$. The composition
\begin{equation*}
p'_{w}:\cA'\subset \cA\ot \ov k\xr{p_{w,\ov k}}\frt^{*}_{w}\ot \ov k\cong \frt^{*}_{\ov k}
\end{equation*}
is the restriction to $\G(w|\frt_{\ov k})$ (identified with $\frt_{\ov k}$ via the first projection). Linear maps on $\frt_{\ov k}\times \frt_{\ov k}$ give elements in $\cA'$, whose restriction to $\G(w|\frt_{\ov k})$ already gives a surjection onto $\frt^{*}_{\ov k}$. Therefore $p'_{w}$ is surjective, hence so is $p_{w,\ov k}$ and $p_{w}$.
\end{proof}

Let $\frM_{k}$ be the set of nonzero linear $k$-subspaces $\fra\subset \G=\cup_{w\in W}\frt_{w}$ that are $W$-relevant in  $\frt_{w}$ for some $w\in W$. Let $\frM^{\min}_{k}$ be the minimal elements in $\frM_{k}$ under inclusion, and let $\frM^{\min}_{k}(i)$ be the subset of those $\fra$ with dimension $i$. Let $M^{\min}_{k}(i)=\#\frM^{\min}_{k}(i)$. For each $\fra\in \frM_{k}$, let $\fra^{\bot}\subset \frA$ be the subspace of admissible collections that vanish on $\fra$. By Lemma \ref{l:adm surj}, we have $\codim_{\frA} \fra^{\bot}=\dim \fra$.  We have
\begin{equation*}
\cA-\cA^{\c}=\cup_{\fra\in \frM^{\min}_{k}}\fra^{\bot}.
\end{equation*}
Therefore if the right side has fewer $k$-points than $\cA$, $\cA^{\c}(k)\ne\vn$. In other words, if
\begin{equation}\label{Mq}
\sum_{i\ge1}M^{\min}_{k}(i)q^{-i}<1,
\end{equation}
then $W$-coregular admissible collection exists.

We now show that $M^{\min}_{k}(i)$ can be bounded using only $(W,\d)$ (and not $k$). Let $\Phi=\Phi(G_{\ov k}, T_{\ov k})$ be the abstract root system of $G$ with respect to the universal Cartan. Let $\wt\frM$ be the set of pairs $(\Phi', \om)$ where $\Phi'\subset \Phi$ is a rationally closed root subsystem, and $\om$ is an automorphism of $\Phi'$ induced by some element in $W\d$ stabilizing $\Phi'$. Let $\frM$ be the quotient of $\wt\frM$ modulo the equivalence relation $(\Phi',\om_{1})\sim (\Phi', \om_{2})$ if $\om_{1}$ and $\om_{2}$ are conjugate under the parabolic subgroup $W(\Phi')\subset W$.

We have a map $\wt m: \wt\frM\to \frM_{k}$ defined as follows. First let $\fra_{\ov k}=\frt_{\ov k}^{W(\Phi')}$. Let $w^{-1}\d\in W\d$ inducing $\om$ on $\Phi'$. Then $w^{-1}\Fr$ ($\Fr$ denotes the automorphism on $\frt_{\ov k}$) stabilizes $\fra_{\ov k}$ and gives a $k$-form $\fra=\{x\in \fra_{\ov k}|\Fr(x)=wx\}$. We have $\fra\subset \frt_{w}$ is $W$-relevant. Different choices of $w$ lifting $\om$ give the same subspace in $\G$. 

Moreover, if $\om$ and $\om'$ are conjugate under $W(\Phi')$, their liftings $w^{-1}\Fr$ and $w'^{-1}\Fr$ have the same action on $\fra_{\ov k}$. This shows that the map $\wt m$ factors through a map
\begin{equation*}
m: \frM\to \frM_{k}.
\end{equation*}
When $p$ is \goodp for $G$, by Lemma \ref{l:W rel}(3), all relevant subspaces in $\frt_{\ov k}$ are fixed points of parabolic subgroups of $W$. From this we see that $m$ is a bijection. Now let $\frM^{\min}$ be minimal elements in $\frM$, and $\frM^{\min}(i)\subset \frM^{\min}$ be those where $\Phi'$ has rank $r-i$ (where $r$ is the rank of $\Phi$). Let $M^{\min}(i)=\#\frM^{\min}(i)$. Then  $M^{\min}(i)\ge M^{\min}_{k}(i)$, at least when $p$ is not too small. This implies that each $M^{\min}_{k}(i)$ can be bounded using only $(W,\d)$, across all finite fields $k$. Therefore, there exists some $M>0$, depending only on $(W,\d)$, such that \eqref{Mq} holds whenever $q>M$, whence $W$-coregular admissible collection exists.
\end{proof}

\begin{remark}\label{r:q bound} Let $\{\a_{i}\}_{i\in I}$ be a set of simple roots for $\Phi$. From the discussion above we get the following formula for $M(i)=\#\frM(i)$:
\begin{equation*}
M(i)=\sum_{J\subset I, \#J=i}\#(W/W_{J})\#(N_{W\d}(W_{J})/W_{J}).
\end{equation*}
Therefore, when $p$ is \goodp for $G$, we may replace the $M^{\min}_{k}(i)$ in \eqref{Mq} by $M(i)$ above to obtain the bound for $q$.

When $G=\GL_{n}$, $\frM^{\min}(i)$ is non-empty only for $i=1$ or $(i+1)|n$. consists of two kinds of elements. For $d|n$, the pair $(\Phi', c)\in \frM^{\min}(d-1)$, where $\Phi'$ has type $(A_{n/d-1})^{d}$, and $c_{d}$ is a cyclic permutation of the $d$ irreducible factors of $\Phi'$. There are $\frac{n!}{d((n/d)!)^{d}}$ of them. In addition to these, the pairs $(\Phi', 1)\in \frM^{\min}(1)$, where $\Phi'\subset \Phi$ is a root system of type $A_{a-1}\times A_{b-1}$, $a+b=n$. There are $2^{n-1}$ of them. 
\end{remark}

\begin{exam} Consider the case $G=G_{2}$. Assume $p\ne 2,3$. 

In this case $W$ is a dihedral group of order $12$. 
Then for each reflection $r_{\a}$ that negates the roots $\pm\a$, $\frt_{r_{\a}}$ contains two lines $\ell^{+}_{\a}$ (which is perpendicular to $\a$) and $\ell^{-}_{\a}$ (parallel to $\a^{\vee}$) that are $W$-relevant, and are contained in $\frt_{1}$ and $\frt_{-1}$ (where $-1=w_{0}$ is the longest element in $W$) respectively. Then $\l_{r_{\a}}$ is uniquely determined by the requirement that $\l_{r_{\a}}|_{\ell^{+}_{\a}}=\l_{1}|_{\ell^{+}_{\a}}$ and $\l_{r_{\a}}|_{\ell^{-}_{\a}}=\l_{-1}|_{\ell^{-}_{\a}}$. By choosing $\l_{1}$ and $\l_{-1}$ to be $G$-coregular, and choosing $\l_{w}$ to be nonzero for the other rotations $w$, the resulting $\l_{w}$ is $G$-coregular for all $w\in W$.

From this we conclude that when $q>6$, there exists a $W$-coregular admissible collection for $G_{2}$. 
\end{exam}

\begin{prop}\label{p:adm kl} For any admissible collection of linear functions $\{\l_{w}\}_{w\in W}$ on $\frt_{w}$, we have an equality of functions on $\frg$
\begin{equation*}
\sum_{w\in W}\k_{\l_{w},w}=\chi^{*}\left(\sum_{w\in W}\chi_{w!}(\psi\c\l_{w})\right).
\end{equation*}
\end{prop}
\begin{proof}
Let $x\in \frg$ and let $\Xi\subset \frt_{\ov k}$ be the $W$-orbit that has the same image in $\frc$ as $x$. Let $\cB_{x}$ be the Springer fiber of $x$, then we have a decomposition
\begin{equation*}
\cB_{x}=\coprod_{y\in \Xi}\cB^{y}_{x}
\end{equation*}
according to the image of the projection $\wt\frg\to \frt$. Recall we have a natural action of  $W$ on $\cohog{*}{\cB_{x}}=\op_{y\in \Xi}\cohog{*}{\cB^{y}_{x}}$. The element $w\in W$ sends the summand $\cohog{*}{\cB^{y}_{x}}$ onto $\cohog{*}{\cB^{wy}_{x}}$. If $wy=\Fr(y)$, then $\Fr^{*}$ maps $\cohog{*}{\cB^{wy}_{x}}$ back to $\cohog{*}{\cB^{y}_{x}}$. We have an endomorphism $\Fr^{*}\c w\in \End(\cohog{*}{\cB^{y}_{x}})$ when $y\in \Xi\cap \frt_{w}$. We have
\begin{equation*}
\k_{\l_{w},w}(x)=\sum_{y\in \Xi\cap \frt_{w}}\Tr(\Fr^{*}\c w, \cohog{*}{\cB^{y}_{x}})\psi(\j{\l_{w},y}).
\end{equation*}
Summing over $w\in W$, we see
\begin{eqnarray}\label{k double sum}
\sum_{w\in W}\k_{\l_{w},w}(x)&=&\sum_{y\in \Xi}\sum_{w\in W, \Fr(y)=wy}\Tr(\Fr^{*}\c w, \cohog{*}{\cB^{y}_{x}})\psi(\j{\l_{w},y})\\
\notag&=&\sum_{y\in \Xi}\psi(\j{\l,y})\sum_{w\in W, \Fr(y)=wy}\Tr(\Fr^{*}\c w, \cohog{*}{\cB^{y}_{x}})
\end{eqnarray}
Here we use $\j{\l,y}$ to denote the common value of $\j{\l_{w}, y}$ whenever $w\in W$ is such that $\Fr(y)=wy$ (this is independent of the choice of $w$ by the admissibility of $\l$). 

Now consider the sum
\begin{equation*}
\sum_{w\in W, \Fr(y)=wy}\Tr(\Fr^{*}\c w, \cohog{*}{\cB^{y}_{x}}).
\end{equation*}
Since the sum is over a coset of $W_{y}$, we may choose any $w_{1}\in W$ such that $\Fr(y)=w_{1}y$, and rewrite it as
\begin{equation}\label{wy}
\sum_{w\in W_{y}}\Tr(\Fr^{*}\c w_{1}w, \cohog{*}{\cB^{y}_{x}})=\#W_{y}\cdot \Tr(\Fr^{*}\c w_{1}, \cohog{*}{\cB^{y}_{x}}^{W_{y}}).
\end{equation}
Observe that
\begin{equation*}
\cohog{*}{\cB^{y}_{x}}^{W_{y}}=\cohog{0}{\cB^{y}_{x}}\cong \Qlbar.
\end{equation*}
Here we use the fact that $\cB^{y}_{x}$ is connected. The action of $\Fr^{*}\c w_{1}$ on the above is by identity. Therefore the sum in \eqref{wy} is $\#W_{y}$. Plugging back into \eqref{k double sum}, we see that
\begin{equation*}
\sum_{w\in W}\k_{\l_{w},w}(x)=\sum_{y\in \Xi}\psi(\j{\l,y})\#W_{y}=\sum_{w\in W}\sum_{y\in \Xi\cap \frt_{w}}\psi(\j{\l_{w},y}).
\end{equation*}
The $w$-summand on the right side above is exactly the value of $\chi_{w!}(\psi\c\l_{w})$ at $x$. This proves the proposition.
\end{proof}

\section{Automorphism groups of $G$-bundles}\label{s:aut}

In this section, let $k$ be an algebraically closed field. Let $X$ be an algebraic stack over $k$ and $G$ a connected reductive group over $k$. The goal of this section is show that for a $G$-bundle $\cE$ over $X$, its automorphism group $\Aut(\cE)$ (assumed affine of finite type over $k$), is reduced under a mild assumption on $\ch(k)$. Moreover, we will show that there are strong constraints on the maximal diagonalizable subgroups (resp. maximal tori) of $\Aut(\cE)$.

\subsection{Automorphism group and evaluation maps}\label{ss:ev map}
Let $\cE$ be a $G$-bundle over $X$. Let $\Aut(\cE)$ denote the $k$-group scheme of automorphisms of $\cE$. We assume
\begin{equation*}
\mbox{$\Aut(\cE)$ is an affine $k$-group of finite type.} 
\end{equation*}
This is always the case if $X$ admits a surjective map from a proper $k$-scheme.

Denote the Lie algebra of $\Aut(\cE)$ by $\aut(\cE)$. Note that $\aut(\cE)=\G(X, \Ad(\cE))$. 

For $x\in X(k)$, we have the evaluation map 
\begin{equation*}
\ev_{x}: \Aut(\cE)\to G[\cE_{x}]:=\Aut_{G}(\cE_{x}).
\end{equation*}
After choosing a trivialization of $\cE_{x}$ we get an isomorphism $G[\cE_{x}]\cong G$. Different trivializations change the isomorphism by conjugation. Therefore $\ev_{x}$ can be viewed as an inner class of homomorphisms $\Aut(\cE)\to G$.

{\bf Convention.} By a {\em diagonalizable group} over $k$,  we mean a $k$-group scheme of the form $\Spec k[\L]$ for some abelian group $\L$.  It may not be reduced when $\ch(k)=p$ and $\L$ has $p$-torsion.  A finite type affine $k$-group scheme $H$ is diagonalizable if and only if every finite dimensional $k$-representation of $H$ is a direct sum of $1$-dimensional representations.

\begin{lemma}\label{l:ker ev unip}
Suppose $X$ is connected and $x\in X(k)$. Then $\ker(\ev_{x})$ does not contain any non-trivial diagonalizable subgroup. In particular, all $k$-points of $\ker(\ev_{x})$ are unipotent, hence the reduced structure of $\ker(\ev_{x})$ is unipotent.
\end{lemma}
\begin{proof}
Let $D\subset \ker(\ev_{x})\subset\Aut(\cE)$ be a diagonalizable subgroup. For each $V\in \Rep_{k}(G)$, the associated vector bundle $\cE(V)=\cE\twtimes{G}V$ carries an action of $D$, hence decomposes into eigen-summands $\cE(V)=\op_{\chi\in \xch(D)}\cE(V)[\chi]$. Since $D\subset \ker(\ev_{x})$, $\cE(V)[\chi]_{x}=0$ unless $\chi$ is the trivial character of $D$.  Thus for any nontrivial $\chi\in \xch(D)$, the vector bundle $\cE(V)[\chi]$ on $X$ must be zero because its fiber at $x$ is zero and $X$ is connected. We conclude that $D$ acts trivially on $\cE(V)$ for all $V\in \Rep_{k}(G)$, hence $D$ itself is trivial.

Any semisimple element of $\ker(\ev_{x})$ lies in a diagonalizable subgroup of $\ker(\ev_{x})$, therefore must be the identity. By Jordan decomposition, all elements in $\ker(\ev_{x})$ are unipotent.
\end{proof}

\subsection{Unipotent automorphisms and nilpotent endomorphisms}\label{ss:unip}
Let $A_{\cE}$ be the reduced structure of $\Aut(\cE)$. Consider the natural inclusion of Lie algebras 
\begin{equation*}
\io_{\cE}: \Lie A_{\cE}\incl \aut(\cE)
\end{equation*}

Let $\cU\subset G$ (resp. $\cN\subset \frg$) be the unipotent (resp. nilpotent) varieties. A {\em unipotent logarithm} for $G$ is a $G$-equivariant isomorphism $\rho: \cU\to \cN$ (which necessarily sends $1$ to $0$) such that its tangent map at $1\in \cU$ is the identity map between the tangent spaces $T_{1}\cU\subset T_{1}G=\frg$ and $T_{0}\cN\subset T_{0}\frg=\frg$, both as subspaces of $\frg$.

\begin{lemma}\label{l:nilp aut} Suppose $G$ has a unipotent logarithm $\rho:\cU\isom \cN$. Then the image of $\io_{\cE}$ contains all nilpotent elements of $\aut(\cE)$.
\end{lemma}
\begin{proof} Let $\cU\subset G$ (resp. $\cN\subset \frg$) be the unipotent (resp. nilpotent) varieties.  Let $\un\Aut(\cE)=\cE\times^{G,\Ad}G$ be the group scheme over $X$ of automorphisms of $\cE$. Then $\un\Aut(\cE)$ contains a subscheme $\cU_{\cE}=\cE\times^{G}\cU$, and the global section $\G(X,\cU_{\cE})$ is the the unipotent subscheme of $\Aut(\cE)=\G(X,\un\Aut(\cE))$. Similarly, we have $\cN_{\cE}:=\cE\times^{G}\cN\incl \Ad(\cE)=\cE\times^{G}\frg$, and $\G(X,\cN_{\cE})$ is the nilpotent subscheme of $\aut(\cE)=\G(X,\Ad(\cE))$.  

A unipotent logarithm $\rho: \cU\isom \cN$ induces an isomorphism $\rho_{\cE}: \cU_{\cE}\isom \cN_{\cE}$ over $X$, hence an isomorphism $\g_{\cE}: \G(X,\cU_{\cE})\isom \G(X,\cN_{\cE})$. Let $e\in \G(X,\cN_{\cE})$ be a nilpotent element of $\aut(\cE)$. Let $\e: \AA^{1}\to \G(X,\cN_{\cE})$ be the map $t\mapsto te$. Composing with $\g^{-1}_{\cE}$ we get a map $\wt \e: \AA^{1}\to \G(X,\cU_{\cE})\subset \Aut(\cE)$. The differential of $\wt\e$ at $0$ (note that $\wt\e(0)$ is the identity section) gives a global section of the relative tangent bundle $(T_{1}\cU)_{\cE}=\cE\times^{G}T_{1}\cU\subset \Ad(\cE)$, which via $d\g_{\cE}$ is identified with the subbundle $(T_{0}\cN)_{\cE}=\cE\times^{G}T_{0}\cN\subset \Ad(\cE)$. We conclude that $d\wt\e(0)=e$ as global sections of $\Ad(\cE)$, i.e., as elements in $\aut(\cE)$. 

Since $\AA^{1}$ is reduced, we have a factorization $\wt\e: \AA^{1}\to A_{\cE}$, hence $e=d\wt\e(0)\in \Lie A_{\cE}$, i.e., $e\in \Im(\Lie A_{\cE}\to \aut(\cE))$. 
\end{proof}

\begin{lemma}\label{l:unip log}
Suppose $p$ is good for $G$ and $\xch(ZG)$ has no $p$-torsion, then a unipotent logarithm for $G$ exists.
\end{lemma}
\begin{proof} The projection $G\to G/C_G$ induces an isomorphism on unipotent and nilpotent varieties. Moreover, by our assumption on $p$, the further projection $G/C_G\to G^\ad$ is a central isogeny of degree prime to $p$, hence also induces an isomorphism on unipotent and nilpotent varieties. 
We thus reduce to the case where $G$ is adjoint, therefore we may further reduce to the case where $G$ is adjoint and simple, and $p$ is good for $G$.

Recall the notion of a quasi-logarithm for $G$ introduced in \cite[Definition 1.8.1]{KV}, compare \cite[\S9.3]{BR}. It means a $G$-equivariant morphism $f: G\to \frg$ such that $f(1)=0$ and $df$ at $1\in G$ is the identity map. By the argument of \cite[Corollary 9.3.4]{BR}, any quasi-logarithm $f: G\to \frg$ restricts to a unipotent logarithm $\rho: \cU\isom \cN$. 

If $G$ is not of type $A$, and $p$ is good for $G$, then $G$ admits a quasi-logarithm by \cite[Proposition 9.3.3]{BR} (which is based on \cite[Ch I, \S5]{SS}). Therefore in this case $G$ admits a unipotent logarithm. If $G\cong \PGL_{n}$, then by the remarks in the first paragraph, $G$ has a unipotent logarithm if and only if $\GL_{n}$ has one. The map $\GL_{n}\to \gl_{n}$ given by $X\mapsto X-1$ is a quasi-logarithm. So $\GL_{n}$ and hence $\SL_{n}$ admit a unipotent logarithm.
\end{proof}

Combining Lemma \ref{l:nilp aut} and Lemma \ref{l:unip log}, we get:
\begin{cor}\label{c:nilp aut} Suppose $p$ is good for $G$ and $\xch(ZG)$ has no $p$-torsion, then the image of $\io_{\cE}$ contains all nilpotent elements of $\aut(\cE)$.
\end{cor}

\subsection{Saturated diagonalizable subgroups}

\begin{defn} A diagonalizable subgroup $D\subset G$ is called {\em saturated}, if the inclusion $D\incl Z(C_{G}(D))$ is an equality of subschemes of $G$.
\end{defn}

Note that $C_{G}(D)$ is always a (not necessarily connected) reductive group (in particular reduced): the fact it is reduced follows from \cite[XI,5.3]{SGA3}; to check its neutral component is reductive, we can repeatedly apply the fact that the reduced fixed point subgroup of a connected reductive group under a finite order automorphism is reductive.  In particular, $Z(C_{G}(D))$ is diagonalizable.

\begin{exam}\label{ex:sat}
\begin{enumerate}
\item The center $ZG$ is saturated, and any saturated diagonalizable subgroup of $G$ contains $ZG$. 
\item More generally, for any Levi subgroup $L\subset G$, $ZL$ is a saturated diagonalizable subgroup of $G$.
\item Let $V$ be a symplectic space over $k$ and $G=\Sp(V)$. Let $V=\op_{i=1}^{s}V_{i}$ be an orthogonal decomposition of $V$ into symplectic subspaces. Let $D\cong (\mu_{2})^{i}\subset G$ be the subgroup that acts on $V_{i}$ by the $i$-th projection to $\mu_{2}$ via scalar multiplication. Then $D$ is a saturated diagonalizable subgroup of $G$.
\item\label{PGLn mun sat} Let $G=\PGL_{n}$ with diagonal torus $T_{0}=(\Gm)^{n}/\D(\Gm)$. Consider the embedding $i: \mu_{n}\incl T_{0}$ that is the image of the map $\wt i: \mu_{n}\to (\Gm)^{n}$, $i(z)=(1,z,z^{2},\cdots, z^{n-1})$. Then $C_{G}(i(\mu_{n}))$ contains $T_{0}$ as the neutral component, and $C_{G}(i(\mu_{n}))/T_{0}\cong \ZZ/n\ZZ$ is identified with the subgroup of $S_{n}$ generated by a cyclic permutation. It turns out that $i(\mu_{n})$ is saturated.
\item Again let $G=\PGL_{n}$. Let $D$ be the subgroup generated by the $i(\mu_{n})$ in the previous example, and the cyclic permutation matrix corresponding to $(12\cdots n)$. When $p\nmid n$, $D$ is a diagonalizable subgroup of $G$, and it is saturated. Indeed $C_{G}(D)=D$.
\end{enumerate}
\end{exam}

\begin{lemma}\label{l:max torus D relevant}
Let $D\subset G$ be a saturated diagonalizable subgroup, and $A\subset D$ be the maximal torus of $D$. Then $A$ is a $G$-relevant torus of $G$ (see Definition \ref{def:rel}).
\end{lemma}
\begin{proof}
Let $L=C_{G}(A)$, a Levi subgroup of $G$. Since $A\subset D$, we have $C_{G}(D)\subset C_{G}(A)=L$. Taking centers we get $ZL\subset D$. We thus have inclusions
\begin{equation*}
A\subset C_{L}\subset ZL\subset D.
\end{equation*}
Since $A$ is the maximal torus in $D$, we must have $A=C_{L}$, which implies $A$ is a $G$-relevant torus in $G$.  
\end{proof}

\begin{prop}\label{p:dc reduced}
If $p$ is \goodp for $G$ (see \S\ref{sss:goodp}), then all saturated diagonalizable subgroups of $G$ are reduced.
\end{prop}
\begin{proof}
Let $D$ be a saturated diagonalizable subgroup of $G$. We would like to show that $\xch(D)$ has no $p$-torsion. Suppose $\xch(D)$ has a surjection onto $\ZZ/p\ZZ$. It corresponds to an embedding $i: \mu_{p}\incl D$. It suffices to show that $i$ can be extended to a homomorphism $\Gm\to D$. By Lemma \ref{l:centralizer Levi}(2), $H=C_{G}(i(\mu_{p}))$ is a Levi subgroup of $G$. Since $i(\mu_{p})\subset D$, we have $C_{G}(D)\subset H$. Taking centers we get $ZH\subset D$. 

We claim that
\begin{equation}\label{XZH no p}
\mbox{$\xch(ZH)$ has no $p$-torsion.}
\end{equation}
Indeed, since $H$ is a Levi subgroup of $G$, $\xch(ZH)_{\tors}\incl \xch(ZG)_{\tors}$ by \eqref{xz tors}. Since $\xch(ZG)$ has no $p$-torsion by assumption, neither does $\xch(ZH)$. This shows \eqref{XZH no p}.

By \eqref{XZH no p}, $\mu_{p}\incl  ZH$ extends to a homomorphism $\Gm\to ZH$. Composing with the inclusion into $D$, it gives the extension of $i$ to $\Gm\to ZH\subset D$.
\end{proof}

\begin{exam} When $G=\GL_{n}$, all saturated diagonalizable subgroups are tori, hence reduced in all characteristics.

Let us see that the assumptions on $p$ are necessary in type $A$ in extreme cases. When $G=\SL_{n}$, in order for $ZG$ to be reduced, we need $p\nmid n$. When $G=\PGL_{n}$, we see from Example \ref{ex:sat}\eqref{PGLn mun sat} that it has a saturated diagonalizable subgroup isomorphic to $\mu_{n}$. For it to be reduced we also need $p\nmid n$. 
\end{exam}

The next two lemmas are preparatory results for the next subsection.

\begin{lemma}\label{l:max diag exists}
Let $H$ be an affine $k$-group of finite type. Then each diagonal subgroup of $H$ is contained in a maximal diagonalizable subgroup of $H$.
\end{lemma}
\begin{proof}
Let $D_{1}\subset D_{2}\subset \cdots \subset  D_{n}\subset\cdots $ be an increasing sequence of diagonalizable subgroup of $H$. Let $D$ be the Zariski closure of $\cup_{n\ge1} D_{n}$. We claim that $D$ is also diagonalizable. This would prove the existence of a maximal diagonalizable subgroup of $H$.

To show $D$ is diagonalizable, since $D$ is an affine $k$-group of finite type, it suffices to show that any finite-dimensional $k$-representation $V$ of $D$ decomposes into a direct sum of $1$-dimensional representations. Now each $D_{n}$ gives a decomposition $V=\oplus_{\chi\in \xch(D_{n})}V(n;\chi)$. Under the surjection $\xch(D_{n})\surj \xch(D_{n-1})$, the decomposition under $D_{n}$ refines that of $D_{n-1}$. Let $\L=\varprojlim_{n}\xch(D_{n})$. Then we get a decomposition $V=\op_{\chi\in \L}V(\chi)$ such that $D_{n}$ acts on $V(\chi)$ under the character $\chi_{n}$ (image of $\chi$ under $\L\to \xch(D_{n})$). Decompose each $V(\chi)$ into lines $V(\chi)_{1}\op V(\chi)_{2}\op\cdots $ arbitrarily. Then each $V(\chi)_{i}$ is stable under each $D_{n}$, hence stable under $D$ because $\cup D_{n}$ is dense in $D$. This shows that $V$ is a direct sum of $1$-dimensional representations of $D$. This finishes the proof of the claim.
\end{proof}

\begin{lemma}\label{l:hom from D isolated}
    Let $D$ be a diagonalizable group over $k$. Let $H$ be a smooth affine $k$-group. Consider the affine scheme $\Hom(D,H)$ of homomorphisms from $D$ to $H$. Then each connected component of $\Hom(D,H)$ is a single orbit under the conjugation action of $H^\c$.
\end{lemma}
\begin{proof} Fix a homomorphism $\ph: D\to H$, we only need to show that the infinitesimal action $a_{\frh,\ph}: \frh=\Lie H\to T_\ph\Hom(D,H)$ is surjective. Note that $T_\ph\Hom(D,H)$ can be identified with the vector space $Z^1_\ph(D,\frh)$ of $\ph$-twisted 1-cocycles valued in $\frh$, i.e.,   regular maps $\th: D\to \frh$ such that
\begin{equation}\label{ph cocycle}
    \th(xy)=\th(x)+\Ad(\ph(x))\th(y)
\end{equation}
holds, where both sides are viewed as regular maps $D\times D\to \frh$. The map $a_{\frh, \ph}:\frh\to Z^1_\ph(D,\frh)$ sends $v\in \frh$ to the function $x\mapsto v-\Ad(\ph(x))v$.

For any representation $V$ of $D$, we can similarly form the cocycles $Z^1(D,V)$ and the coboundary map $a_V: V\to Z^1(D,V)$. Then $a_{\frh,\ph}$ is the special case of this construction when $V=\frh$ with the action of $D$ via $\Ad\c \ph$. We shall prove that $a_V$ is surjective for all $V\in \Rep(D)$. The formation of $a_V$ is clearly additive in $V$, so to show that $a_V$ is surjective for all $V$, it suffices to treat the case $V=k(\chi)$ is one-dimensional, corresponding to a character $\chi$ of $D$. In this case, $\th\in Z^1(D,k(\chi))$ is the space of regular functions $\th\in k[D]=k[\xch(D)]$ such that
\begin{equation*}
    \D\th=\th\ot1+\chi\ot \th.
\end{equation*}
Expanding $\th$ using the basis $\xch(D)$, we see that $\th$ is a $k$-multiple of $1-\chi$, which is equal to $a_{k(\chi)}(1)$. This shows that $a_{V}$ is surjective for all 1-dimensional $V$, hence for all $V$ and in particular for $(\frh, \Ad\c\ph)$.
\end{proof}

\subsection{Semisimple automorphisms and endomorphisms}
In this subsection we assume $X$ is {\em connected} and $\cE$ is a $G$-bundle on $X$.

\begin{lemma}\label{l:A conj}
Let $D\subset \Aut(\cE)$ be a diagonalizable subgroup. Then for any two points $x_{1}, x_{2}\in X(k)$, the maps $\ev_{x_{1}}: D\to G$ and $\ev_{x_{2}}:D\to G$ are conjugate by an element in $G$.
\end{lemma}
\begin{proof}
By the connectedness of $X$, we reduce to the following: there exists a connected affine $k$-scheme $S$ with a map $S\to X$, such that $\cE_S$ is trivial, and both $x_1$ and $x_2$ lift to $S$. Fix a trivialization of $\cE_S$. Then the action of $D$ on $\cE_S$ becomes a homomorphism $D\times S\to \Aut(\cE_S)\cong G\times S$ of constant group schemes over $S$. This homomorphism is a map $\ph: S\to \Hom(D,G)$, so that $\ph(x_i)$ is conjugate to $\ev_{x_i}$. Now by Lemma \ref{l:hom from D isolated}, each connected component of $\Hom(D,G)$ is a single orbit under the conjugation action of $G$. Since $S$ is connected, $\ph(x_1)$ and $\ph(x_2)$ lie in the same $G$-orbit of $\Hom(D,G)$, i.e., $\ev_{x_1}$ and $\ev_{x_2}$ are $G$-conjugate.
\end{proof}

\begin{lemma}\label{l:red to centralizer}
Let $D\subset \Aut(\cE)$ be a diagonalizable subgroup. Let $\ev_{x}(D)=D_{x}$. Then $\cE$ admits a reduction of structure group to $H=C_{G}(D_{x})$.
\end{lemma}
\begin{proof} 
The total space $\cE$ represents the functor that sends an affine $k$-scheme $S$ to the set of pairs $(y,\t)$ where $y\in X(S)$ and $\t: G\times S\isom \cE_{y}$ is a trivialization of the pullback $G$-bundle $\cE_{y}$ over $S$ (where $G\times S$ is viewed as a right $G$-bundle under right multiplication). We define the following closed subscheme $\cF$ of $\cE$ by imposing the condition that $\t$ should be $D$-equivariant, where $D$ acts on $G\times S$ by left multiplication via the map $\ev_{x}: D\to D_{x}$ (which is an isomorphism by Lemma \ref{l:ker ev unip}), and acts on $\cE_{y}$ because $D\subset \Aut(\cE)$. There is an action of $H=C_{G}(D_{x})$ on $\cF$ by pre-composing on $\t$. 

We claim that $\cF$ is indeed an $H$-bundle under this action. It is easy to see that $H$ acts simply-transitively on non-empty fibers of $\cF\to X$. So it remains to see that $\cF_{y}$ is non-empty for all geometric points $y\in X$.  After choosing a trivialization of $\t_{y}: G\isom \cE_{y}$, the action of $D$ on $\cE_{y}\cong G$ is by left multiplication via the map $\ev_{y}: D\to G$. By Lemma \ref{l:A conj}, $\ev_{y}: D\to G$ is conjugate to $\ev_{x}:D\to G$ in $G$, therefore $\cF$ has a non-empty fiber at every geometric point $y\in X$. 

From the definition we have a $G$-equivariant map $\cF\twtimes{H}G\to\cE$, which is easily seen to be an isomorphism by checking fiberwise. In other words, $\cF$ is an $H$-reduction of $\cE$.
\end{proof}

\begin{cor}\label{c:ev image dc} Let $D\subset \Aut(\cE)$ be a 
maximal diagonalizable subgroup, which exists by Lemma \ref{l:max diag exists}. Then for any $x\in X(k)$, $\ev_{x}(D)\subset G$ (well-defined up to $G$-conjugacy) is a saturated diagonalizable subgroup of $G$.
\end{cor}
\begin{proof}
Let $D_{x}=\ev_{x}(D)$, then $\ev_{x}$ restricts to an isomorphism $D\isom D_{x}$ by Lemma \ref{l:ker ev unip}. By Lemma \ref{l:red to centralizer}, $\cE$ admits a reduction of structure group to $H=C_{G}(D_{x})$, i.e., there is an $H$-bundle $\cF$ over $X$ with an isomorphism of $G$-bundles $\cF\times^{H}G\isom \cE$. We have
\begin{equation*}
D_{x}\subset ZH\subset \Aut(\cF)\subset \Aut(\cE).
\end{equation*}
The image of these inclusions is $D$. By the maximality of $D_{x}$ among diagonalizable subgroups of $\Aut(\cE)$, we must have $D_{x}=ZH$, which means $D_{x}$ is saturated.
\end{proof}

\begin{cor}\label{c:diag reduced} If $p$ is \goodp for $G$, then all diagonalizable subgroups of $\Aut(\cE)$ are reduced. 
\end{cor}
\begin{proof}
Let $D\subset \Aut(\cE)$ be diagonalizable. To show the reducedness we may assume $D$ is a maximal diagonalizable subgroup of $\Aut(\cE)$ (by Lemma \ref{l:max diag exists}). In this case, $\ev_{x}$ restricts to an isomorphism $D\isom D_{x}:=\ev_{x}(D)$ by Lemma \ref{l:ker ev unip}. By Corollary \ref{c:ev image dc}, $D_{x}$ is saturated in $G$. Under our assumption of $p$, by Proposition \ref{p:dc reduced}, $D_{x}$ is reduced. Therefore $D$ is also reduced.
\end{proof}

\begin{cor}\label{c:ss aut} If $p$ is \goodp for $G$, then the image of $\io_{\cE}$ contains all semisimple elements of $\aut(\cE)$.
\end{cor}
\begin{proof}
For any affine $k$-group scheme $H$ of finite type and $\g\in \frh=\Lie H$ semisimple, we will define a diagonalizable subgroup $D_{\g}\subset H$ canonically attached to $\g$ such that $\g\in \Lie D_{\g}$. 

For any $V\in \Rep_{k}(H)$, $\g$ induces a semisimple endomorphism on $V$, hence an eigenspace decomposition $V=\op_{\l\in k}V[\l]$. This decomposition is compatible with tensor products with respect to the addition of eigenvalues. Let $\L_{\g}$ be the additive subgroup of $k$ generated by all eigenvalues that appear in the $\g$-action on $V$ for some $V\in \Rep_{k}(H)$. Let $D_{\g}=\Spec k[\L_{\g}]$ be the diagonalizable group with $\xch(D_{\g})=\L_{\g}$. By construction we have a symmetric monoidal functor $\rho: \Rep_{k}(H)\to \Rep_{k}(D_{\g})=\Vect_{k,\L_{\g}}$ (finite-dimensional $k$-vector spaces graded by $\L_{\g}$), compatible with the forgetful functors to $\Vect_{k}$. By Tannakian formalism, $\rho$ determines a canonical homomorphism $\rho_{*}: D_{\g}\to H$. Since all characters of $D_{\g}$ appear in the image of $\rho$, $\rho_{*}$ is an embedding. Note that we have a canonical isomorphism of $k$-vector spaces
\begin{equation*}
\Lie D_{\g}\cong \Hom_{\ZZ}(\L_{\g}, k).
\end{equation*}
By construction, $\L_{\g}$ comes with an embedding $\L_{\g}\subset k$, which corresponds to an element $\d\in \Lie D_{\g}$ under the above isomorphism. By definition, the action of $\d$ on each $V\in \Rep_{k}(H)$ has the same eigenspace decomposition as $\g$, therefore $\d=\g$. In other words, $\g\in\Lie D_{\g}$.

Applying the above discussions to $H=\Aut(\cE)$, we conclude that any semisimple $\g\in \aut(\cE)$ is contained in the Lie algebra of a diagonalizable subgroup $D_{\g}\subset \Aut(\cE)$. By Corollary \ref{c:diag reduced}, $D_{\g}$ is reduced, hence $\Lie D_{\g}\subset \Im(\io_{\cE})$. In particular $\g\in \Im(\io_{\cE})$.
\end{proof}

\subsection{Consequences on the structure of $\Aut(\cE)$}

\begin{cor}\label{c:aut red} Suppose $p=\ch(k)$ is \goodp for the connected reductive group $G$. Let $X$ be a $k$-stack and $\cE$ a $G$-bundle over $X$ such that the automorphism group $\Aut(\cE)$ is affine of finite type over $k$. Then $\Aut(\cE)$ is reduced.
\end{cor}
\begin{proof} If $X=\coprod_{i} X_{i}$ is the decomposition of $X$ into connected components, then $\Aut(\cE)=\prod_{i}\Aut(\cE|_{X_{i}})$. Therefore it suffices to show each $\Aut(\cE|_{X_{i}})$ is reduced. We thus reduce to the case where $X$ is connected.

When $X$ is connected, combining Corollary \ref{c:nilp aut} and Corollary \ref{c:ss aut}, we see that $\io_{\cE}$ is an isomorphism. This implies that $A_{\cE}=\Aut(\cE)$, hence $\Aut(\cE)$ is reduced.
\end{proof}

\begin{remark} When $G=\GL_{n}$, all primes are pretty good for $G$. A $G$-bundle $\cE$ corresponds to  a vector bundle $\cV$ over $X$. Assume $X$ is a proper scheme over $k$. Then $\Aut(\cV)$ is an open subscheme in the finite dimensional $k$-vector space $\End(\cV)$ viewed as an affine space. This shows directly that $\Aut(\cE)=\Aut(\cV)$ is reduced.
\end{remark}

\begin{cor}\label{c:ev image rel} Under the same assumptions of Corollary \ref{c:aut red}, for any $x\in X(k)$, the image of the evaluation map $\ev_{x}:\Aut(\cE)\to G$ (well-defined up to conjugacy) is a $G$-relevant subgroup.

\end{cor}
\begin{proof} 
The question only concerns the connected component of $X$ containing $x$, therefore we may assume $X$ is connected. 

Fix a trivialization of $\cE_{x}$, then we have the evaluation map $\ev_{x}: \Aut(\cE)\to G$. Let $A_{x}\subset \Im(\ev_{x})$ be a maximal torus. Since $\Aut(\cE)$ is reduced, so is $\ker(\ev_{x})$ and $\ev^{-1}_{x}(A_{x})$. Then $\ev^{-1}_{x}(A_{x})$ is a connected solvable group with unipotent radical $\ker(\ev_{x})$ by Lemma \ref{l:ker ev unip}. Pick any maximal torus $A\subset \ev^{-1}_{x}(A_{x})$, then $\ev_{x}$ restricts to an isomorphism $A\isom A_{x}$ (see \cite[Theorem 10.6(4)]{Borel}). Since $\ker(\ev_{x})$ is unipotent, $A$ is a maximal torus in $\Aut(\cE)$. 

Let $D\subset \Aut(\cE)$ be a maximal diagonalizable subgroup containing $A$, which exists by Lemma \ref{l:max diag exists}. By Corollary \ref{c:ev image dc}, $D_{x}=\ev_{x}(D)$ is saturated. Now $A_{x}\subset D_{x}$ is the maximal torus, and we conclude by Lemma \ref{l:max torus D relevant} that $A_{x}$ is a $G$-relevant torus in $G$. Since $A_{x}$ is a maximal torus in $\Im(\ev_{x})$ (which is a reduced subgroup of $G$), $\Im(\ev_{x})$ is $G$-relevant. 
\end{proof}

We now generalize the previous results to $H$-bundles where $H$ is not necessarily a reductive group.

\begin{theorem}
Let $H$ be a connected algebraic group over $k$. Let $\cE$ be an $H$-bundle over $X$ such that $\Aut(\cE)$ is affine of finite type over $k$. Assume further that
\begin{enumerate}
    \item $H$ admits a unipotent logarithm $\cU_H\isom \cN_H$.
    \item $p=\ch(k)$ is \goodp for the reductive quotient of $H$.
\end{enumerate}
Then $\Aut(\cE)$ is reduced.
\end{theorem}
The proof follows the same lines as Corollary \ref{c:aut red}. We omit details.

A particular case where $H$ admits a unipotent logarithm is when $H=P$ is a parabolic subgroup of a connected reductive group $G$, and $p=\ch(k)$ is good for $G$. Indeed,  Lemma \ref{l:unip log} guarantees that $G$ has a unipotent logarithm $\rho: \cU_G\isom \cN_G$. Now $P$ is the attracting subscheme of $G$ with respect to a one-parameter subgroup $\l: \Gm\to G$, and $\frp=\Lie P$ is the attracting subspace of $\frg$ with respect to the same $\l$. Because $\rho$ is $\Ad(G)$-equivariant,  it restricts to an isomorphism $P\cap \cU_G\isom \frp\cap \cN_G$ between the respective attracting subschemes of $\cU_G$ and $\cN_G$, which provides a unipotent logarithm for $P$.  We thus conclude:

\begin{cor}\label{c:aut red par}
    Let $P\subset G$ be a parabolic subgroup of a connected reductive group over $k$. Let $\cE$ be a $P$-bundle over $X$ such that $\Aut(\cE)$ is affine of finite type over $k$. Assume further that $p=\ch(k)$ is \goodp for $G$. Then $\Aut(\cE)$ is reduced.
\end{cor}

\subsection{Application to the reducedness of centralizers} As a byproduct, we obtain the following theorem about the reducedness of fixed point subgroups of a connected reductive group $G$.

\begin{cor}\label{c:Gfix} Assume $p=\ch(k)$ is \goodp for the connected reductive group $G$ over $k$. Let $H$ be a $k$-group scheme acting on $G$ by group automorphisms. Then the fixed point subgroup $G^{H}$ is reduced (hence smooth over $k$) and $G$-relevant.
\end{cor}
\begin{proof}
Consider the stack $X=\pt/H$ and the $G$-bundle $\cE=G/H$ over it. Then $\Aut(\cE)=G^{H}$ (clearly affine and of finite type over $k$). Corollary \ref{c:aut red} and \ref{c:ev image rel} applied to this situation then yields the claim.
\end{proof}

\begin{remark} When $H\subset G$ acts on $G$ by conjugation,  the above theorem asserts that the centralizer group scheme $C_{G}(H)$ is reduced. This is proved by Herpel in \cite[Theorem 1.1]{Herpel}. Our argument is different.
\end{remark}

\section{Counting absolutely indecomposable $G$-bundles}
Let $k=\FF_{q}$ be a finite field with characteristic $p$. Let $G$ be a connected reductive group over $k$. Let $X$ be a smooth,  projective and geometrically connected curve over $k$. 

We will prove a version of Theorem \ref{th:intro} for reductive groups in this section. Note that some immediate results in this section do not assume $X$ is smooth.

\subsection{Absolutely indecomposable $G$-bundles}
Let $|X|$ be the set of closed points of $X$.  Let $\Bun_{G}$ denote the moduli stack of $G$-bundles on $X$.

\begin{defn}\label{d:ai} A $k$-point $\cE\in \Bun_{G}(k)$ is called {\em absolutely indecomposable} if the quotient $\Aut(\cE)/C_{G}$ does not contain a nontrivial torus. Here recall that $C_{G}$ is the maximal torus in the center of $G$.

We denote the full subgroupoid of absolutely indecomposable objects of $\Bun_{G}(k)$ by $\Bun_{G}(k)^{\AI}$.
\end{defn}

\begin{warning} The groupoid $\Bun_{G}(k)^{\AI}$ is not in general the groupoid of $k$-points of a substack of $\Bun_{G}$.
\end{warning}

Consider the set of connected components $\pi_{0}(\Bun_{G})(k)$ of $\Bun_{G}$ that contain a $k$-point. For $\g\in \pi_{0}(\Bun_{G})(k)$ we denote the corresponding connected component by $\Bun_{G}^{\g}$ (which is defined over $k$).  Let $\Bun^{\g}_{G}(k)^{\AI}\subset \Bun_{G}^{\g}(k)$ be the full subgroupoid of absolutely indecomposable objects.

\begin{exam} When $G=\GL_{n}$, absolutely indecomposable $k$-points in $\Bun_{G}\cong \Bun_{n}$ in the sense of Definition \ref{d:ai} coincide with the usual notion of absolutely indecomposable vector bundles: they are vector bundles $\cV$ of rank $n$ over $X$ such that its pullback to $X_{\ov k}$ does not decompose into the direct sum of two subbundles of positive rank.

This can be seen as follows. Let $\ov\Aut(\cV)=\Aut(\cV)/\Gm$ (modulo scalar action) as an algebraic group over $k$. We need to show that $\cV_{\ov k}:=\cV|_{X_{\ov k}}$ is decomposable if and only if $\ov\Aut(\cV)$ contains a nontrivial torus. 

Suppose $\cV_{\ov k}\cong \cV'\op \cV''$, with $\rk\cV'>0$ and $\rk\cV''>0$,  then we have a subtorus $\l: \Gm\to \ov\Aut(\cV)_{\ov k}$ that acts on $\cV'$ by scaling and acts on $\cV''$ by the identity. Thereforer $\ov\Aut(\cV)^{\c}$ is not unipotent.

Conversely, if $\ov\Aut(\cV)$ contains a nontrivial torus, then there is a nontrivial homomorphism $\l: \Gm\to \ov\Aut(\cV)^{\c}$. This gives an action of $\Gm$ on each fiber of $\cV_{\ov k}$, and decomposes $\cV$ into a direct sum of weight spaces $\cV_{\ov k}=\op_{n\in \ZZ}\cV_{\ov k}(n)$.  Since $\l$ is nontrivial, this decomposition has more than one nonzero summands, hence $\cV_{\ov k}$ is decomposable.
\end{exam}

\begin{exam} For general $G$ and an absolutely indecomposable $k$-point $\cE$ of $\Bun_{G}$, $\Aut(\cE)$ may not be connected. For example, let $G=\Sp_{2n}$ and $p\ne 2$. Let $n=a+b$ with $a,b>0$. Let $\cE'$ and $\cE''$ be symplectic vector bundles of rank $2a$ and $2b$ respectively, that are stable as vector bundles. Then $\cE=\cE'\op\cE''$ is absolutely indecomposable, although it looks like decomposable. We have $\Aut(\cE)\cong\mu_{2}\times\mu_{2}$, with the first (resp. second) factor of $\mu_{2}$ acting on $\cE'$ (resp. $\cE''$) by scaling.
\end{exam}

\begin{defn}\label{def:AIk}
Let $\AI_{G,X}(k)$ be the set of isomorphism classes of absolutely indecomposable $G$-bundles over $X_{\ov k}$ whose isomorphism class is defined over $k$. 

For each $\g\in \pi_{0}(\Bun_{G})(k)$, let $\AI^{\g}_{G,X}(k)\subset \AI_{G,X}(k)$ be the subset of those absolutely indecomposable $G$-bundles over $X_{\ov k}$ that lie in the component $\Bun^{\g}_{G}$.
\end{defn}

\begin{remark} There is a natural map $\Bun_{G}(k)^{\AI}\to \AI_{G,X}(k)$, sending an absolutely indecomposable bundle $\cE$ to its pullback to $X_{\ov k}$. This map is surjective but in general not injective on isomorphism classes.  The following lemma makes this remark more precise.
\end{remark}

\begin{lemma}\label{lem:AIG} For each $\g\in \pi_{0}(\Bun_{G})(k)$, $\AI^{\g}_{G,X}(k)$ is a finite set with cardinality
\begin{equation}\label{AIG}
\#\AI^{\g}_{G,X}(k)=\sum_{\cE\in \Bun^{\g}_{G}(k)^{\AI}}\frac{1}{\#\pi_{0}(\Aut(\cE))(k)}.
\end{equation}
Here $\pi_{0}(\Aut(\cE))$ is the finite $k$-group scheme of connected components of $\Aut(\cE)$, and $\pi_{0}(\Aut(\cE))(k)$ is its $k$-points (those connected components of $\Aut(\cE)$ that contain $k$-points of $\Aut(\cE)$).
\end{lemma}
\begin{proof}
The finiteness of $\AI^{\g}_{G,X}(k)$ follows from the fact that $\Bun_{G}^{\g}(k)^{\AI}$ is contained in a finite union of Harder-Narasimhan strata, each of finite type.

Consider the base change map $\b: \Bun_{G}(k)^{\AI}\to \AI_{G,X}(k)$. Let $\cE_{0}\in\Bun_{G}(k)^{\AI}$, $\ov\cE=\b(\cE_{0})$ and consider the fiber $F_{\ov\cE}:=\b^{-1}(\ov\cE)$.  Then $F_{\ov\cE}$ is the groupoid of $k$-points of the stack $\pt/\Aut(\cE_{0})$ over $k$. The isomorphism classes of $F_{\ov\cE}$ are in bijection with $\cohog{1}{k, \Aut(\cE_{0})}\cong \cohog{1}{k,A_{0}}$, where $A_{0}=\pi_{0}(\Aut(\cE_{0}))$ as a finite algebraic group over $k$. For $\xi\in \cohog{1}{k,A_{0}}$ we get a point $\cE_{\xi}\in F_{\ov\cE}$, an inner form $A_{\xi}$ of $A_{0}$ over $k$, with an isomorphism $\pi_{0}(\Aut(\cE_{\xi}))\cong A_{\xi}$ over $k$. Thus we have
\begin{equation*}
\sum_{\cE\in F_{\ov\cE}}\frac{1}{\#\pi_{0}(\Aut(\cE))(k)}=\sum_{\xi\in \cohog{1}{k,A_{0}}}\frac{1}{\#A_{\xi}(k)}
\end{equation*}
The right side above is $1$ by the orbit-stabilizer relation applied to the Frobenius-twisted conjugation action of the finite group $A_{0}(\ov k)$ on itself. Thus
\begin{equation}\label{sum fiber Ebar}
\sum_{\cE\in F_{\ov\cE}}\frac{1}{\#\pi_{0}(\Aut(\cE))(k)}=1.
\end{equation}
Summing \eqref{sum fiber Ebar} over all points $\ov\cE\in \AI^{\g}_{G,X}(k)$, we get \eqref{AIG}.
\end{proof}

\begin{exam} Suppose $\ch(k)\ne 2$ and $G=\Sp_{2n}$. Let $(\cE, \om)$ be a symplectic vector bundle on $X$ of rank $2n$ whose underlying vector bundle $\cE$ is absolutely indecomposable. Let $c\in k^{\times}-(k^{\times})^{2}$ and consider the symplectic vector bundle $(\cE, c\om)$. Then $(\cE, \om)$ and $(\cE, c\om)$ are non-isomorphic over $k$ but become isomorphic over $\ov k$ (or the quadratic extension of $k$). They correspond to the same element in  $\AI_{G,X}(k)$, and $(\cE, c\om)$ is the only other point of $\Bun_{G}(k)$ that becomes isomorphic to $(\cE, \om)$ over $\ov k$. Note that $\pi_{0}(\Aut(\cE))\cong \mu_{2}$ so they each contribute $1/2$ to the right side of \eqref{AIG}, verifying the identity \eqref{AIG} in the current case.
\end{exam}

\subsection{Counting as an integral over Higgs moduli stack}\label{ss:counting Higgs}
From now on, we assume
\begin{itemize}
\item $p=\ch(k)$ is \goodp for $G$.
\item $X$ has a $k$-point $x$. 
\end{itemize}
For $\cE\in \Bun_{G}(k)$ and a trivialization of $\cE_{x}$, we get an evaluation map as in \S\ref{ss:ev map}
\begin{equation*}
\ev_{x}: \Aut(\cE)\to \Aut(\cE_{x})\cong G.
\end{equation*}
It induces a homomorphism of Lie algebras
\begin{equation}\label{Lie alg ev}
\frev_{x}: \aut(\cE)=\G(X,\Ad(\cE))\to \frg.
\end{equation}

\begin{lemma}\label{l:AI ess unip} A $G$-bundle $\cE$ over $X$ is absolutely indecomposable if and only if $\Im(\ev_{x})$ (which is a $G$-relevant subgroup of $G$) is essentially unipotent (see Definition \ref{d:ess unip}).
\end{lemma}
\begin{proof} By Lemma \ref{l:ker ev unip}, $\ker(\ev_{x})$ is a unipotent group (since it is reduced by Corollary \ref{c:aut red}).  Therefore $C_{G}$ (viewed as a subgroup of $\Aut(\cE)$) is the maximal torus in $\Aut(\cE)$ (i.e., $\cE$ is absolutely indecomposable) if and only if its image under $\ev_{x}$ (which is $C_{G}\subset G$) is the maximal torus in $\Im(\ev_{x})$ (i.e., $\Im(\ev_{x})$ is essentially unipotent).
\end{proof}

\sss{The Higgs moduli stack}
Let $\om$ be the sheaf of $1$-forms on $X$. Since $X$ has dimension $1$, a deformation calculation shows that $\Bun_{G}$ is a smooth Artin stack. Let $\cM_{G}=T^{*}\Bun_{G}$ be the (classical) cotangent stack of $\Bun_{G}$. It classifies pairs $(\cE, \ph)$ where $\cE$ is a $G$-bundle on $X$ and $\ph$ is a global section of $\Ad^{*}(\cE)\ot \om$. 

We consider a variant of $\cM_{G}$. Let $\cM_{G,x}$ be the moduli stack of pairs $(\cE, \ph)$ where $\cE$ is a $G$-bundle on $X$ and $\ph$ is a section of $\Ad^{*}(\cE)\ot \om(x)$.  Taking residue of $\ph$ at $x$ gives a map
\begin{equation*}
\res_{x}: \cM_{G,x}\to [\frg^{*}/G].
\end{equation*}
Let 
\begin{equation*}
b: \cM_{G,x}\to \Bun_{G}
\end{equation*}
be the forgetful map. For $\g\in \pi_{0}(\Bun_{G})(k)$, let $\cM^{\g}_{G,x}\subset \cM_{G,x}$ be the preimage of $\Bun_{G}^{\g}$ under $b$.

\sss{The function $f_{x}$}\label{sss:fx}
Let $\xi: \frc(k)\to \Qlbar$ be a selection function, then $\chi^{*}\xi$ is an $\Ad(G)$-invariant function on $\frg$.  Let $f_{x}=\FT(\chi^{*}\xi): \frg^{*}\to \Qlbar$ be the Fourier transform of $\chi^{*}\xi$. We use the following normalization of Fourier transform for functions $f$ on $\frg$:
\begin{equation}\label{FT}
\FT(f)(u)=\sum_{v\in \frg}f(v)\psi(\j{u,v}), \quad u\in \frg^{*}.
\end{equation}
Since $f_{x}$ is invariant under the coadjoint action of $G(k)$, we may view $f_{x}$ as a function on $[\frg^{*}/G](k)=\frg^{*}/G(k)$. 

\begin{lemma}\label{l:b!}
For the function $f_{x}$ on $\frg^{*}/G(k)$ constructed from a selection function $\xi$ on $\frc(k)$ as above, the function $b_{!}(\res_{x}^{*}f_{x}): \Bun_{G}(k)\to \Qlbar$ is supported on absolutely indecomposable $G$-bundles, where it takes the  value $q^{h^{0}(\Ad^{*}(\cE)\ot\om)+\dim G}.$
\end{lemma}
\begin{proof}
The value of $b_{!}(\res_{x}^{*}f_{x})$ at $\cE\in \Bun_{G}(k)$ is 
\begin{equation}\label{value of b!}
\sum_{\ph\in \upH^{0}(X,\Ad^{*}(\cE)\ot\om(x))}f_{x}(\res_{x}\ph).
\end{equation}
Choose a trivialization of $\cE_{x}$, we have a short exact sequence of coherent sheaves on $X$
\begin{equation*}
\xymatrix{0\ar[r] &  \Ad^{*}(\cE)\ot \om \ar[r] & \Ad^{*}(\cE)\ot \om(x) \ar[r] & i_{x*} \frg^{*}\ar[r] & 0}
\end{equation*}
where $i_{x}: \{x\}\incl X$ is the inclusion. This gives a long exact sequence on cohomology
\begin{equation*}
\xymatrix{0 \ar[r] & \upH^{0}(\Ad^{*}(\cE)\ot\om)\ar[r] & \upH^{0}(\Ad^{*}(\cE)\ot\om(x)) \ar[r]^-{\res_{x}} & \frg^{*} \ar[r]^-{\d} & \upH^{1}(X,\Ad^{*}(\cE)\ot \om ).}
\end{equation*}
Grothendieck-Serre duality identifies $\upH^{1}(X,\Ad^{*}(\cE)\ot \om )$ with the dual of $\upH^{0}(X, \Ad(\cE))=\aut(\cE)$, under which the connecting homomorphism $\d$ above becomes the adjoint of the evaluation map $\frev_{x}$ in \eqref{Lie alg ev}. Let $H_{x}\subset G$ be the image of $\ev_{x}$, and $\frh_{x}=\Lie(H_{x})\subset \frg$ be the image of $\frev_{x}$. The exactness of the above sequence at $\frg^{*}$ implies
\begin{equation*}
\Im(\res_{x})=\frh_{x}^{\bot}.
\end{equation*}
Therefore \eqref{value of b!} can be written as
\begin{equation}\label{value of b! 1}
q^{h^{0}(\Ad^{*}(\cE)\ot\om)}\j{\one_{\frh^{\bot}_{x}}, f_{x}}_{\frg^{*}}.
\end{equation}
By the Plancherel formula (note our normalization of Fourier transform \eqref{FT}), we have
\begin{equation}\label{value of b! 2}
\j{\one_{\frh^{\bot}_{x}}, f_{x}}_{\frg^{*}}=\j{\one_{\frh^{\bot}_{x}}, \FT(\chi^{*}\xi)}_{\frg^{*}}=q^{\dim \frh_{x}^{\bot}}\j{\one_{\frh_{x}}, \chi^{*}\xi}_{\frg}.
\end{equation}

By Corollary \ref{c:ev image rel}, $H_{x}\subset G$ is a relevant subgroup, and hence $\frh_{x}\subset \frg$ is a relevant subalgebra. By Proposition \ref{p:sel fun}, the right side is zero unless $H_{x}$ is essentially unipotent, which is equivalent to that $\cE$ is absolutely indecomposable by Lemma \ref{l:AI ess unip}. When $H^{\c}_{x}$ is unipotent, $\frh_{x}$ consists of nilpotent elements, on which $\chi^{*}\xi$ takes value $1$. Therefore, the right side of \eqref{value of b! 2} is 
\begin{equation*}
q^{\dim \frh_{x}^{\bot}}q^{\dim \frh_{x}}=q^{\dim G}.
\end{equation*}
Combined with \eqref{value of b! 1} we see that the value of $b_{!}(\res_{x}^{*}f_{x})$ at an absolutely indecomposable $\cE\in \Bun_{G}(k)$ is $q^{h^{0}(\Ad^{*}(\cE)\ot \om)+\dim G}$.
\end{proof}

Let $\g\in \pi_{0}(\Bun_{G})(k)$. The above lemma then allows us to make sense of the a priori infinite summation
\begin{equation*}
\int_{\cM^{\g}_{G,x}(k)}\res_{x}^{*}f_{x}.
\end{equation*}
Indeed, we interpret the above as
\begin{equation}\label{int b! fin}
\int_{\Bun^{\g}_{G}(k)}b_{!}\res_{x}^{*}f_{x},
\end{equation}
and by Lemma \ref{l:b!}, the integrand is supported on the subgroupoid $\Bun^{\g}_{G}(k)^{{\AI}}$, which is finite by Lemma \ref{lem:AIG}.

\begin{cor}\label{c:G-stable counting of AI} For a selection function $\xi$ on $\frc(k)$ and the resulting function $f_{x}=\FT(\chi^{*}\xi )$ on $\frg^{*}/G(k)$, we have
\begin{equation*}
\# \AI^{\g}_{G,X}(k)=\frac{\#C_{G}(k)}{q^{g\dim G+\dim C_{G}}}\int_{\cM^{\g}_{G,x}(k)}\res_{x}^{*}f_{x}
\end{equation*}
for each $\g\in \pi_{0}(\Bun_{G})(k)$ (under the interpretation of the right side as \eqref{int b! fin}). 
\end{cor}
\begin{proof}
Let $\cE\in \Bun^{\g}_{G}(k)^{\AI}$. Since $\Aut(\cE)^{\c}/C_{G}$ is unipotent, we have
\begin{equation*}
\#\Aut(\cE)^{\c}(k)=\#C_{G}(k)\cdot q^{\dim \Aut(\cE)-\dim C_{G}}=\frac{\#C_{G}(k)}{q^{\dim C_{G}}}q^{h^{0}(\Ad(\cE))}=\frac{\#C_{G}(k)}{q^{\dim C_{G}}}q^{h^{1}(\Ad^{*}(\cE)\ot\om)}
\end{equation*}
where the last equality uses Serre duality.  By Lemma \ref{l:b!}, the contribution of $\cE$ to $\int_{\cM^{\g}_{G,x}(k)}\res_{x}^{*}f_{x}$ is 
\begin{eqnarray}
&&\frac{q^{h^{0}(\Ad^{*}(\cE)\ot \om)+\dim G}}{\#\Aut(\cE)(k)}=\frac{q^{\dim C_{G}}}{\#C_{G}(k)}\frac{q^{h^{0}(\Ad^{*}(\cE)\ot\om)-h^{1}(\Ad^{*}(\cE)\ot\om)+\dim G}}{\#\pi_{0}(\Aut(\cE))(k)}\\
&=&\frac{q^{\dim C_{G}}}{\#C_{G}(k)}\frac{q^{(g-1)\dim G+\dim G}}{\#\pi_{0}(\Aut(\cE))(k)}=\frac{q^{\dim C_{G}}}{\#C_{G}(k)}\frac{q^{g\dim G}}{\#\pi_{0}(\Aut(\cE))(k)}.
\end{eqnarray}
The result now follows from Lemma \ref{lem:AIG}.
\end{proof}

Next we plug-in the selection function constructed in Proposition \ref{p:cons sel fn from t} from a $W$-coregular admissible collection $\l=(\l_w)_{w\in W}$, where $\l_w\in \frt_w^*$. 

For any $c\in \frc^*(k)$, let $\cM_{G,x}(c)\subset \cM_{G,x}$ be the fiber over $c$ of the map
\begin{equation*}
    \cM_{G,x}\xr{\res_x} [\frg^*/G]\to \frc^*.
\end{equation*}
Let $\cM^\g_{G,x}(c)=\cM^\g_{G,x}\cap \cM_{G,x}(c)$.

Note that the $W$-coregular locus in $\frt^*$ is a $W$-invariant open subset. It thus descends to an open subset $(\frc^*)^{\coreg}\subset \frc^*$. We call points in $(\frc^*)^{\coreg}$ {\em $W$-coregular}. We shall prove in Lemma \ref{l:coreg lam stable}(2) that for any $W$-coregular $c\in \frc^*(\ov k)$, $\cM_{G,x}^\g(c)$ is of finite type over $\ov k$. In particular, if $c\in \frc^*(k)$ is $W$-coregular, $\cM_{G,x}^\g(c)$ is of finite type over $k$, and it has finitely many $k$-points.

\begin{cor}\label{c:AI M Spr} Assume $q=\#k$ is large enough with respect to $W$ so that a $W$-coregular admissible collection $\l=\{\l_{w}\}_{w\in W}$ exists (see Lemma \ref{l:adm exist}). Let $c_w\in \frc^*(k)$ be the image of $\l_w$. 

Then for each $\g\in \pi_{0}(\Bun_{G})(k)$ we have
\begin{equation}\label{AI M Spr}
\#\AI^{\g}_{G,X}(k)= \frac{\#C_{G}(k)}{q^{(g-1)\dim G+\dim \cB+\dim C_{G}}}\cdot \frac{1}{\#W}\sum_{w\in W} \sgn(w)\int_{(\cE,\ph)\in \cM^\g_{G,x}(c_w)(k)}\Tr(\Fr^*\c w, \cohog{*}{\cB^{\l_w}_{\res_x\ph}}).
\end{equation}
Here the action of $\Fr^*\c w$ on part of the Springer fiber $\cB^{\l_w}_{\res_x\ph}$ is explained in \S\ref{sss:tw Fr Spr}. 
\end{cor}
\begin{proof}
We can rewrite the right side of \eqref{AI M Spr} as
\begin{equation*}
    \frac{\#C_{G}(k)}{q^{(g-1)\dim G+\dim \cB+\dim C_{G}}}\int_{(\cE,\ph)\in \cM^\g_{G,x}(k)}\left(\frac{1}{\#W} \sum_{w\in W} \sgn(w)\Tr(\Fr^*\c w, \cohog{*}{\cB^{\l_w}_{\res_x\ph}})\right)
\end{equation*}
with the understanding that the integrand vanishes outside  finitely many points of $\cM^\g_{G,x}$.

    Let $\xi=\frac{1}{\#W}\sum_{w\in W}\chi_{w!}(\psi\c\l_{w})$ be the selection function on $\frc(k)$ corresponding to the $W$-coregular admissible collection $\l$. Let $f_{x} = \FT(\chi^{*}\xi)$. By Corollary \ref{c:G-stable counting of AI}, we only need to show
\begin{equation}\label{int res sum w}
\int_{\cM^{\g}_{G,x}(k)}\res^{*}_{x}f_{x} = \frac{q^{\dim G-\dim\cB}}{\#W} \int_{(\cE,\ph)\in \cM^\g_{G,x}(k)}\left(\frac{1}{\#W} \sum_{w\in W} \sgn(w)\Tr(\Fr^*\c w, \cohog{*}{\cB^{\l_w}_{\res_x\ph}})\right).
\end{equation}
Therefore we only need to match the integrands on both sides.

We use Proposition \ref{p:adm kl} to rewrite the function $\chi^{*}\xi$ as
\begin{equation*}
\frac{1}{\#W}\sum_{w\in W}\k_{\l_{w},w}.
\end{equation*}
By Corollary \ref{cor:FTkappa}, the value of $\FT(\k_{\l_{w},w})$ at $v\in \frg^{*}$ is
\begin{equation*}
q^{\dim G-\dim\cB}\sgn(w)\Tr(\Fr^{*}w, \cohog{*}{\cB^{\l_{w}}_{v}}).
\end{equation*}

Therefore, we get 
\begin{equation}\label{fx}
f_{x}(v)=\frac{q^{\dim G-\dim\cB}}{\#W}\sum_{w\in W}\sgn(w)\Tr(\Fr^{*}w, \cohog{*}{\cB^{\l_{w}}_{v}}).
\end{equation}
This is exactly saying that the integrands of both sides of \eqref{int res sum w} are the same.
\end{proof}

We can further simplify the right side of \eqref{AI M Spr} under an additional assumption on $\l$.

\begin{defn}\label{def:reg}
\begin{enumerate}
    \item We say $\l\in \frt^*_{\ov k}$ is {\em $W$-regular} if its stabilizer under $W$ is trivial. 
    \item Let $\frt_0\subset \frg$ be a Cartan subalgebra. An element $\l\in \frt_0^*$ is called {\em $W$-regular} if it is regular when viewed as an element in $\frt^*_{\ov k}$ under some (equivalently any) identification $\io: \frt_{0,\ov k}\cong \frt_{\ov k}$ obtained by choosing a Borel over $\ov k$ containing $\frt_{0,\ov k}$ (see \S\ref{sss:univ Cartan}).
    
    \item An admissible collection $\l=\{\l_w\}_{w\in W}$ is called {\em $W$-regular} if each $\l_w$ is. 
\end{enumerate}
\end{defn}

The same argument as in Lemma \ref{l:adm exist} shows that when $q$ is sufficiently large with respect to $(W,\d)$, there exist admissible collections $(\l_{w})_{w\in W}$ that are both $W$-regular and $W$-coregular. Again let $c_w\in \frc^*(k)$ be the image of $\l_w$.

\begin{cor}\label{c:AI Mc}
    Assume $q=\#k$ is large enough with respect to $W$ so that a $W$-coregular and $W$-regular admissible collection $\l=\{\l_{w}\}_{w\in W}$ exists. Let $c_w\in \frc^*(k)$ be the image of $\l_w$ under the map $\frt_w^*\to \frc^*$. Then for each $\g\in \pi_{0}(\Bun_{G})(k)$ we have
\begin{equation}\label{AI Mc}
\#\AI^{\g}_{G,X}(k)= \frac{\#C_{G}(k)}{q^{(g-1)\dim G+\dim \cB+\dim C_{G}}\#W}\sum_{w\in W} \sgn(w)\int_{\cM^\g_{G,x}(c_w)(k)}1.
\end{equation}
\end{cor}
\begin{proof}
    We only need to observe that when $\l$ is $W$-regular, for any $v\in \frg^*$ with the same image as $\l$ in $\frc^*$, the Springer fiber $\cB^{\l}_v$ is a point. Therefore the integrand in \eqref{AI M Spr} is the constant function $1$.
\end{proof}

\subsection{Stable parabolic Higgs bundles}\label{ss:stable Higgs}

We continue to make the assumptions in the beginning of \S\ref{ss:counting Higgs}.

\sss{Parabolic $G$-bundles and parabolic Higgs bundles} Let $\Bun_{G}(\bI_{x})$ be the moduli stack of triples $(\cE, \cE^{B}_{x})$ where $\cE$ is a $G$-bundle on $X$ and $\cE^{B}_{x}$ a $B$-reduction at $x$. Let $\cM_{G}(\bI_{x})$ be the moduli stack of triples $(\cE, \cE^{B}_{x}, \ph)$ where $\cE$ is a $G$-bundle on $X$, $\cE^{B}_{x}$ a $B$-reduction at $x$, and $\ph\in \G(X, \Ad^{*}(\cE)\ot \om(x))$ is such that $\res_{x}(\ph)\in \cE_{x}\twtimes{G}\frg^{*}$ lies in $\cE^{B}_{x}\twtimes{B}\frn^{\bot}$. Here $\frn$ is the nilpotent radical of $\frb=\Lie B$. 

We have a map
\begin{equation*}
\phi_x: \cM_{G}(\bI_{x})\to \frt^{*}
\end{equation*}
taking $(\cE,\cE^{B}_{x}, \ph)$ to the image of $\res_{x}(\ph)$ under the map $\cE^{B}_{x}\twtimes{B}\frn^{\bot}\to \cE^{T}_{x}\twtimes{T}\frt^{*}=\frt^{*}$. Here $\cE^{T}_{x}$ is the $T$-bundle induced from $\cE^{B}_{x}$. We have a Cartesian diagram
\begin{equation}\label{Mpar cart}
    \xymatrix{ 
\cM_{G}(\bI_{x}) \ar[r]\ar[d]^{\wt\pi_{\cM}}	& [\wt{\frg}^{*}/ G] \ar[d]^{\wt\pi'}  \\
\cM_{G,x}\times_{\frc^*}\frt^* 	\ar[r]^{\res_{x}}	& [\frg^{*}\times_{\frc^*}\frt^* / G ]
}
\end{equation}
Here $\wt\pi_\cM$ sends $(\cE,\cE^B_x,\ph)$ to $(\cE, \ph, \phi_x(\ph))$.  By proper base change, we have a canonical isomorphism
\begin{equation*}
\wt\pi_{\cM!}\Qlbar\cong \res_x^*\cK'_\univ.
\end{equation*}
In particular, $\wt\pi_{\cM!}\Qlbar\in D^b_c(\cM_{G,x}\times_{\frc^*}\frt^*)$ carries a canonical $W$-equivariant structure compatible with the $W$-action on $\frt^*$.

For $\l\in \frt^*(\ov k)$, let
\begin{equation*}
\cM_{G}(\bI_{x}; \l)=\phi_{x}^{-1}(\l)
\end{equation*}
be the fiber of $\phi_x$ over $\l$, now a stack over $\ov k$. The $W$-equivariant structure on $\wt\pi_{\cM,!}\Qlbar$ induces an isomorphism
\begin{equation}\label{w eq M}
    w: \cohoc{*}{\cM_{G}(\bI_{x}; \l)}\cong\cohoc{*}{\cM_{G}(\bI_{x}; w\l)} 
\end{equation}
If $\Fr(\l)=w\l$, the same construction as in \S\ref{sss:tw Fr Spr} gives an automorphism
\begin{equation}\label{tw Fr M}
    \Fr^*\c w: \cohoc{*}{\cM_{G}(\bI_{x}; \l)}\xr{w}\cohoc{*}{\cM_{G}(\bI_{x}; w\l)} \xr{\Fr^*}\cohoc{*}{\cM_{G}(\bI_{x}; \l)}.
\end{equation}

\sss{Stability condition for parabolic $G$-Higgs bundles}
We review the notion of stability for parabolic Higgs bundles. The definition here is a slight variant of \cite[\S5-\S6]{CL}; see also \cite[Section V]{F} for a different but equivalent approach.

For a parabolic subgroup $P\subset G$ with unipotent radical $N_{P}$, $\frn_{P}=\Lie N_{P}$ and Levi quotient $L_{P}=P/N_{P}$,   a $P$-reduction of a Higgs field $(\cE, \ph)$ (possibly with poles) is a $P$-reduction $\cE_{P}$ of $\cE$ such that at the generic point of $X$, $\ph$ lies in $\cE_{P}\twtimes{P}\frn_{P}^{\bot}$. Let $\cE_{L_{P}}=\cE_{P}/N_{P}$ be the induced $L_{P}$-bundle over $X$. Let $T_{P}=L_{P}/L_{P}^{\der}$ be the quotient torus. 

Recall we have a degree map
\begin{equation*}
\deg_{L_{P}}: \Bun_{L_{P}}\to \pi_{0}(\Bun_{L_{P}})=\pi_{1}(L_{P})\cong \xcoch(T)/\ZZ\Phi^{\vee}(L_{P})\to\xcoch(T_{P}). 
\end{equation*}
Let $(\cE,\cE^{B}_{x}, \ph)\in \cM_{G}(\bI_{x})$. For a $P$-reduction of $(\cE, \ph)$, we have a $P$-reduction $\cE_{P,x}=\cE_{P}|_{x}$ of the fiber $\cE_{x}$. The relative position of the pair $(\cE_{P,x}, \cE^{B}_{x})$ define a point $pos(\cE_{P,x}, \cE^{B}_{x})$ in $G\bs (\cB_{P}\times \cB)$ (here $\cB_{P}$ is the partial flag variety $G/P$). For $(P_{1}, B_{1})\in \cB_{P}\times \cB$, $P_{1}\cap B_{1}$ is a Borel subgroup in the Levi $L_{P_{1}}$, and the reductive quotient of $P_{1}\cap B_{1}$ is canonically identified with the universal Cartan $T$. Therefore we get a well-defined map $T\to T_{P_{1}}\cong T_{P}$ (the latter is a canonical isomorphism for any two conjugate parabolics $P$ and $P_{1}$) from a point $(P_{1},B_{1})\in \cB_{P}\times \cB$. This construction gives the same map for each $G$-orbit on  $\cB_{P}\times \cB$, hence the relative position $pos(\cE_{P,x}, \cE^{B}_{x})\in G\bs (\cB_{P}\times \cB)$ defines a map $T\to T_P$, and induces a map on cocharacter groups
\begin{equation*}
    \pi(\cE_{P,x}, \cE^{B}_{x}): \xcoch(T)\to \xcoch(T_{P}).
\end{equation*}

A stability condition $\th$ for $\cM_{G}(\bI_{x})$ is a point $\th\in \xcoch(T)_{\RR}$. For $(\cE,\cE^{B}_{x}, \ph)\in \cM_{G}(\bI_{x})$ and a $P$-reduction $\cE_P$ of $(\cE, \ph)$, let
\begin{equation*}
\th(\cE_{P,x}, \cE^{B}_{x})\in \xcoch(T_{P})_{\RR}
\end{equation*}
be the image of $\th$ under the projection $\pi(\cE_{P,x}, \cE^{B}_{x})$ (tensoring with $\RR$). Let $\deg_{\th}(\cE_{P})$ be the sum
\begin{equation*}
\deg_{\th}(\cE_{P})=\deg_{L_{P}}(\cE_{L_{P}})+\th(\cE_{P,x}, \cE^{B}_{x})\in \xcoch(T_{P})_{\RR}.
\end{equation*}

The universal Cartan is equipped with a based root system (see \S\ref{sss:univ Cartan})), hence $\xcoch(T)_{\RR}$ has a well-defined dominant Weyl chamber $\frC$ and a fundamental alcove $A$ (fundamental domain of the affine Weyl group action). We choose $\th$ in the interior of $A$.  

We define the {\em obtuse cone} $\frC'_P\subset \xcoch(T_P)_\RR$ to be the $\RR_\ge0$-span of the images of $\a^\vee$, where $\a$ runs over roots that appear in $\frn_P$. Let $\frC'^{\c}_P$ be the interior of $\frC'_P$.

A point $(\cE,\cE^{B}_{x}, \ph)\in \cM_{G}(\bI_{x})$ is called {\em $\th$-stable} if for any $P$-reduction $\cE_{P}$ of $(\cE, \ph)$, where $P$ is a proper parabolic subgroup of $G$, we have $\deg_{\th}(\cE_{P})\in -\frC_{P}'^{\c}$. It is called {\em $\th$-semistable} if for any such $P$-reduction $\cE_{P}$ of $(\cE, \ph)$, we have $\deg_{\th}(\cE_{P})\in -\frC'_{P}$. 

For a generic $\th$ in the fundamental alcove $A$,  $\th$-semistable is equivalent to $\th$-stable. Let $\cM_G(\bI_{x})^{\tst}$ be the open substack of $\cM_G(\bI_{x})$ consisting of $\th$-stable points. 

\sss{Geometric properties}
Below we will consider the stack
\begin{equation*}
    \ov\cM_G(\bI_{x}):=\cM_{G}(\bI_{x})/\BB(C_G)
\end{equation*}
obtained by removing the obvious copy of $C_{G}$ from the automorphism groups of all points. Similarly we define
\begin{equation*}
    \ov\cM^\g_G(\bI_{x})^{\tst}:=\cM^\g_{G}(\bI_{x})^{\tst}/\BB(C_G).
\end{equation*}

To arrive at a better formula that eliminates the choice of $\l=\{\l_w\}$, we need some geometric properties of the stack $\ov\cM_G(\bI_x)^{\tst}$. Unfortunately for these geometric properties to be available we need $p$ to be sufficiently large.

We have the Hitchin map
\begin{equation*}
h:  \cM_{G}(\bI_{x})\to \cA_{G}(\bI_{x}),
\end{equation*}
where $\cA_G(\bI_{x})$ is the affine space parametrizing the invariant  polynomials of Higgs fields with a simple pole at $x$. The map $h$ clearly factors through $\ov\cM_G(\bI_{x})$. We denote by
\begin{equation}\label{om Hitchin map}
    \ov h^{\g, \tst}:  \ov\cM^\g_{G}(\bI_{x})^{\tst}/\BB(C_G)\to \cA_{G}(\bI_{x})
\end{equation}
the restriction of the Hitchin map to the $\th$-stable locus in the $\g$-component.

In the Appendix we will prove the following.
\begin{theorem}\label{th:geom M}
    \begin{enumerate}
        \item\label{M DM} Assume $p$ is larger than the Coxeter number of each simple factor of $G$.  The stack $\ov\cM_{G}(\bI_{x})^{\tst}$ is a Deligne-Mumford stack of finite type over $k$.
        \item\label{M sm} Assume $p$ is larger than the Coxeter number of each simple factor of $G$. The map $\phi^{\tst}_x: \cM_G(\bI_{x})^{\tst}\to \frt^*$ obtained as the restriction of $\phi_x$ is smooth. In particular, $\cM_G(\bI_{x})^{\tst}$ and $\ov\cM_{G}(\bI_{x})^{\tst}$ are smooth.
        \item\label{h sep} The Hitchin map $\ov h^{\g, \tst}$ is separated.
        \item\label{h proper} There exists an integer $N$ depending only on the Dynkin diagram of $G$ and the genus of $X$, such that whenever $p>N$, the Hitchin map $\ov h^{\g, \tst}$ is proper.
    \end{enumerate}
\end{theorem}

We believe the bound on the characteristic $p$ to ensure the properness of the Hitchin map can be made explicit and reasonably small. We do not know how to do this at this point.

\begin{lemma}\label{l:coreg lam stable}
Let $\g\in \pi_0(\Bun_G)(k)$.
\begin{enumerate}
    \item For any $W$-coregular $\l\in\frt^{*}(\ov k)$, we have $\cM_{G}(\bI_{x}; \l)\subset \cM_G(\bI_{x})^{\tst}_{\ov k}$ for any stability condition $\th\in \xcoch(T)_{\RR}$. In particular, $\cM^\g_{G}(\bI_{x}; \l)$ is of finite type over $\ov k$.
    \item For any $W$-coregular $c\in \frc^*(\ov k)$, $\cM^\g_{G,x}(c)$ is of finite type over $\ov k$.
\end{enumerate}
\end{lemma}
\begin{proof}
(1) We claim that if $(\cE,\cE^{B}_{x},\ph)\in \cM_{G}(\bI_{x};\l)$, there is no $P$-reduction of $(\cE,\ph)$ for any proper parabolic $P\subset G$, which then implies that $(\cE,\cE^{B}_{x},\ph)$ is stable for any stability condition. 

If not, let $\cE_P$ be a $P$-reduction of $(\cE,\ph)$. The reduction $\cE_{L}$ to the Levi quotient $L$ of $P$ carries a Higgs field $\ph_L\in \G(X, \Ad^*(\cE_{L})\ot\om_X(x))$ as the image of $\ph$ under the map $\cE_P\times^P \frn_P^\bot\to \cE_{L}\times^{L}\frl^*$. Taking residue at $x$ gives a point $\res_{L,x}\ph_L\in [\frl^*/L]$. The Borel reduction $\cE^B_x$ induces a Borel reduction of $\cE_L|_x$ and hence a lifting $\wt{\res_{L,x}\ph_L}\in [\wt\frl^*/L]$ of $\res_{L,x}\ph_L$, such that $\l=\phi_x(\ph)\in \frt^*$ is also the image of $\wt{\res_{L,x}\ph_L}$ under the map $\vep'_L: [\wt\frl^*/L]\to \frt^*$.

Let $\frz_L=\Lie(ZL)$. We have a commutative diagram
\begin{equation*}
    \xymatrix{[\wt\frl^*/L]\ar[d]^{\pi'_L}\ar[r]^{\vep'_L} & \frt^*\ar[d]^{r_{\frz_L}}\\
    [\frl^*/L]\ar[r]^{\z} & \frz^*_L
    }
\end{equation*}
Therefore
\begin{equation}\label{eq res}
   r_{\frz_L}(\l)= r_{\frz_L}\phi_x(\ph)= \z(\res_{L,x}\ph_L).
\end{equation}

On the other hand, via the $L$-invariant map $\z: \frl^*\to \frz_L^*$, $\ph_L$ induces a global section $\ov\ph_L$ of $\frz^*_L\ot\om_X(x)$ such that
\begin{equation}\label{eq res 2}
\z(\res_{L,x}\ph_L)=\res_x(\ov\ph_L)\in\frz_L^*. 
\end{equation}
By the residue theorem, $\ov\ph_L$ indeed lies in $\frz_{L_P}^*\ot\om_X$. This implies that the right side of \eqref{eq res 2} is zero, hence combined with \eqref{eq res}, we get $\l|_{\frz_L}=0$. Now $\frz_L$ is a $G$-relevant subspace of $\frt$, hence $W$-relevant by Lemma \ref{l:W rel}.  Since $P$ is properly contained in $G$, $\frz_L\ne \frz$. The vanishing of $\l|_{\frz_L}$ then contradicts the assumption that $\l$ is $W$-coregular.

(2) Choose a coregular $\l\in \frt^*(\ov k)$ that maps to $c$. By the diagram \eqref{Mpar cart}, we have a surjective map 
\begin{equation*}
    \cM_G^\g(\bI_x; \l)\to \cM^\g_{G,x}(c).
\end{equation*}
By (1), $\cM_G^\g(\bI_x; \l)$ is of finite type; we also know that $\cM^\g_{G,x}(c)$ is locally of finite type, therefore it is also of finite type over $\ov k$.
\end{proof}

\begin{cor}[of Corollary \ref{c:AI M Spr}]\label{c:AI Mpar} Assume $q=\#k$ is large enough with respect to $W$ so that a $W$-coregular admissible collection $\l=\{\l_{w}\}_{w\in W}$ exists (see Lemma \ref{l:adm exist}). Then for each $\g\in \pi_{0}(\Bun_{G})(k)$, the total cohomology $\cohoc{*}{\ov\cM^{\g}_{G}(\bI_{x}; \l_{w})}$ is finite-dimensional, and we have
\begin{equation}\label{AI Mpar}
\#\AI^{\g}_{G,X}(k)= \frac{1}{q^{\dim\Bun_{G}(\bI_{x})+\dim C_{G}}\cdot \#W} \sum_{w\in W} \sgn(w)\Tr(\Fr^*w, \cohoc{*}{\ov\cM^{\g}_{G}(\bI_{x}; \l_{w})}).
\end{equation}
Here, the action of $\Fr^*\c w$ on $\cohoc{*}{\ov\cM^{\g}_{G}(\bI_{x}; \l_{w})}$ is defined in \S\ref{tw Fr M}.
\end{cor}
\begin{proof} First, observe that 
\begin{equation*}
\dim\Bun_{G}(\bI_{x})=(g-1)\dim G+\dim \cB,
\end{equation*}
so the constant on the right side of \eqref{AI Mpar} matches the constant on the right side of \eqref{AI M Spr}. For each $w\in W$, by Lemma \ref{l:coreg lam stable} and Theorem \ref{th:geom M}\eqref{M DM}, $\ov\cM^{\g}_{G}(\bI_{x}; \l_{w})$ is a Deligne-Mumford stack of finite type over $\ov k$, therefore the total cohomology $\cohoc{*}{\ov\cM^{\g}_{G}(\bI_{x}; \l_{w})}$ is finite-dimensional, so that the trace of $\Fr^*\c w$ on its cohomology is defined. By the Cartesian diagram \eqref{Mpar cart}, the fiber of the natural map $\ov\cM_G^\g(\bI_x; \l_w)\to \ov\cM_{G,x}^\g(c_w):=\cM_{G,x}^\g(c_w)/\BB(C_G)$ over $(\cE,\ph)$ is $\cB^{\l_w}_{\res_x\ph}$ in the Springer fiber of $\res_x\ph\in \frg^*/G$. Since the actions of $\Fr^*\c w$ on the cohomology of $\ov\cM_G^\g(\bI_x; \l_w)$ and on the cohomology of $\cB^{\l_w}_{\res_x\ph}$ are defined in a compatible way, we have
\begin{equation*}
    \Tr(\Fr^*\c w, \cohoc{*}{\ov\cM^{\g}_{G}(\bI_{x}; \l_{w})})=\int_{(\cE,\ph)\in \ov\cM^\g_{G,x}(c_w)(k)}\Tr(\Fr^*\c w, \cohog{*}{\cB^{\l_w}_{\res_x\ph}}).
\end{equation*}
Summing over $w\in W$ we get the swapped double sum on the right side of \eqref{AI M Spr}; the factor $C_G(k)$ in \eqref{AI M Spr} disappears because it is the ratio between integrating over $\ov\cM^\g_{G,x}(c_w)(k)$ versus integrating over $\cM^\g_{G,x}(c_w)(k)$.
\end{proof}

\begin{prop} Assume $p$ is large enough so that Theorem \ref{th:geom M} holds. Let 
\begin{equation*}
    \ov\phi^{\g,\tst}_x: \ov\cM^{\g}_G(\bI_{x})^{\tst}\to \frt^{*}
\end{equation*}
be the map induced by $\phi_x$. Then for each $i\in \ZZ$, $\bR^{i}\phi^{\g,\tst}_{x,!}\Qlbar$ is geometrically a constant sheaf on $\frt^{*}$ that carries a canonical $W$-equivariant structure. After taking stalks at a $G$-coregular $\l$ and its $W$-translates, the $W$-equivariant structure coincides with the one defined in \eqref{w eq M}.
\end{prop}
\begin{proof} By Theorem \ref{th:geom M}\eqref{M sm}, 
the map $\ov\phi_x^{\g,\tst}$ is smooth. There is an action of $\Gm$ on $\ov\cM^{\g}_G(\bI_{x})^{\tst}$ by scaling the Higgs field. This action is contracting to the fixed point locus. Moreover, the fixed point locus is contained in the zero fiber of the Hitchin map $\ov h^{\g,\tst}$ in \eqref{om Hitchin map} (because the $\Gm$-action on $\cA_{G}(\bI_{x})$ is contracting to the origin). Since $\ov h^{\g,\tst}$ is proper by Theorem \ref{th:geom M}\eqref{h proper}, $\ov \cM^{\g}_G(\bI_{x})^{\tst,\Gm}$ is proper over $k$.

In this situation we may apply \cite[Theorem 10.2.2]{DGT} to conclude. The result in {\em loc.cit} is stated in the context of quasi-projective schemes but the argument works for stacks such as $\ov \cM^{\g}_G(\bI_{x})^{\tst}$. 
\end{proof}

\begin{cor}\label{cor:W-action} Assume $p$ is large enough so that Theorem \ref{th:geom M} holds. Then there is a canonical action of $W$ on $\cohoc{*}{\ov\cM^{\g}_G(\bI_{x}; 0)^{\tst}}$.

Moreover, when $\l_{w}\in \frt_{w}^{*}$ is $W$-coregular, we have
\begin{equation*}
\Tr(\Fr^*\c w, \cohoc{*}{\ov\cM^{\g}_G(\bI_{x};\l)})=\Tr(\Fr^*\c w, \cohoc{*}{\ov\cM^{\g}_G(\bI_{x};0)^{\tst}}).
\end{equation*}
\end{cor}

We denote the sign isotypic part of $\cohoc{*}{\ov\cM^{\g}_G(\bI_{x}; 0)^{\tst}}$ under the $W$-action by $\cohoc{*}{\ov\cM^{\g}_G(\bI_{x}; 0)^{\tst}}\j{\sgn}$. It still carries a Frobenius action.

\begin{cor}[of Corollaries \ref{c:AI Mpar} and \ref{cor:W-action}]\label{c:main}
Assume $p$ is large enough so that Theorem \ref{th:geom M} holds, and $q=\#k$ is large enough with respect to $W$ so that a coregular admissible collection $\l=\{\l_{w}\}_{w\in W}$ exists (see Lemma \ref{l:adm exist}). Then for each $\g\in \pi_{0}(\Bun_{G})(k)$ we have
\begin{eqnarray}
\notag\#\AI^{\g}_{G,X}(k)&=&q^{-\dim\Bun_{G}(\bI_{x})-\dim C_{G}}\Tr(\Fr,\cohoc{*}{\ov\cM^{\g}_G(\bI_{x}; 0)^{\tst}}\j{\sgn}).
\end{eqnarray}
\end{cor}

\sss{Global nilpotent cone}
Let 
\begin{equation*}
\cN^{\g}_G(\bI_{x}):=h^{\g, -1}(0)\subset \cM^{\g}_G(\bI_{x}; 0)
\end{equation*}
be the parabolic version of the global nilpotent cone.  Let $\cN^{\g}_G(\bI_{x})^{\tst}=\cN^{\g}_G(\bI_{x})\cap\cM^{\g}_G(\bI_{x}; 0)^{\tst}$ and
\begin{equation*}
\ov\cN^{\g}_G(\bI_{x})^{\tst}=\cN^{\g}_G(\bI_{x})^{\tst}/\BB(C_G).
\end{equation*}
If $p$ is large enough so that Theorem \ref{th:geom M} holds, then $\ov\cN^{\g}_G(\bI_{x})^{\tst}$ is a proper DM stack over $k$. Let $D=\dim \Bun_{G}(\bI_{x})$. Since $\cM_G(\bI_{x}; 0)=T^*\Bun_{G}(\bI_{x})$, it has dimension $2D$. Then
\begin{equation*}\dim \ov\cM^{\g}_G(\bI_{x}; 0)^{\tst}=2D+\dim C_G.
\end{equation*}
By the contracting principle under the $\Gm$-action by scaling, restriction gives an isomorphism:
\begin{eqnarray*}
\cohog{*}{\ov\cM^{\g}_G(\bI_{x}; 0)^{\tst}}\isom \cohog{*}{\ov\cN^{\g}_G(\bI_{x}; 0)^{\tst}}.
\end{eqnarray*}
Dualizing we get a Frobenius-equivariant quasi-isomorphism of complexes
\begin{equation*}
\homog{*}{\ov\cN^{\g}_G(\bI_{x}; 0)^{\tst}}\cong \cohoc{*}{\ov\cM^{\g}_G(\bI_{x}; 0)^{\tst}}[4D+2\dim C_G](2D+\dim C_G).
\end{equation*}
In particular, $\homog{*}{\ov\cN^{\g}_G(\bI_{x})^{\tst}}$ also carries an action of $W$, whose sign isotypic summand is denoted $\homog{*}{\ov\cN^{\g}_G(\bI_{x})^{\tst}}\j{\sgn}$.

\begin{cor}\label{c:main N} Under the same assumptions as in Corollary \ref{c:main}, for each $\g\in \pi_{0}(\Bun_{G})(k)$ we have
\begin{eqnarray}\label{AI homology N}
\#\AI^{\g}_{G,X}(k)=q^{\dim\Bun_{G}(\bI_{x})}\Tr(\Fr,\homog{*}{\ov\cN^{\g}_G(\bI_{x})^{\tst}}\j{\sgn}).
\end{eqnarray}
\end{cor}

\begin{remark} Let $\cM_{G, x, \nil}$ be the substack of $\cM_{G,x}$ where $\res_{x}(\ph)$ is nilpotent, i.e., lies in the nilpotent cone $\cN^{*}\subset \frg^{*}$. The forgetful map $\cM_{G}(\bI_{x})\to \cM_{G,x}$ restricts to a map $\Pi: \cM_{G}(\bI_{x};0)\to \cM_{G,x,\nil}$ that fits into a Cartesian diagram
\begin{equation*}
\xymatrix{\cM_{G}(\bI_{x};0) \ar[r]\ar[d]^{\Pi} &  [\wt\cN^{*}/G]\ar[d]^{\pi}\\
\cM_{G, x, \nil}\ar[r]^{\res_{x}} & [\cN^{*}/G].
}
\end{equation*}
In general, the open substack $\cM_{G}(\bI_{x};0) ^{\tst}$ is not the preimage of an substack of $\cM_{G, x, \nil}$ under $\Pi$. However, when $G=\GL_{n}$, $d\in \ZZ$ such that $(d,n)=1$, the component $\cM^{d}_{G}(\bI_{x};0) ^{\tst}$ of degree $d$ Higgs bundles is the preimage of $\cM^{d, \st}_{G,x,\nil}$ under $\Pi$. Springer theory then implies that the sign isotypic part of $\cohoc{*}{\ov\cM_G(\bI_{x}; 0)^{\tst}}$ is precisely $\cohoc{*}{\ov\cM_{G}^{d, \st}}$, the compactly supported cohomology of stable rank $n$ Higgs bundles of degree $d$. This recovers the result of Schiffmann \cite{Sch}, at least for $p$ sufficiently large.

For $G=\GL_n$ and arbitrary degree $\g\in \ZZ$, Corollary \ref{c:main} recovers the result of Dobrovolska, Ginzburg, and Travkin \cite[Theorem 1.4.5]{DGT}.
\end{remark}

\section{Indecomposable $G$-bundles on elliptic curves}\label{s:ell}
In this section, let $X$ be an elliptic curve over an {\em algebraically closed} field $k$ of any characteristic. Let $x_{0}\in X$ be the origin. Let $G$ be a connected reductive group over $k$. The aim is to give a parametrization of indecomposable $G$-bundles on $X$ in Lie-theoretic terms.

\subsection{Reduction to semistable $G$-bundles}
We recall the canonical parabolic reduction of a $G$-bundle $\cE$ over $X$ generalizing the Harder-Narasimhan filtration when $G=\GL_{n}$. Each $G$-bundle $\cE$ on $X$ carries a canonical reduction of $\cE$ to a $P$-bundle $\cE_{P}$ for some parabolic subgroup $P\subsetneq G$ characterized by the following conditions:
\begin{enumerate}
\item Let $L_{P}=P/N_{P}$ be the Levi quotient of $P$. Then the associated $L_{P}$-bundle $\cE_{L_{P}}=\cE_{P}/N_{P}$ is semistable. 
\item Let $A_{P}=C_{L_P}$ be the central torus of $L_{P}$, and let $T_{P}=L_{P}/L_{P}^{\der}$ be the quotient torus of $L_{P}$. The natural map $A_{P}\to T_{P}$ is an isogeny and induces an isomorphism $\xcoch(A_{P})_{\QQ}\cong \xcoch(T_{P})_{\QQ}$. The $L_{P}$-bundle $\cE_{L_{P}}$ further induces a $T_{P}$-bundle, and defines a degree $\deg(\cE_{L_{P}})\in\xcoch(T_{P})$. We may identify $\deg(\cE_{L_{P}})$ as an element in $\xcoch(A_{P})_{\QQ}$ using the prior identification. Then $\deg(\cE_{L_{P}})$ must satisfy the following condition
\begin{equation}\label{pos deg HN}
\j{\deg(\cE_{L_{P}}), \mu}>0 \mbox{ for any weight $\mu$ of $A_{P}$ on $\frn_{P}$}.
\end{equation}
\end{enumerate}

\begin{lemma}\label{l:ind sst}
Assume $p=\ch(k)$ is \goodp for $G$. Then any indecomposable $G$-bundle on an elliptic curve $X$ is semistable.
\end{lemma}
\begin{proof}
Let $\cE$ be a $G$-bundle on $X$, and let $\cE_{P}$ be the canonical reduction of $\cE$.  Fix a section $\io: L_{P}\incl P$ from the Levi factor $L_{P}$ of $P$. Let $\cF=\cE_{P}/N_{P}$ be the induced $L_{P}$-bundle. 

We claim that the natural map 
\begin{equation}\label{surj Aut}
    \Aut(\cE_P)^\c\to \Aut(\cF)^\c
\end{equation}
is surjective. Here $(-)^\c$ denotes neutral components. Indeed, by Corollary \ref{c:aut red par}, $\Aut(\cE_P)^\c$ is reduced, therefore it suffices to check the Lie algebra map $\aut(\cE_P)\to \aut(\cF)$ is surjective. Let $\frp=\Lie P, \frl_P=\Lie L_P$ and $\frn_P$ be the nilpotent radical of $\frp$, so that $0\to \frn_P\to \frp\to \frl_P\to 0$ is short exact. We thus have a short exact sequence of the associated vector bundles 
\begin{equation*}
    0\to \cE_P\times^P\frn_P\to \Ad(\cE_P)\to \Ad(\cF)\to 0
\end{equation*}
The corresponding long exact sequence reads
\begin{equation}\label{long on aut}
    \cdots\to \aut(\cE_P)\to \aut(\cF)\to \cohog{1}{X,\cE_P\times^P\frn_P}\to\cdots.
\end{equation}
The vector bundle $\cE\times^{P}\frn_{P}$ has a filtration where the associated graded are of the form $\cV_{i}:=\cF\times^{L_{P}}V_{i}$, where $V_i$ are irreducible subquotients of $\frn_{P}$ as a representation of $L_{P}$. By \cite[Proposition 2.5, Remark 2.6]{FGL}, since $\cF$ is a semistable $L_{P}$-bundle and $V_i$ is irreducible and in particular highest weight, $\cV_{i}$ is a semistable vector bundle (\cite[Prop. 2.5]{FGL} is stated for degree $0$ $L_P$-bundles, but according to \cite[Remark 2.6]{FGL}, the same argument in {\em loc.cit.} works for $L_P$-bundles of arbitrary degree as long as we take the associated vector bundle with respect to a highest weight representation of $L_P$). Moreover, \eqref{pos deg HN} for $\deg(\cF)=\deg(\cE_{L_P})$ implies that all slopes of $\cV_{i}$ are positive. This implies $\cohog{1}{X, \cV_{i}}=0$ for all $i$ (using the fact that $X$ has genus one). Therefore $\cohog{1}{X, \cE_P\times^{P}\frn_{P}}=0$, and the map $\aut(\cE_P)\to \aut(\cF)$ is surjective by \eqref{long on aut}. We conclude that \eqref{surj Aut} is surjective.

If $\cE$ is not semistable, $P\subset G$ is a proper parabolic and $\dim C_{L_P}>\dim C_G$. The surjectivity of \eqref{surj Aut} implies that $\Aut(\cE_P)$ contains a subquotient isomorphic to $C_{L_P}$. Therefore $\Aut(\cE_P)$, and hence $\Aut(\cE)$, contains a torus with strictly larger dimension than $C_G$, which means that $\cE$ is not indecomposable. Therefore any indecomposable $\cE$ has to be semistable. 
\end{proof}

\subsection{Tannakian description of semistable bundles on an elliptic curve (following Fratila, Gunningham and Penghui Li \cite{FGL})}
From now on, let $G$ be a {\em semisimple group} over $k$.

Let $\Vect^{\sst,0}_{X}$ be the $k$-linear tensor category of semistable vector bundles of degree $0$. It has a fiber functor $\om_{x_{0}}$ by taking the fiber at the origin $x_{0}$. 

Let $\Bun^{\sst}_{G,X}(k)$ denote the groupoid of semistable $G$-bundles. By  \cite[Corollary 5.6]{FGL}, there is an equivalence of groupoids \footnote{In \cite[Corollary 5.6]{FGL}, the left side of \eqref{Gbun Tann} was incorrectly taken to be the groupoid of degree zero semistable $G$-bundles; the degree zero condition should be removed in order for the equivalence to hold.}
\begin{equation}\label{Gbun Tann}
\Bun^{\sst}_{G,X}(k)\cong \Fun^{\ot}(\Rep_{k}(G), \Vect^{\sst,0}_{X}).
\end{equation}

Let $J$ be the Jacobian of $X$ (of course we know that $J$ can be identified with $X$ canonically, but we prefer to distinguish them). By \cite[Theorem 5.7]{FGL}, Fourier-Mukai transform gives an equivalence of tensor categories
\begin{equation*}
\Vect^{\sst,0}_{X}\cong \Coh^{\tors}(J).
\end{equation*}
Here the right side denotes the groupoid of torsion coherent sheaves on $J$, with tensor structure given by convolution. Under the Fourier-Mukai transform, the fiber functor $\om_{x_{0}}$ on $\Vect^{\sst,0}_{X}$ corresponds to the global section functor on $\Coh^{\tors}(J)$.

Let $\Coh^{\tors}(J)_{\textup{ss}}\subset \Coh^{\tors}(J)$ be the full sub-groupoid of objects that are direct sums of skyscraper sheaves. Let $\Coh^{\tors}(J)_{1}$ be those supported at the identity element of $J$. Both are closed under convolution, and we have an equivalence of tensor categories
\begin{equation*}
\Coh^{\tors}(J)=\Coh^{\tors}(J)_{\textup{ss}}\ot \Coh^{\tors}(J)_{1}.
\end{equation*}
Correspondingly we have the subgroupoid $(\Vect^{\sst,0}_{X})_{\textup{ss}}\subset \Vect^{\sst,0}_{X}$ consisting of direct sums of line bundles, and $(\Vect^{\sst,0}_{X})_{\uni}$ consisting of extensions of the trivial bundle. We have an equivalence of tensor categories
\begin{equation}\label{Vect ss uni}
\Vect^{\sst,0}_{X}=(\Vect^{\sst,0}_{X})_{\textup{ss}}\ot (\Vect^{\sst,0}_{X})_{\uni}.
\end{equation}

Let $\Pi_{X}$ be the pro-algebraic group attached to the Tannakian category $\Vect^{\sst,0}_{X}$ using the fiber functor  $\om_{x_{0}}$. According to the decomposition \eqref{Vect ss uni}, we have an isomorphism of commutative pro-algebraic groups
\begin{equation*}
\Pi_{X}\cong T_{J}\times U_{J}
\end{equation*}
where $U_{J}$ is the unipotent radical of $\Pi_{X}$ such that $\Rep_{k}(U_{J})\cong (\Vect^{\sst,0}_{X})_{\uni}$, and $T_{J}$ is the diagonalizable group over $k$ with $\Rep_{k}(T_{J})\cong (\Vect^{\sst,0}_{X})_{\textup{ss}}$.  Using \eqref{Gbun Tann}, we get an equivalence of groupoids
\begin{equation}\label{Gbun Hom}
\Bun^{\sst}_{G,X}(k)\cong \Hom_{k-\gp}(\Pi_{X}, G)/G\cong \Hom_{k-\gp}(T_{J}\times U_{J}, G)/G.
\end{equation}
Moreover, \cite[Corollary 5.6]{FGL} implies a stronger result on the automorphism groups of both sides: if $\cE\in \Bun^{\sst}_{G,X}(k)$ corresponds to $\ph\in \Hom_{k-\gp}(T_{J}\times U_{J}, G)$, the isomorphism between their automorphism groups on both sides (as abstract groups) can be canonically upgraded to an isomorphism of algebraic groups over $k$. In other words, \eqref{Gbun Hom} is an equivalence of groupoids {\em enriched in $k$-groups}.

\begin{lemma}\label{l:ind centralizer} Let $\cE$ be a semistable $G$-bundle on $X$ corresponding to a homomorphism $\rho: \Pi_X\to G$ according to \eqref{Gbun Hom}. Then $\cE$ is indecomposable if and only if the centralizer of $\rho$ under $G$ does not contain a nontrivial torus.
\end{lemma}
\begin{proof}
By the equivalence \eqref{Gbun Hom} of groupoids enriched in $k$-groups, the algebraic group $\Aut(\cE)$ is the centralizer of $\rho$ under $G$-conjugation. The lemma then follows from Definition \ref{d:ai}.
\end{proof}

\subsection{Description of the Tannakian group}

\sss{$T_{J}$}
Since one-dimensional objects in $(\Vect^{\sst,0}_{X})_{\textup{ss}}$ (under tensor product) are parametrized by the group $J(k)$, we have a canonical isomorphism
\begin{equation*}
J(k)\cong \xch(T_{J}).
\end{equation*}

\sss{Cartier duality}
Recall the contravariant equivalence between commutative affine $k$-group schemes and commutative formal $k$-groups (Cartier duality). For a  commutative affine $k$-group scheme $G=\Spec A$, where $A$ is a commutative and cocommutative Hopf $k$-algebra, $\DD(G)=\Spf A^{*}$, where $A^{*}=\Hom_{k}(A,k)$ is the linear dual of $A$ viewed as a topological $k$-algebra (the topology on $A^*$ is defined as the limit of discrete spaces $V^{*}$, where $V\subset A$ runs over finite-dimensional $k$-subspaces of $A$). Conversely, to a commutative $k$-formal group $\cG=\Spf B$, where $B$ is a commutative and cocommutative topological Hopf $k$-algebra, we attach the group scheme $\DD(\cG)=\Spec A$, where $A=\Hom_{\cont}(B,k)$ is the Hopf algebra of continuous $k$-linear functions on $B$. 

For a diagonalizable group $D$ over $k$, its Cartier dual is the constant group $\un{\xch(D)}$, viewed as a formal group over $k$. In particular, $T_{J}$ is the Cartier dual of the constant formal group $\un{J(k)}$.

\sss{$U_{J}$}\label{sss:UJ}
Let $\cD_{J}=\Hom_{\cont}(\wh \cO_{J,1}, k)$, the $k$-vector space of distributions on $J$ supported at $1$ (continuous with respect to the $\fm$-adic topology of $\cO_{J,1}$, where $\fm$ is the maximal ideal). This is a Hopf algebra, corresponding to the commutative $k$-group scheme $\DD(\wh J)$ Cartier dual to the formal group $\wh J$ (formal completion of $J$ at $1$). 

Now $\Coh^{\tors}(J)_{1}$ is the same as topological $\wh \cO_{J,1}$-modules that are finite-dimensional over $k$, hence it is canonical equivalent to the category of finite-dimensional $\cD_{J}$-comodules as tensor categories. Therefore, we have a canonical isomorphism
\begin{equation*}
U_{J}\cong \DD(\wh J).
\end{equation*}

When $\ch(k)=0$, we have $U_{J}\cong \Ga$. More precisely, to fix such an isomorphism we need to choose a $k$-basis $\xi$ of the cotangent space $T^{*}_{1}J\cong \cohog0{X,\om_{X}}$. With this choice, there is a unique isomorphism of formal group laws $\wh\cO_{J, 1}\cong k\tl{t}$ such that $t$ has image $\xi$ in $\fm/\fm^{2}$. Taking Cartier duals, we get an isomorphism $U_{J}\cong \Ga$.  After fixing such an isomorphism, the datum of a homomorphism $U_{J}\to G$ is the same as a unipotent element in $G$. In other words, $\Hom_k(U_J, G)\cong \cU_G(k)$, where $\cU_{G}$ is the unipotent variety of $G$.

When $\ch(k)=p$ and $X$ is ordinary, we have $\wh J\cong \wh \Gm$. Taking Cartier dual we get an isomorphism $U_{J}\cong \un{\ZZ_{p}}$ (constant group as a group scheme over $k$; the isomorphism is not canonical). If we choose such an isomorphism, we again get a bijection $\Hom_{k}(U_{J}, G)\cong \Hom_{k}(\un{\ZZ_{p}}, G)\isom \cU_{G}(k)$ by  evaluating a map $u: \un{\ZZ_{p}}\to G$ at $1\in \ZZ_{p}$.

When $X$ is supersingular, $\wh J$ is a one-dimensional formal group of height $2$. We do not know how to classify homomorphisms $U_{J}\to G$ in Lie-theoretic terms.

\subsection{Iso-unipotent $G$-bundles} A vector bundle $\cV$ on $X$ is called {\em iso-unipotent} if there is a finite \'etale map $X'\to X$ such that $\cV|_{X'}$ is a successive extension of trivial line bundles. An iso-unipotent vector bundle on $X$ is automatically semistable of degree 0. We have the full tensor subcategory $\Vect^{\isu}_{X}\subset \Vect^{\sst,0}_{X}$ of iso-unipotent objects. Denote the Tannakian group of this category by $\Pi^{\isu}_{X}$. 

A $G$-bundle $\cE$ on $X$ is called {\em iso-unipotent} if for every representation $V$ of $G$, the associated vector bundle $\cE_V$ is iso-unipotent in the sense defined above. Then iso-unipotent $G$-bundles are automatically semistable by the equivalence \eqref{Gbun Tann}. We have a subgroupoid $\Bun^{\isu}_{G,X}(k)\subset \Bun^{\sst,\tors}_{G,X}(k)$ of iso-unipotent objects. Then \eqref{Gbun Hom} restricts to an equivalence enriched in $k$-groups
\begin{equation}\label{Gbun iso}
\Bun^{\isu}_{G,X}(k)\cong \Fun^{\ot}(\Rep_{k}(G), \Vect^{\isu}_{X})\cong \Hom_{k}(\Pi^{\isu}_{X}, G)/G.
\end{equation}

\sss{Structure of $\Pi^{\isu}_{X}$}
Let $T_{J}^{\isu}$ be the diagonalizable group with $\xch(T_{J}^{\isu})=J(k)_{\tors}$. Then we have a canonical quotient map $T_{J}\surj T_{J}^{\isu}$ whose kernel is a pro-torus. We have a canonical isomorphism
\begin{equation*}
\Pi_{X}^{\isu}\cong T_{J}^{\isu}\times U_{J}.
\end{equation*}

We can describe $\Pi_{X}^{\isu}$ all at once using Cartier duality.  Let $\wh J_{\tors}$ be the completion of $J$ at all of its torsion points. More precisely, it is the colimit of $\wh J_{n}$, the completion of $J$ along $J[n]$, over all positive integers $n$.

\begin{lemma} The commutative $k$-group $\Pi^{\isu}_{X}$ is canonically Cartier dual to the formal group $\wh J_{\tors}$.
\end{lemma}
\begin{proof}
Under the Fourier-Mukai transform (see \cite[Theorem 5.7]{FGL}), $\Vect^{\isu}_{X}$ is equivalent to the full subcategory of $\Coh^{\tors}(J)$ consisting of torsion sheaves that are supported at torsion points of $J$. Let $\Coh^{\tors}(J)_{J[n]}$ be the torsion sheaves set-theoretically supported at $J[n]$, then
\begin{equation*}
\Vect^{\isu}_{X}\cong \colim_{n}\Coh^{\tors}(J)_{J[n]}
\end{equation*}
as tensor categories (with convolution on the right). Now $\Coh^{\tors}(J)_{J[n]}$ is the same as topological $\wh \cO_{J,J[n]}$-modules that are finite-dimensional over $k$, which can be identified with $\Hom_{\cont}(\wh \cO_{J,J[n]}, k)$-comodules  that are finite-dimensional over $k$. Therefore we have a tensor equivalence
\begin{equation*}
\Coh^{\tors}(J)_{J[n]}\cong \Rep_{k}(\DD(\wh J_{n}))
\end{equation*}
compatible with transition functors for $n|n'$. Passing to the colimit we get
\begin{equation*}
\Vect^{\isu}_{X}\cong \colim_{n}\Rep_{k}(\DD(\wh J_{n}))\cong \Rep_{k}(\DD(\colim_{n}\wh J_{n}))=\Rep_{k}(\DD(\wh J_{\tors})).
\end{equation*}
The fiber functor $\om_{x_{0}}$ on $\Vect^{\isu}_{X}$ corresponds to the global sections functor on $\Coh^{\tors}(J)$, hence to the forgetful functor on $\Rep_{k}(\DD(\wh J_{\tors}))$. This gives a canonical isomorphism $\Pi_{X}^{\isu}\cong \DD(\wh J_{\tors})$.
\end{proof}

Using \eqref{Gbun iso} we get:

\begin{cor} There is a canonical equivalence of groupoids (enriched in $k$-groups):
\begin{equation*}
\Bun^{\isu}_{G,X}(k)\cong \Hom_{k}(\DD(\wh J_{\tors}), G)/G.
\end{equation*}
\end{cor}

The reason we care about iso-unipotent $G$-bundles is the following.

\begin{cor}\label{c:ind isotriv} Any indecomposable $G$-bundle on $X$ is iso-unipotent.
\end{cor}
\begin{proof} Let $\cE$ be an indecomposable $G$-bundle over $X$. By Lemma \ref{l:ind sst}, $\cE$ is semistable. By \eqref{Gbun Hom}, $\cE$ corresponds to a homomorphism $\rho: \Pi_X\cong T_J\times U_J\to G$. By Lemma \ref{l:ind centralizer}, the centralizer of $\rho$ contains no nontrivial torus. Note that $\rho(T_J)$ is in the centralizer of $\rho$, so $D=\rho(T_J)$ does not contain any nontrivial torus. This means $\xch(D)\subset \xch(T_J)$ is torsion, i.e., the map $\rho|_{T_J}$ factors through the quotient $T_J^{\isu}$, or that $\rho$ factors through $\Pi_X^{\isu}$.  
\end{proof}

Next we we give a more concrete description of $\wh J_{\tors}$, and hence of $\Bun^{\isu}_{G,X}(k)$.

Let $C_{2,1}(G)$ be the set of commuting triples $(s,t,u)$ in $G$, where $s,t$ are semisimple, and $u$ is unipotent.
Let $C_{2,1}(G)_\tors$ be the subset of $C_{2,1}(G)$ where $s$ and $t$ have finite order.

\sss{Characteristic zero}
When $\ch(k)=0$, for any positive integer $n$, we have a canonical isomorphism
\begin{equation*}
J[n](k)\cong \upH^{1}_{\et}(X, \mu_{n})\cong \Hom(\pi^{\et}_{1}(X)/n, \mu_{n}(k)).
\end{equation*}
This identifies $\un{J[n](k)}$ with the Cartier dual to $\un{\pi^{\et}_{1}(X)/n}$, viewed as a constant formal group and a constant group over $k$ respectively. Taking colimit (resp. limit) over $n$, we conclude that $\un{J(k)_\tors}$ is in Cartier duality with the affine $k$-group $\un{\pi^{\et}_{1}(X)}$ associated to the profinite group $\pi^{\et}_{1}(X)$. We get a canonical isomorphism of $k$-groups 
\begin{equation}\label{TJ et char 0}
T_{J}^{\isu}\cong \un{\pi^{\et}_{1}(X)}.
\end{equation}

\begin{prop}\label{p:isu ch0} Let $\ch(k)=0$. Choose a $\wh\ZZ$-basis for $\pi^{\et}_{1}(X)$ and a $k$-basis of $\cohog{0}{X,\om_{X}}$. Then there is a canonical equivalence of groupoids enriched in $k$-groups
\begin{equation*}
\Bun^{\isu}_{G,X}(k)\cong C_{2,1}(G)_{\tors}/G.
\end{equation*}
\end{prop}
\begin{proof}
The chosen $\wh\ZZ$-basis for $\pi^\et_{1}(X)$ gives an isomorphism $T_{J}^{\isu}\cong\un{\wh\ZZ}\times \un{\wh\ZZ}$ by \eqref{TJ et char 0}. The choice of $\xi$, thought of as a basis of $T^{*}_{1}J$, gives an isomorphism $U_{J}\cong \Ga$ (see \S\ref{sss:UJ}). Altogether we have fixed an isomorphism of pro-algebraic groups
\begin{equation*}
\Pi_X^\isu\cong T_{J}^{\isu}\times U_{J}\cong \un{\wh\ZZ}\times \un{\wh\ZZ}\times \Ga.
\end{equation*}
Taking homomorphisms to $G$, noting that a homomorphism $\un{\wh\ZZ}\to G$ is the same datum as an element of $G$ of finite order (by evaluation at $1$), we get a canonical bijection between $\Hom(\Pi^\isu_X, G)$ and $C_{2,1}(G)_\tors$, equivariant under the conjugation actions on both. We conclude using \eqref{Gbun iso}.
\end{proof}

\sss{Characteristic $p$}
First consider the ordinary case. We need a variant of $C_{2,1}(G)_\tors$. Let $C^{ord}_{2,1}(G)_{\tors}$ be the set of commuting quadruples $(r,s,t,u)$, where $s,t$ are semisimple elements in $G$ of finite order (whose orders are necessarily prime to $p$), $u$ is a unipotent element in $G$, and $r$ is a homomorphism $\ZZ_{p}(1)\to G$, where $\ZZ_{p}(1)=\lim_{n}\mu_{p^{n}}$ (as an affine group over $k$) is Cartier dual to the constant formal group $\un{\QQ_{p}/\ZZ_{p}}$. 

\begin{prop}\label{p:isu ord} Suppose $\ch(k)=p$ and $X$ is ordinary. Make the following choices:
\begin{itemize}
\item A $\wh\ZZ^{(p')}$-basis of the prime-to-$p$ quotient of the fundamental group $\pi^{\et}_{1}(X)^{(p')}$.
\item An isomorphism $J[p^{\infty}](k)\cong\QQ_{p}/\ZZ_{p}$. 
\end{itemize}
Then there is a canonical equivalence of groupoids
\begin{equation*}
\Bun^{\isu}_{G,X}(k)\cong C^{ord}_{2,1}(G)_{\tors}/G.
\end{equation*}
\end{prop}
\begin{proof}
We have a canonical isomorphism
\begin{equation*}
T_{J}^{\isu}\cong \DD(\un{J(k)_{\tors}})=\DD(\un{J[p^{\infty}](k)})\times \un{\pi_{1}(X)^{(p')}}.
\end{equation*}
Here the second isomorphism uses the Cartier duality between the prime-to-$p$ torsion $\un{J(k)^{(p')}_{\tors}}$ and $\un{\pi_{1}(X)^{(p')}}$ as in the characteristic zero case.

The choices made in the statement of the proposition give an isomorphism 
\begin{equation*}
T_{J}^{\isu}\cong \ZZ_{p}(1)\times \un{\wh\ZZ}^{(p')}\times  \un{\wh\ZZ}^{(p')}
\end{equation*}
of pro-algebraic groups over $k$.
Therefore $\Hom(T_{J}^{\isu}, G)$ can be identified with the set of commuting triples $(r,s,t)$ as in the definition of $C^{ord}_{2,1}(G)_{\tors}$. 

The isomorphism $J[p^{\infty}](k)\cong \QQ_{p}/\ZZ_{p}$ after Cartier-Serre duality (within $p$-divisible groups) induces an isomorphism of formal groups $\wh J\cong \wh \Gm$, and hence $U_{J}\cong \DD(\wh J)\cong \un{\ZZ_{p}}$. This isomorphism identifies $\Hom(U_{J}, G)$ with the set of unipotent elements in $G$ canonically (see \S\ref{sss:UJ}). Altogether we get a canonical bijection between $\Hom(T_{J}^{\isu}\times U_{J}, G)$ and $C^{ord}_{2,1}(G)_{\tors}$.
\end{proof}

Finally, for the supersingular case, we let $C^{ssing}_{2,1}(G)_\tors$ be the set of commuting triples $(s,t,u)$, where $s,t\in G$ are semisimple of finite order, and $u: U_{J}=\DD(\wh J)\to G$ is a homomorphism.  The same argument as in the ordinary case shows:

\begin{prop}\label{p:isu ssing} Suppose $\ch(k)=p$ and $X$ is supersingular. Choose a $\wh\ZZ^{(p')}$-basis of the prime-to-$p$ quotient of the fundamental group $\pi^{\et}_{1}(X)^{(p')}$. Then there is a canonical equivalence of groupoids enriched in $k$-groups
\begin{equation*}
\Bun^{\isu}_{G,X}(k)\cong C^{ssing}_{2,1}(G)_{\tors}/G.
\end{equation*}
\end{prop}

\subsection{Indecomposable $G$-bundles on elliptic curves}
Let $k$ be any algebraically closed field, and $X$ be an elliptic curve over $k$.
Let $\Bun_{G,X}(k)^{\AI}$ be the groupoid of indecomposable $G$-bundles on $X$, whose set of isomorphism classes is denoted $\AI_{G,X}(k)$.

Let $C_{2,1}(G)^{\AI}\subset C_{2,1}(G)$ be the subset of triples $(s,t,u)$ such that the simultaneous centralizer $G_{s,t,u}$ does not contain a nontrivial torus (if it is reduced then it is equivalent to $G^{\c}_{s,t,u}$ being unipotent). Recall that $N_2(G)$ defined in \S\ref{sss:n2} (now $G$ is over an arbitrary algebraically closed field $k$) is exactly the set of $G$-conjugacy classes in $C_{2,1}(G)^{\AI}$.

\begin{cor}\label{c:AI ell N2} Suppose either $\ch(k)=0$,  or $\ch(k)=p$ is \goodp for $G$ and $X$ is ordinary. Make choices as in Proposition \ref{p:isu ch0} or \ref{p:isu ord} accordingly. Then there is a canonical equivalence of groupoids enriched in $k$-groups
\begin{eqnarray}
    \Bun_{G,X}(k)^{\AI}\cong C_{2,1}(G)^{\AI}/G.
\end{eqnarray}
In particular, 
\begin{equation*}
\#\AI_{G,X}(k)=n_{2}(G).
\end{equation*}
\end{cor}
\begin{proof}

By Corollary \ref{c:ind isotriv}, $\Bun_{G,X}(k)^{AI}\subset \Bun^\isu_{G,X}(k)$. 

First consider the case $\ch(k)=0$. After choosing a basis for $\pi^{\et}_1(X)$ and a basis for $\upH^0(X,\om)$, we get a canonical fully faithful embedding of groupoids
\begin{equation*}
    \Bun_{G,X}(k)^{\AI}\subset \Bun^\isu_{G,X}(k)\isom C_{2,1}(G)_\tors/G.
\end{equation*}
We claim that the image of the above embedding is exactly $C_{2,1}(G)^\AI/G$. The only non-obvious thing here is that $C_{2,1}(G)^\AI\subset C_{2,1}(G)_\tors$. We shall prove this for any $k$ (not necessarily characteristic zero). Indeed, for the contrary, if $(s,t,u)\in C_{2,1}(G)$ and, say, $s$  is not of finite order, then the Zariski closure of $\{s^n; n\in \ZZ\}$ is a diagonalizable $k$-group of dimension at least one, hence must contain a nontrivial torus, which is contained in the centralizer $G_{s,t,u}$.

Now consider the case where $\ch(k)=p$ and $X$ is ordinary. As in the above discussion, upon making the choices as in Proposition \ref{p:isu ord}, we have a canonical fully faithful embedding of groupoids enriched in $k$-groups
\begin{equation*}
   \ph_{ord}: \Bun_{G,X}(k)^{\AI}\subset \Bun^\isu_{G,X}(k)\isom C^{ord}_{2,1}(G)_\tors/G.
\end{equation*}
We claim that the image of $\ph_{ord}$ consists of quadruples $(r,s,t,u)$ where $r=1$ and $(s,t,u)\in C_{2,1}(G)^\AI$ (where $s,t$ are necessarily of finite order by the previous paragraph). Indeed, $r$ factors through a homomorphism  $\ov r: \mu_{p^n}\to G$ for some $n\ge0$. By Lemma \ref{l:centralizer Levi}(2), since $\pi_1(G)$ has no $p$-torsion, $C_G(\ov r)$ is a Levi subgroup of $G$. Since $(r,s,t,u)$ corresponds to an indecomposable $G$-bundle, $G_{r,s,t,u}$ contains no nontrivial torus. Now $G_{r,s,t,u}\subset C_G(\ov r)$, hence $Z(C_G(\ov r))$ contains no nontrivial torus, which forces $C_G(\ov r)=G$ and $\Im(\ov r)\subset ZG$. By assumption, $\xch(ZG)$ also has no $p$-torsion, therefore $\ov r$, and hence $r$, has to be trivial.

The above argument gives an equivalence of groupoids $\Im(\ph_{ord})\cong C_{2,1}(G)^\AI/G$. Combined with the fact that $\ph_{\ord}$ is a fully faithful embedding of groupoids, we get the desired equivalence between $\Bun_{G,X}(k)^\AI$ and $C_{2,1}(G)^\AI/G$.
\end{proof}

\begin{cor}\label{c:AI counting ell} Let $X$ be an ordinary elliptic curve defined over a finite field $\FF_{q}$. Assume $p=\ch(\FF_{q})$ is \goodp for $G$, and $q$ is large enough so that $\Gal(\ov\FF_{q}/\FF_{q})$ acts trivially on the set $\AI_{G,X}(\ov \FF_{q})$. Then
\begin{equation*}
\#\AI_{G,X}(\FF_{q})=n_{2}(G).
\end{equation*}
\end{cor}

\begin{cor}
Let $X$ be an ordinary elliptic curve over an algebraically closed field $k$ with origin $x_0$. Assume $p$ is large enough so that Theorem \ref{th:geom M} holds for the semisimple group $G$ and $X$. Then for a generic choice of the stability parameter $\th$, the sign isotypic summand of the homology $\homog{*}{\cN_G(\bI_{x_0})^{\tst}}\j{\sgn}$ is concentrated in degree $2\dim\cB$, and
\begin{equation*}
n_{2}(G)=\dim \homog{2\dim \cB}{\cN_G(\bI_{x_0})^{\tst}}\j{\sgn}.
\end{equation*}
\end{cor}
\begin{proof} We first treat the case $k=\ov \FF_p$. We may assume that $X$ is defined over some $\FF_q\subset k$. Note that $\dim\Bun_{G,X}=0$ when $X$ has genus one, and $\dim\Bun_{G,X}(\bI_x)=\dim \cB$. 
By Corollary \ref{c:main N}, we have (at least for $n$ large enough)
\begin{equation}\label{AI ell qn}
    \#\AI_{G,X}(\FF_{q^n})=q^{n\dim\cB}\Tr(\Fr^n, \homog{*}{\cN_G(\bI_{x_0})^{\tst}}\j{\sgn}).
\end{equation}
Here $\Fr\in \Gal(k/\FF_q)$ is the geometric Frobenius for $\FF_q$. Note that $\cN_G(\bI_{x_0})^{\tst}$ is proper over $\FF_q$ of dimension $\dim \cB$. Therefore $\cohog{i}{\cN_G(\bI_{x_0})^{\tst}}$ has weights $\le i$. On the other hand, the inclusion $\cN_G(\bI_{x_0})^{\tst}\incl \cM_G(\bI_{x_0})^{\tst}$ induces an isomorphism on cohomology by the dilation $\Gm$-action. Since $\cM_G(\bI_{x_0})^{\tst}$ is smooth, $\cohog{i}{\cN_G(\bI_{x_0})^{\tst}}\cong \cohog{i}{\cM_G(\bI_{x_0})^{\tst}}$ has weights $\ge i$. Altogether we see that $\cohog{i}{\cN_G(\bI_{x_0})^{\tst}}$ is pure of weight $i$, and $\homog{i}{\cN_G(\bI_{x_0})^{\tst}}$ is pure of weight $-i$. 

Take $n$ sufficiently divisible so that $\Fr^n$ acts trivially on $\AI_{G,X}(\ov \FF_q)$, then the left side of \eqref{AI ell qn} is equal to $n_2(G)$ by Corollary \ref{c:AI counting ell}. On the other hand, for the right side of \eqref{AI ell qn} to be constant for all $n$ sufficiently divisible, using purity, the only possibility is that $\homog{*}{\cN_G(\bI_{x_0})^{\tst}}\j{\sgn}$ is concentrated in degree $2\dim \cB$. Comparing the archimedean norms of both sides of \eqref{AI ell qn} as $n\to\infty$ and being sufficiently divisible at the same time, we get the desired identity.

For a general algebraically closed field $k$ of characteristic $p$, it is easy to see that $n_2(G)$ is the same as in the case $k=\ov\FF_p$. Using a spreading-out argument (choosing a finitely generated subring $R_0\subset k$ over which $X$ is defined, so that $\cN_G(\bI_{x_0})^{\tst}$ is also defined over $R_0$), we see that the dimension of $\homog{*}{\cN_G(\bI_{x_0})^{\tst}}\j{\sgn}$ is the same as its counterpart over $\ov\FF_p$ (choose a sufficiently general closed point of $\Spec R_0$ and compare homology of its fiber with the generic fiber using nearby cycles). The general case then follows from the case $k=\ov\FF_p$.
\end{proof}

\section{Counting with parahoric level structures}\label{s:par}

We will generalize Corollaries \ref{c:AI Mc} and \ref{c:main} to the counting problem of indecomposable $G$-bundles with parahoric level structures.

\subsection{Preparations on parahoric subgroups}

Let $k$ be a field and $G$
a connected reductive group over $k$. Consider the loop group $LG$ (as an ind-scheme over $k$). Let $\bP\subset LG$ be a parahoric subgroup with reductive quotient $M$, which is a connected reductive group over $k$. 

\sss{Embedding of $M$ into $G$}\label{sss:emb M to G}
We assume $k$ is algebraically closed in the discussion below. We claim:
\begin{equation}\label{M to G}
    \mbox{There is a canonical $G$-conjugacy class of embeddings $M\incl G$.}
\end{equation} 
To see this, we first choose a Borel subgroup $B_0\subset G_0$ and a maximal torus $T_0$ therein. This determines a standard Iwahori subgroup $\bI\subset LG$, which is the preimage of $B_{0}\subset G$ under $L^{+}_{x}G\to G$. 

We first assume $\bP$ is a standard parahoric subgroup, i.e., $\bP\supset\bI$.  Then $\bP$ corresponds to a facet $\cF$ in the $T_0$-apartment $\cA$ of the building of $LG$. Let $M_{\cF}\subset G$ be the connected reductive subgroup of $G$ containing $T_{0}$ and having roots $\a\in \Phi(G,T_{0})$ such that there exists a hyperplane $H_{\wt \a}\subset \cA$ defined by an affine root $\wt\a$ with finite part  $\a$, such that $\cF\subset H_{\wt\a}$. Then there is a canonical isomorphism $M\cong M_{\cF}$, giving an embedding $M\incl G$. 

General $\bP$ are $(LG)^{\c}$-conjugate to a unique standard parahoric subgroup $\bQ\supset \bI$; the resulting isomorphism $\bP\cong \bQ$ is unique up to inner automorphisms. Therefore the resulting embedding $M\incl G$ is canonical up to precomposing with an inner automorphism of $M$. 

Finally, any two choices of $(B_0, T_0)$ are $G$-conjugate, hence changing the choice of $(B_0,T_0)$ only changes the embedding $M\incl G$ by post-composing with an inner automorphism of $G$. Therefore the embedding $M\incl G$ is canonical up to $G$-conjugation.

\begin{lemma}\label{l:p good for M}
    Let $\bP \subset LG$ be a parahoric subgroup with maximal reductive quotient $M$. Then, if $p=\ch(k)$ is pretty good for $G$, it is pretty good for $M$. 
\end{lemma}
\begin{proof}

    We may assume $k$ is algebraically closed.
    Suppose $p=\ch(k)$ is \goodp for $G$. Let $T\subset G$ be a maximal torus of $G$, $\Phi$ the root system of $G$ with respect to $T$. The discussions in \S\ref{sss:emb M to G} shows that $M$ can be identified with a subgroup of $G$ containing $T$. Under that identification, let $\Phi_{M} \subset \Phi$ be the root subsystem of $M$ with coroot system $\Phi_{M}^{\vee}\subset \Phi^{\vee}$. 
    We need to verify that $p$ is good for $M$, and that $\xch(ZM)$ and $\pi_{1}(M)$ have no $p$-torsion. 
    
    We first show that $p$ is good for $M$. This property only depends only on the Dynkin diagram of $M$, and can be checked case-by-case by reducing to the case $G$ almost simple and $\bP$ a maximal parahoric subgroup.

    Next we show that $\xch(ZM)$ has no $p$-torsion. Let $X=\xch(T)$, then we have  identifications $\xch(ZM) = X / \ZZ \Phi_{M}$ and $\xch(ZG)=X/\ZZ\Phi$.
    By assumption,  $X/\ZZ \Phi$ has no $p$-torsion, so it suffices to prove $\ZZ \Phi / \ZZ \Phi_{M}$ has no $p$-torsion. This allows us to reduce to the case where $G$ is almost simple. One can choose a set $\{\a_i\}_{i\in I}$ of simple roots for $G$, and set $\a_0=-\th$ (where $\th$ is the corresponding highest root) and $\wt I=I\sqcup \{0\}$, such that $\Phi_M$ has simple roots $\{\a_j\}_{j\in J}$ for some proper subset $J\subset \wt I$. Write $\th=\sum_{i\in I}n_i\a_i$ for some $n_i\in \ZZ$. If $J\subset I$, then $\ZZ \Phi / \ZZ \Phi_{M}$ is torsion-free. If $0\in J$, then the torsion part of $\ZZ \Phi / \ZZ \Phi_{M}$ is the cyclic group of order equal to
    \begin{equation*}
        \gcd\{n_i|i\in \wt I-J\}.
    \end{equation*}
    Since $p$ is good for $G$, $p$ does not divide any $n_i$, and we find that $X/\ZZ \Phi$ has no $p$-torsion. 
    
    Finally we prove that $\pi_1(M)$ has no $p$-torsion. Let $Y=\xcoch(T)$, then we have identifications  $\pi_{1}(M) = Y / \ZZ \Phi_{M}^{\vee}$ and $\pi_1(G)=Y/\ZZ\Phi^\vee$. The argument is similar to the previous paragraph, except that $\{n_i\}$ are replaced by the coefficients $\{n_i^\vee\}_{i\in I}$ of $\th^\vee$ in terms of simple coroots. When $p$ divides one of the $n_i^\vee$, we say it is \emph{torsion} for $G$. Using for example the explicit list of bad and torsion primes in \cite[Table 14.1.]{Malle-Testerman} one easily sees that torsion implies bad, i.e., good primes are non-torsion, proving the claim. 
\end{proof}

We can now extend the notion of $G$-relevant subgroup and subalgebras a bit.

\begin{defn}\label{def:Grel M} Let $k$ be any field. Suppose $M$ is an algebraic $k$-group, equipped with a $G(\ov k)$-conjugacy class of embeddings $M_{\ov k}\subset G_{\ov k}$ (for example, $M$ may be the reductive quotient of a parahoric subgroup $\bP\subset LG$, by \eqref{M to G}). Then we say that a $k$-subgroup $H\subset M$ is {\em $G$-relevant} if, under some (equivalently any) embedding $\io: M_{\ov k}\subset G_{\ov k}$ in the fixed conjugacy class of embeddings, $\io(H_{\ov k})$ is $G_{\ov k}$-relevant. We say that a $k$-subalgebra $\frh\subset \fm$ is {\em $G$-relevant} if it is the Lie algebra of some $G$-relevant subgroup $H\subset M$.
\end{defn}

Now let $k$ be a finite field. Assume $M$ is a connected reductive $k$-group with a fixed $G(\ov k)$-conjugacy class of embeddings $i: M_{\ov k}\incl G_{\ov k}$ (again our examples will be reductive quotients of parahoric subgroups of $LG$). Let $\frc_M=\fm\sslash M$ and let $\chi_M: \fm\to \frc_M$ be the projection.

\begin{defn} In the above situation, a {\em $G$-selection function} $\xi: \frc_{M}(k)\to \Qlbar$ is one that satisfies 
\begin{enumerate}
\item $\j{\one_{\fra}, \chi_{M}^{*}\xi}_{\fm}=0$ for any $G$-relevant nontrivial torus $A\subset M$ such that $A\ne C_{M}$ (where  $\fra=\Lie A$).
\item $\xi(\frz_{M})=1$, where $\frz_{M}=\Lie ZM$.
\end{enumerate}
\end{defn}

One can prove a variant of Proposition \eqref{p:sel fun} in this situation: for any $G$-selection function $\xi$ on $\frc_{M}$, and any $G$-relevant subgroup $H\subset M$ whose neutral component is not unipotent, we have $\j{\one_{\frh}, \chi^{*}_{M}\xi}_{\fm}=0$.

The above definition applies to the reductive quotient $M$ of a parahoric subgroup $\bP\subset LG$ by \eqref{M to G}. In the case $\bP=\bI$ is an Iwahori subgroup, its reductive quotient can be identified canonically with the universal Cartan $T$. In this case, and assuming $p$ is \goodp for $G$, a $G$-selection function on $\frc_T(k)=\frt$ can be given as 
\begin{equation*}
\frt\xr{\l}k\xr{\psi}\Qlbar^\times
\end{equation*}
where $\l\in \frt^*$ is $W$-coregular.

\subsection{Preparations on bundles with parahoric level structures}
Fix a finite subset $S\subset |X|$ and a parahoric subgroup $\bP_{x}\subset L_{x}G$ for each $x\in S$. Let $\bP_{S}$ be the collection $\{\bP_{x}\}_{x\in S}$. Consider the moduli stack $\Bun_{G}(\bP_{S})$. We define the notion of absolute indecomposability for a $G$-bundle with $\bP_{S}$-level structures in a way similar to Definition \ref{d:ai}: $\cE\in \Bun_{G}(\bP_{S})(k)$ is absolutely indecomposable if $C_{G}$ is the maximal torus in $\Aut(\cE)$. The goal here is to generalize the main result Corollary \ref{c:main} to the counting of absolutely indecomposable $G$-bundles with parahoric level structures.

For $\g\in \pi_{0}(\Bun_{G})(k)=\pi_{0}(\Bun_{G}(\bP_{S}))(k)$, the groupoids $\Bun_{G}(\bP_{S})^{\AI}$ and $\Bun^{\g}_{G}(\bP_{S})^{\AI}$  are defined similarly to the case without level structures. Similarly to Definition \ref{def:AIk}, let $\AI_{G,X,\bP_S}(k)$ be the set of isomorphism classes of absolutely indecomposable $G$-bundles over $X_{\ov k}$ with $\bP_S$-level structures whose isomorphism class is defined over $k$. Let $\AI^{\g}_{G,X,\bP_S}(k)\subset \AI_{G,\bP_S, X}(k)$ be the subset of bundles in the $\g$-component of $\Bun_G(\bP_S)$.

\begin{exam} Let $G=\GL_{n}$. Assume $S$ consists of $k$-points of $X$. Let $P_{x}\subset G$ be the parabolic subgroup of block upper triangular matrices with block sizes $n_{x,1},\cdots, n_{x,\ell_{x}}$.  Let $\bP_{x}\subset L_{x}G$  be the parahoric subgroup that is the preimage of $P_{x}$ under the evaluation map $L^{+}_{x}G\to G$.

In this case, a $k$-point of $\Bun_{G}(\bP_{S})$ is the same datum as a rank $n$ vector bundles $\cV$ on $X$ together with a partial flag  
$$F_{x,\bu}: 0=F_{x,0}\subset F_{x,1}\subset \cdots\subset F_{x,\ell_{x}}=\cV_{x}$$ 
in the fiber $\cV_{x}$ with $\dim_{k} F_{x,i}/F_{x,i-1}=n_{x,i}$ for $1\le i\le \ell_{x}$. Then $\cE=(\cV, \{F_{x,\bu}\}_{x\in S})$ is an absolutely indecomposable point of $\Bun_{G}(\bP_{S})$ if, after base change to $\ov k$, there does not exist a decomposition of vector bundles $\cV_{\ov k}=\cV'\op\cV''$ with $\rk \cV'>0, \rk\cV''>0$, such that for all $x\in S$ and $1\le i\le \ell_{x}$,  $F_{x,i}=(F_{x,i}\cap \cV'_{x})\op (F_{x,i}\cap \cV''_{x})$ (each factor here is allowed to be zero). 
\end{exam}

Our goal in this section is to give formulas for the cardinality of $\AI^\g_{G,\bP_S, X}(k)$ analogous to Corollaries \ref{c:AI Mc} and \ref{c:main}. In this subsection, we indicate how to generalize preparatory results in \S\ref{s:aut} to the parahoric case.

Let $\cE\in \Bun_{G}(\bP_{S})(k)$. For $x\in |X-S|$ the evaluation map $\ev_{x}: \Aut(\cE)_{k(x)}\to \Aut_{G}(\cE_{x})\cong G_{k(x)}$ is defined as in \S\ref{ss:ev map}.  When $x\in S$, let $M_{x}$ be the reductive quotient of the parahoric subgroup $\bP_{x}$. Let $\ov\cE_{x}$ be the induced $M_{x}$-bundle from the $\bP_{x}$-level structure on $\cE$. We still have the evaluation map
\begin{equation*}
\ev_{x}: \Aut(\cE)_{k(x)}\to G[\cE_{x}]:=\Aut_{M_{x}}(\ov\cE_{x}).
\end{equation*}
Note that $G[\cE_{x}]$ is an inner form of $M_{x}$. In either case, for any $x\in |X|$, we have a connected reductive group $G[\cE_{x}]$ over $k(x)$, equipped with a $G(\ov{k(x)})$-conjugacy class of embeddings $G[\cE_{x}]\incl G_{k(x)}$.  It then makes sense to talk about $G$-relevant subgroups of $G[\cE_{x}]$ defined in Definition \ref{def:Grel M}.

\begin{lemma}\label{lem:par aut red}
    Assume $p$ is pretty good for $G$. Then for any $\cE\in \Bun_{G}(\bP_{S})(k)$, $\Aut(\cE)$ is reduced. 
\end{lemma}
\begin{proof}
Since the proof is similar to the case without level structure, we only sketch it. For ease of notation, we assume $S$ contains a single point, and we denote the corresponding moduli stack by $\Bun_{G}(\bP)$. Let $\cE \in \Bun_{G}(\bP)$, and let $A_{\cE}$ be the induced reduced subscheme of $\Aut(\cE)$. We will prove that the canonical inclusion $\io_{\cE}: \Lie A_{\cE}\incl \aut(\cE)$ is surjective. 

For that, we first reduce to the case when $\bP = \bI$ is an Iwahori subgroup. Let $M$ be the maximal reductive quotient of $\bP$ with Lie algebra $\fm$. After choosing a trivialization of $\cE$ at $x$, we obtain an evaluation map $\mathfrak{ev} :\aut(\cE) \to \fm $. For any $\a\in \aut(\cE)$ we may consider its image $\mathfrak{ev}(\alpha) \in \fm$. Choose any Borel subalgebra $\frb \subset \fm$ for which $\mathfrak{e}\mathfrak{v}(\alpha)\in \frb$. Such a Borel subalgebra always exists, and it corresponds to a Borel subgroup $B$ of $M$. We can pull back $B$ to endow $\cE$ with Iwahori level structure. Denote this bundle by $\cE'$, and its automorphism group (as bundle with Iwahori level) by $\Aut(\cE')$. 

By construction, $\alpha$ is contained in the subalgebra $\aut(\cE') \subset \aut(\cE)$, which is the Lie algebra of $\Aut(\cE')\subset \Aut(\cE)$. It therefore suffices to prove that $\Aut(\cE')$ is reduced, because in that case, $\aut(\cE')\subset \Im(\io_{\cE})$. So we may assume that $\bP = \bI$. In the case of Iwahori level we now check that 
\[ \io_{\cE}: \Lie A_{\cE}\incl \aut(\cE) \]
is surjective by considering nilpotent and semisimple elements separately. 

Let $e\in \aut(\cE)$ be nilpotent. We use a unipotent logarithm induced from a quasi-logarithm as before, i.e. we argue as in Lemma \ref{l:nilp aut}. By \cite[Lemma 1.8.3]{KV}, any quasi-logarithm $f$ satisfies $f(B) \subset \Lie \, B$ for all Borel subgroups $B$ of $G$. Therefore the unipotent logarithm gives a curve $\AA^1 \to \Aut(\cE)$ with derivative $e$ at $0$ which factors through the reduced subscheme $A_{\cE}$, so $e \in \Im(\Lie A_{\cE}\to \aut(\cE))$.

Let $\g \in \aut(\cE)$ be semisimple. We apply the construction in the proof of Corollary \ref{c:ss aut} to $H=\Aut(\cE)$ and obtain a diagonalizable subgroup $D_{\g} \subset \Aut(\cE)$ with $\g \in \Lie D_{\g}$. Denote by $\underline{\cE}$ the underlying $G$-bundle of $\cE$. Then we have a closed embedding $\Aut(\cE) \subset \Aut(\underline{\cE})$. Thus, $D_{\g}$ is a diagonalizable subgroup of $\Aut(\underline{\cE})$, hence reduced by Corollary \ref{c:diag reduced}. This means $D_{\g} \to \Aut(\cE)$ factors through $A_{\cE}$, hence $\Lie D_{\g} \subset \Im(\iota_{\cE})$. 
\end{proof}

The same argument as in Lemmas \ref{l:ker ev unip} and \ref{l:AI ess unip} shows: 
\begin{lemma}\label{l:ker ev unip par} Let $\cE\in \Bun_{G}(\bP_{S})(k)$ and $x\in |X|$. Then the reduced structure of $\ker(\ev_{x})$ is unipotent. In particular, when $p$ is pretty good for $G$, $\cE$ is absolutely indecomposable if and only if the image of $\ev_{x}$ is essentially unipotent. 
\end{lemma}

\begin{lemma}\label{l:ev image rel par} Let $\cE\in \Bun_{G}(\bP_{S})(k)$ and $x\in |X|$. Assume that $p$ is pretty good for $G$. Then the image of $\ev_{x}$ is a $G$-relevant subgroup of $G[\cE_{x}]$, in the sense of Definition \ref{def:Grel M}. 
\end{lemma}
\begin{proof}  Again we may assume $k$ is algebraically closed. Let $A\subset \Aut(\cE)$ be a maximal torus. For each $s\in S$ we choose a Borel reduction of $\ov\cE_{s}$ that is stable under $A$. These Borel reductions give a point $\cE^{\sh}\in \Bun_{G}(\bI_{S})(k)$ such that $A\subset \Aut(\cE^{\sh})$. Since $\Aut(\cE^{\sh})\subset \Aut(\cE)$, $A$ is a maximal torus in $\Aut(\cE^{\sh})$.  It is therefore sufficient to show that $A_{x}=\ev_{x}(A)$, which is a maximal torus in $G[\cE^{\sh}_{x}]$ (because $\ker(\ev_{x})$ is unipotent), is a $G$-relevant torus in $G[\cE^{\sh}_{x}]$. This allows us to reduce to the case where all $\bP_{s}$ are Iwahori subgroups. 

When  $\bP_{s}=\bI_{s}$ for all $s\in S$, let $\un \cE$ be the underlying $G$-bundle of $\cE$. By Lemma \ref{l:red to centralizer}, we have an $L=C_{G}(A_{x})$-reduction $\un\cF$ of $\un \cE$. At $s\in S$, we choose a point in $\un\cF_{s}$ corresponding to a trivialization $\t_{s}: G\isom \un\cE_{s}$ that intertwines the left translation of $A_{x}$ on $G$ and the $A$-action on $\un\cE_{s}$. The Borel reduction $\cE^{B}_{s}$  of $\un\cE_{s}$ given by $\cE$ corresponds under $\t_{s}$ to a Borel reduction of the trivial $G$-bundle at $s$, which then gives a Borel subgroup $B_{s}\subset G$ as the stabilizer of this Borel reduction. We can thus view $\cE^{B}_{s}$ as a $B_{s}$-bundle at $s$ via $\t_{s}$. Since $A$ preserves $\cE^{B}_{x}$,  we conclude that $A_{x}\subset B_{s}$. Observe that $L\cap B_{s}$ is a Borel subgroup of $L$, since $A_{x}$ is central in $L$ and $A_{x}\subset B_{s}$. Therefore we get a $L\cap B_{s}$-reduction $\cF^{B}_{s}$ of $\un\cF_{s}$ by taking the $L\cap B_{s}$-orbit of $\t_{s}$. Moreover, we have a canonical isomorphism of $B_{s}$-bundles at $s$
\begin{equation}\label{FB EB}
\cF^{B}_{s}\twtimes{L\cap B_{s}}B_{s}=\cE^{B}_{s}.
\end{equation}

Applying the above procedure to every point $s\in S$, we obtain a lifting of $\un \cF$ to a point $\cF\in \Bun_{L}(\bI_{S})(k)$ that induces the Borel reduction of $\un\cE$ given by $\cE$ at every $s\in S$ in the sense of an isomorphism \eqref{FB EB}. In particular, $\Aut(\cF)\subset \Aut(\cE)$. Clearly $A_{x}\subset (ZL)^{\c}\subset \Aut(\cF)\subset \Aut(\cE)$ whose image is identified with $A$. The maximality of $A$ implies $A_{x}=(ZL)^{\c}$, which is $G$-relevant. 
\end{proof}

\subsection{Counting with parahoric level structures}
In this case we can do either of the following:
\begin{itemize}
\item Choose a point $x\in |X|\remove S$ and use $f_{x}$ (coming from a selection function) defined in \S\ref{sss:fx}.
\item Choose a point $x\in S$ and put a function $f_{x}$ on $\fm^{*}_{x}$ constructed from a $G$-selection function $\xi_{x}$ on $\frc_{M_{x}}(k)$ (where $M_{x}$ is the reductive quotient of $\bP_{x}$). More precisely, we set $f_{x}: \fm^{*}_{x}\to \Qlbar$ to be the Fourier transform of $\chi_{M_{x}}^{*}\xi_{x}$, and let $f_{x}: \frp_{x}^{*}\to \Qlbar$ be its extension by zero outside $\fm^{*}_{x}$.
\end{itemize}

The first option is actually a special case of the second, since we can enlarge $S$ to include $x$, and set $\bP_x=L^+_xG$.
Therefore it does not lose generality by considering only the second option. 

We explain how to prove a parahoric version of our counting result using a selection function at $x\in S(k)$. Let $M$ be the maximal reductive quotient of $\bP_x$, and denote its Lie algebra by $\fm$. In addition, $\fm$ has a universal Cartan algebra $\frt_{M}$, and we let $W_M$ be the abstract Weyl group of $M$, which acts on $\frt_{M}$. 

\sss{Moduli of Higgs bundles with parahoric level structures}
Let $\cG_S$ be the parahoric group scheme over $X$ with level structure $\bP_s$ at $s\in S$. Then $\Bun_G(\bP_S)$ can be identified with the moduli stack of $\cG_S$-torsors. For $\cE\in \Bun_G(\bP_S)$ corresponding to a $\cG_S$-torsor $\cE^\na$, we denote by $\Ad(\cE)$ the adjoint bundle of $\cE^\na$, namely $\Ad(\cE)=\cE^\na\times^{G_S}\Lie(\cG_S)$. Note that the fiber $\Ad(\cE)_x$ maps to $\Ad(\ov E_x)=\ov E_x\times^M \fm$. Dually, we have an embedding
\begin{equation}\label{coadj m to g}
    \Ad^*(\ov E_x)\incl \Ad^*(\cE)_x.
\end{equation}
We let $\cM_{G,x}(\bP_{S})$ be the moduli stack classifying pairs $(\cE, \ph)$ where $\cE\in \Bun_{G}(\bP_{S})$ and $\ph\in \upH^{0}(X, \Ad^{*}(\cE)\ot \om(x))$ such that  $\res_x(\ph)$, a priori an element of $ \Ad^*(\cE)_x$, lies in the subspace $\Ad^*(\ov E_x)$ via the embedding \eqref{coadj m to g}. Taking the image of the residue of $\ph$ at $x$ defines a map
\begin{equation*}
    \ov{\res}_x: \cM_{G,x}(\bP_{S}) 	\to	 [\fm^{*} / M ].
\end{equation*}

Let $\bP'_S$ be the level structure obtained from $\bP_S$ by changing $\bP_x$ to the Iwahori level $\bI_x$, and keeping the rest of the level groups. Then we similarly have the moduli stack $\cM_{G,x}(\bP'_S)$ with a residue map
\begin{equation*}
    \res'_x: \cM_{G,x}(\bP'_S)\to \frt^*.
\end{equation*}
For any $\l_x\in \frt^{*}$, let $\cM_{G}(\bP'_S; \l_x)\subset \cM_{G,x}(\bP'_S)$ be the fiber of $\res'_x$ over $\l_x$. Then we have a Cartesian diagram
\begin{equation}\label{MP Cart}
\xymatrix{ 
\cM_{G,x}(\bP'_{S};\l_{x}) \ar[r]\ar[d]	& [\wt{\fm}^{*}_{\l_x} / M] \ar[d] \\
\cM_{G,x}(\bP_{S}) 	\ar[r]^{\ov\res_{x}}	& [\fm^{*} / M ],
}
\end{equation}
where $\wt{\fm}^{*}$ is the Grothendieck alteration of $\fm^{*}$, and the subscript $\l_x$ denotes the preimage of $\l_x$ under the natural map $\wt{\fm}^{*} \to \frt^*_{M}=\frt^*$.

In both cases, connected components are indexed by $\pi_0(\Bun_{G})$, and for $\g\in \pi_0(\Bun_{G})$ we denote the corresponding component by $\cM^{\g}_{G,x}(\bP_{S})$ resp. by $\cM^{\g}_{G}(\bP'_{S}; \l_{x})$.

The following result follows almost directly from our discussion in the case without level structure. 

\begin{theorem}\label{cor:parahoric-counting-admissible} 
Let $x\in S(k)$ and recall $M=M_x$. Suppose $p=\ch(k)$ is \goodp for $G$ (which implies it is pretty good for $M$ by Lemma \ref{l:p good for M}), and $q=\#k$ is large enough with respect to $W$ so that a $W$-coregular admissible collection $\l_x=\{\l_{w,x}\}_{w\in W_{M}}$ exists. Then for each $\g\in \pi_{0}(\Bun_{G})(k)$ we have
\begin{equation*}
\#\AI^{\g}_{G,X, \bP_{S}}(k)=\frac{\#C_{G}(k)}{q^{\dim\Bun_{G}(\bP_{S})+\dim C_{G}}}\frac{1}{\#W_{M}} \sum_{w\in W_{M}} \sgn(w)\Tr(\Fr^*w, \cohoc{*}{\cM^{\g}_{G}(\bP'_{S}; \l_{w,x})}).
\end{equation*}
\end{theorem}

Let us briefly sketch how our previous arguments generalize to the parahoric setting. First, one needs to find a $W$-coregular admissible collection $\l_{x} = \{\l_{w,x} \}_{x\in W_{M}}$. Here $\l_{w,x}$ is a linear function $\frt_{w}\to k$ ($\frt_{w}$, a priori a Cartan subalgebra in $\frg$, is actually in $\fm$). One can start with a $W$-coregular admissible collection $\{\l_w\}_w\in W$ as guaranteed to exist by Lemma \ref{l:adm exist} when $q$ is large, and simply let $\l_{w,x}=\l_w$ for $w\in W_M$.

We then use the following result, which follows from a similar argument as Corollary \ref{c:G-stable counting of AI}, using Lemma \ref{lem:par aut red}. 

\begin{prop} Let $x\in S(k)$ and let $f_{x}: \fm^{*}_{x}\to \Qlbar$ be the Fourier transform of $\chi_{M}^{*}\xi_{x}$ for a $G$-selection function $\chi_{x}$ on $\frc_{M_x}(k)$. Then we have an equality
\begin{equation*}
\#\AI^{\g}_{G,X, \bP_{S}}(k)=\frac{\#C_{G}(k)}{q^{\dim \Bun_{G}(\bP_{S})+\dim C_{G}}}\int_{\cM^{\g}_{G,x}(\bP_{S})(k)} \ov\res_{x}^*f_{x}
\end{equation*}
\end{prop}

This result reduces to computing the quantity $\int_{\cM_{G,x}(\bP_{S})(k)} \ov\res_{x}^*f_{x}$. This is done as in the case without level structure. We use the $G$-selection function $\xi_x$ on $\fm$ as constructed in Proposition \ref{p:cons sel fn from t}, using the $W$-coregular admissible collection $\{\l_{w,x}\}_{w\in W_M}$. We then apply Proposition \ref{p:adm kl} to the case $\frg = \fm$ to rewrite 
\begin{equation*}
\chi_{M_{x}}^{*}\xi_{x} = \frac{1}{\#W_{M}}\sum_{w\in W_{M}}\k_{\l_{w,x},w}.
\end{equation*}
Note that here $\k_{\l_{w,x},w}$ is defined using the Grothendieck alteration for $\fm$. Denote by $\cB_{M,v}$ the Springer fiber over $v\in \fm$ for $M$. The formula \eqref{fx} applied to $M$ instead of $G$ gives
\begin{equation*}
f_{x}(v)=\FT(\chi_{M}^{*}\xi_{x})(v)=\frac{q^{\dim M-\dim \cB_M}}{\#W_{M}}\sum_{w\in W_{M}}\sgn(w)\Tr(\Fr^{*}w, \cohog{*}{\cB^{\l_{w,x}}_{M,v}}).
\end{equation*}
The diagram \eqref{MP Cart} then implies that
\begin{equation*} \int_{\cM^{\g}_{G,x}(\bP_{S}) }\ov\res_{x}^{*}f_x = \frac{q^{\dim M-\dim \cB_M}}{\#W_{M}}\sum_{w\in W}\sgn(w)\Tr(\Fr^{*}w, \cohoc{*}{\cM^{\g}_{G}(\bP'_{S};\l_{x}) }),
\end{equation*}
which gives Theorem \ref{cor:parahoric-counting-admissible}.

In order to derive an analogue of our main result, we need a notion of stability for $\cM_{G}(\bP_{S};\l_{x})$. For this, note that the square
\begin{equation*} 
\xymatrix{ 
\cM_{G,x}(\bP'_{S}) \ar[r]\ar[d]	& [\wt{\fm}^{*} / M] \ar[d] \\
\cM_{G,x}(\bP_{S}) 	\ar[r]^{\res_{x}}	& [\fm^{*} / M ],
}
\end{equation*}
is Cartesian. Choose a generic stability condition $\th\in \xcoch(T)_{\RR}$ as in \S\ref{ss:stable Higgs}. Here $T$ is the reductive quotient of $\bI_{x}$, which is identified with the universal Cartan. The definition of stability for $\cM_{G,x}(\bP'_{S})$ is the same as before, but it is checked with respect to $P$-reductions as $G$-bundles with $\bP'_S$-level structures compatible with Higgs fields. Since we choose $\theta$ generic, then as before, the notion of $\theta$-stable and $\theta$-semistable coincide. 

Let $\ov\cM_{G,x}(\bP'_S)=\cM_{G,x}(\bP'_S)/\BB(C_G)$, and similarly define $\ov\cM_{G,x}^\g(\bP'_S)^{\tst}$. The analogue of Theorem \ref{th:geom M} for $\ov\cM_{G,x}^\g(\bP'_S)^{\tst}$ holds for large $p$ depending on $G, g$ and $\deg S$. We can therefore argue similarly to Corollary \ref{cor:W-action}, to show that there is a canonical $W_M$-action on $\cohoc{*}{\ov\cM^{\g}_{G}(\bP_{S}; 0_x)^{\tst}}$ (the zero fiber of the residue map to $\frt^*$). Denoting by $\sgn_{M}$ the sign character of $W_M$, we thus find the following parahoric analogue of Corollaries \ref{c:main} and \ref{c:main N}. 

\begin{cor} Assume $p$ is sufficient large with respect to the Dynkin diagram of $G$, the genus $g$ and $\deg S$. Assume $q$ is large enough as in Corollary \ref{c:main}. For each $\g\in \pi_{0}(\Bun_{G})(k)$ we have
\begin{eqnarray}
\notag\#\AI^{\g}_{G,X, \bP_{S}}(k)&=&q^{-\dim\Bun_{G}(\bP_{S})-\dim C_{G}}\Tr(\Fr,\cohoc{*}{\ov\cM^{\g}_{G}(\bP_{S}; 0_x)^{\tst}}\j{\sgn_{M}})\\
&=&q^{\dim\Bun_{G}(\bP_{S})}\Tr(\Fr,\homog{*}{\ov\cN^{\g}_{G,x}(\bP'_{S})^{\tst}}\j{\sgn_{M}}).
\end{eqnarray}
\end{cor}

\appendix

\section{Properties of parabolic stable Hitchin fibration}

In this Appendix we supply proofs of the geometric properties of the moduli stack $\cM_G(\bI_x)$ claimed in Theorem \ref{th:geom M}.

We work over an algebraically closed field $k$ of any characteristic. Let $X$ be a smooth projective connected curve of genus $g$ over $k$.  Let $x\in X(k)$.

There is an easy reduction to the case where $G$ is semisimple. Henceforth we assume that
\begin{equation*}
    \mbox{$G$ is a connected semisimple group over $k$.} 
\end{equation*}
We use the usual notions $T,B,\bI_x$, etc from the main body of the paper. Let $\th$ be generic in the fundamental alcove of $\xcoch(T)_\RR$. Let $\cM_G(\bI_x)$ be the moduli stack of $G$-Higgs bundles with Iwahori level structure at $x$, as defined in \S\ref{ss:stable Higgs}, and $\cM_G(\bI_x)^{\tst}$ be the $\th$-stable open locus. 

Recall that we have the residue map
\begin{equation*}
    \phi_x: \cM_G(\bI_x)\to \frt^*
\end{equation*}
and the Hitchin map
\begin{equation*}
    h: \cM_G(\bI_x)\to \cA_G(\bI_x).
\end{equation*}
We denote by $\phi^{\tst}_x$ and $h^{\tst}$ their restrictions to $\cM_G(\bI_x)^{\tst}$.

\subsection{Deligne-Mumford property}\label{a:DM}
The stack $\cM_G(\bI_{x})^{\tst}$ is of finite type over $k$ because the Harder-Narasimhan polygon of any $\th$-stable point in $\cM_G(\bI_{x})$ is uniformly bounded.

Below we show that $\cM_G(\bI_{x})^{\tst}$ is a Deligne-Mumford stack. For this we need to assume that:
\begin{equation}\label{p>h}
    \mbox{$p>\max\{h_i\}$, where $h_i$ ranges over the Coxeter numbers of each simple factor of $G$.}
\end{equation}

To show $\cM_G(\bI_{x})^{\tst}$ is Deligne-Mumford, it is enough to show that for a $\th$-stable $(\cE, \cE^B_x, \ph)$ over $k$, its infinitesimal automorphisms
\begin{equation}\label{inf auto vanish}
    \G(X,\Ad(\cE, \cE^B_x, \ph))=0.
\end{equation}
Here $\Ad(\cE, \cE^B_x, \ph)\subset \Ad(\cE)$ is the subsheaf of $\Ad(\cE)$ of local sections $\a$ such that $[\a,\ph]=0$, and $\a(x)$ lies in $\Ad(\cE^B_x)$. Let $\a\in \G(X,\Ad(\cE, \cE^B_x, \ph))$, and let $\a=\a_s+\a_n$ be the Jordan decomposition of $\a$ (at the generic point of $X$ first, then $\a_s$ and $\a_n$ are automatically in $\G(X,\Ad(\cE, \cE^B_x, \ph))$. Thus it suffices to treat separately the case $\a$ is semisimple and $\a$ is nilpotent.
    
Suppose $\a$ is semisimple, we will show that $\a=0$. For $x\in X(k)$ let $\a_x\in \frg$ be the value of $\a$ at $x$ upon choosing a trivialization of $\cE_x$, which is well-defined up to the adjoint action of $G$. We claim that the adjoint orbit of $\a_x$ is independent of the point $x$. Indeed, for each $V\in \Rep(G)$, the characteristic polynomial of $\a$ acting on the associated bundle $\cE_V=\cE\times^GV$ has constant coefficients in $k$ (being global regular functions on $X$). Therefore $\a_x$ has the same characteristic polynomial on $\cE_{V,x}$ for all $x\in X$. There are only finitely many semisimple adjoint orbits in $\frg$ with fixed characteristic polynomial on all representations, and they are all closed. By the same continuity argument as in Lemma \ref{l:A conj}, we conclude that the $\a_x$ are in the same adjoint orbit for all $x$. 
    
Let $L=C_G(\a_x)$, a Levi subgroup of $G$. If $\a_x\ne0$, then $L\ne G$. As in Lemma \ref{l:red to centralizer} one shows that $\cE$ admits an $L$-reduction $\cE_L$, such that $\ph$ is a section of $\Ad^*(\cE_L)\ot \om(x)$ because $[\a,\ph]=0$. For any parabolic subgroup $P\subset G$ containing $L$ as a Levi subgroup, let $\cE_P=\cE_L\times^LP$ be the induced $P$-bundle, which is a $P$-reduction of $(\cE,\ph)$. Let $P$ and $Q$ be a pair of opposite parabolics both containing $L$  as their Levi subgroup. The $\th$-stability of $(\cE,\ph,\cE^B_x)$ implies 
    \begin{equation}\label{int two cones}
        -\deg_L(\cE_L)\in (\frC_P'^\c+\th(\cE_{P,x}, \cE^B_x))\cap (\frC_Q'^\c+\th(\cE_{Q,x}, \cE^B_x)).
    \end{equation}
    However, since $P$ and $Q$ are opposite,  we have $\frC'_P=-\frC'_Q$ and $\th(\cE_{P,x}, \cE^B_x)=-\th(\cE_{Q,x}, \cE^B_x)$. Since $\th$ is in the fundamental alcove, the right side  of \eqref{int two cones} is empty. This is a contradiction, which implies that $\a=0$. 

Suppose $\a$ is nilpotent. Choose a generic trivialization of $\cE$, then $\a$ gives an element $e\in \frg(F)$, where $F=k(X)$. Since we assume \eqref{p>h}, by \cite[Theorem 1.1]{ST}, $e$ extends to an $\mathfrak{sl}_2$ triple $\{e,h,f\}$ in $\frg(F)$. Eigenvalues of $h$ give a decomposition $\frg(F)=\op_{i\in \ZZ}\frg(F)_i$, and this induces a filtration $\Ad(\cE)_{\ge i}$ of $\Ad(\cE)$ (taking the subbundle whose generic fiber is $\frg(F)_{\ge i}=\op_{j\ge i}\frg(F)_j$). Let $P\subset G(F)$ be the parabolic subgroup with Lie algebra $\frg(F)_{\ge 0}$, then $\cE$ carries a canonical $P$-reduction $\cE_P$. Since $\ph$ commutes with $\a$, at the generic point $\ph$ has to lie in $\frn_P^\bot=\op_{j\le 0}\frg(F)_j^*$ (for the Lie algebra centralizer of $e$ is contained in $\frg(F)_{\ge 0}$), hence $\ph\in \G(X,\cE_P\times^P\frn_P^\bot)$. In other words, $\cE_P$ is a $P$-reduction of the Higgs bundle $(\cE_,\ph)$. Moreover, since $\cE^B_x$ is preserved by $\a$, the relative position $pos(\cE_{P,x}, \cE^B_x)$ is the closed $G$-orbit in $\cB_P\times \cB$. Therefore $\th(\cE_{P,x}, \cE^B_x)\in \xcoch(T_P)_{\RR}$ is in the image of the dominant chamber, i.e., $\th(\cE_{P,x}, \cE^B_x)\in  \frC_P$. The $\th$-stability of $(\cE,\ph,\cE^B_x)$ implies that $\deg_\th(\cE_P)\in -\frC'^\c_P$ (negative obtuse cone), therefore $\deg_{L_P}(\cE_{L_P})\in -\frC'^{\c}_P$. In particular,
\begin{equation}\label{deg<0}
\deg(\Ad(\cE)_{>0})=\deg(\cE_{L_P}\times^{L_P}\frn_P)<0.
\end{equation}

Let $N$ be the largest degree for which $\frg(F)_N\ne0$. Note that $N\le \max\{h_i\}-1$ (since for each simple factor $\frg'$ of $\frg$, the maximal $N$ is achieved by the regular nilpotent element , in which case $N$ is one less than the Coxeter number of $\frg'$). By assumption \eqref{p>h}, we see that $\ad(e)^i: \frg(F)_{-i}\to \frg(F)_i$ is an isomorphism for $i>0$. This implies $\ad(\a)^i: \Gr^{-i}\Ad(\cE)\to \Gr^{i}\Ad(\cE)$ is generically an isomorphism, hence an injection of coherent sheaves. Summing over all $i>0$, we get an injective map $\Ad(\cE)_{<0}\to \Ad(\cE)_{>0}$. In particular, we must have
\begin{equation}\label{deg AdE}
    \deg(\Ad(\cE)_{<0})<\deg(\Ad(\cE)_{>0}).
\end{equation}If $\a\ne0$, then both $\Ad(\cE)_{<0}$ and $\Ad(\cE)_{>0}$ have the same positive rank. However, \eqref{deg<0} implies that $\deg(\Ad(\cE)_{>0})<0<\deg(\Ad(\cE)_{<0})$, which  contradicts \eqref{deg AdE}. We thus conclude that $\a$ has to be zero. This shows the vanishing \eqref{inf auto vanish}, which proves that $\cM_G(\bI_x)^{\tst}$ is a Deligne-Mumford stack.

\subsection{Smoothness} We make the same assumption on $p$ as in \eqref{p>h}. We prove that the map $\phi^{\tst}_x: \cM_G(\bI_{x})^{\tst}\to \frt^*$ is smooth.

By the contracting nature of the $\Gm$-action on Higgs fields by scaling, it suffices to show that the zero fiber $(\phi^{\tst}_x)^{-1}(0)=\cM_G(\bI_x; 0)^{\tst}$ is smooth over $k$. Note that $\cM_G(\bI_x; 0)$ is the (classical) cotangent stack of $\Bun_G(\bI_x)$, hence $\cM_G(\bI_x; 0)$ has a self-dual tangent complex. The obstruction for infinitesimal deformations of $(\cE,\cE^B_x,\ph)\in \cM_G(\bI_x; 0)$ lies in the dual to the space of infinitesimal automorphisms of $(\cE,\cE^B_x,\ph)$, modulo $\frz$, which is shown to vanish when $(\cE,\cE^B_x,\ph)$ is $\th$-stable by \eqref{inf auto vanish}. Therefore $\cM_G(\bI_x;0)^{\tst}$ is smooth. This shows that $\phi^{\tst}_x$ is smooth. 

\subsection{Separatedness of Hitchin map}\label{a:sep}
We show that $h^{\tst}$ is separated by the valuative criterion. This works for all characteristics.

The argument below works in the relative setting, which will be needed in showing properness.

\sss{Relative setup}\label{sss:rel setting}
Let $S$ be a stack of finite type over $\ZZ$. Let $\pi: X\to S$ be a relative curve, i.e., a smooth proper morphism with connected geometric fibers of dimension one. Suppose we are given a section $\s: S\to X$. In this situation we can define a moduli stack
\begin{equation*}
    \Pi: \cM_{G, \pi}(\bI_\s)\to S
\end{equation*}
whose fiber over a geometric point $s\in S$ is the moduli stack $\cM_{G,X_s}(\bI_{\s(s)})$ of parabolic  $G$-Higgs bundles on the curve $X_s$.

\sss{Uniqueness part of valuative criterion} We shall show that $\cM_{G, \pi}(\bI_\s)^{\tst}$ is separated by the valuative criterion. 

Let  $R$ be a complete DVR with fraction field $K$ and a uniformizer $\pi$. By base changing the situation to an algebraic closure of the residue field of $R$, we may assume that the residue field of $R$ is $k$, i.e., $R\cong k\tl{\pi}$. 

Let $\Spec R\to S$ be a map. Denote by $X_{K}$ (resp. $X_{0}$) the special fiber (resp. generic fiber) of $X_{R}=X\ot_{k}R$. Let $x_{R}=\{x\}\times\Spec R\subset X_{R}$, and denote by $x_{K}$ and $x_{0}$ its generic and special points. 

Given two $\th$-stable parabolic Higgs bundles $(\cE, \cE^{B}_{x}, \ph)$ and $(\cF, \cF^{B}_{x}, \psi)$ over $X_{R}$, together with an isomorphism $\io: (\cE, \cE^{B}_{x}, \ph)|_{X_{K}}\cong (\cF, \cF^{B}_{x}, \psi)|_{X_{K}}$, we need to show that $\io$ can be extended to an isomorphism $\cE\isom \cF$ (which is automatically compatible with the Higgs fields and Borel reductions). 

Let $\y_{0}\in X_{0}$ be the generic point. Let $A$ be the local ring of $X_{R}$ at $\y_{0}$, and $\wh A$ be its completion. So $\wh A\cong F\tl{\pi}$ is a complete DVR with residue field $F=k(X)$ and fraction field $\Frac(\wh A)\cong F\lr{\pi}$. Then $\io$ gives an isomorphism of $\cE_{\wh A}=\cE|_{\Spec \wh A}$ and $\cF_{\wh A}=\cF|_{\Spec \wh A}$ over the punctured disc $\Spec \Frac(\wh A)$. Their relative position is given by a dominant coweight $\l\in \xcoch(T)^{+}$. More precisely, if we choose trivializations of $\cE_{\wh A}$ and $\cF_{\wh A}$, then $\io$ is given by an element $g\in G(\Frac{\wh A})$. By the Cartan decomposition $G(\Frac{\wh A})=\sqcup_{\l\in \xcoch(T)^{+}}G(\wh A)\pi^{\l}G(\wh A)$, and upon changing the trivializations of $\cE_{\wh A}$ and $\cF_{\wh A}$, we may assume $g=\pi^{\l}$ for a unique dominant coweight $\l\in \xcoch(T)^{+}$. Let $P$ be the standard parabolic subgroup of $G$ containing $T$ with roots $\{\a\in \Phi(G,T)|\j{\l, \a}\ge 0\}$; let $\ov P$ be the opposite parabolic containing $T$ with roots $\{\a\in \Phi(G,T)|\j{\l, \a}\le 0\}$, and $L=P\cap \ov P$. Consider the stabilizer $G(\wh A)\cap \Ad(\pi^{\l})G(\wh A)$ of $\pi^{\l}$ under the left and right translation action of $G(\wh A)$, presented as a subgroup of the left translation. Note that the projection $G(\wh A)\cap \Ad(\pi^{\l})G(\wh A)\subset G(\wh A)\surj G(F)$ has image $\ov P(F)$, and the other projection $G(\wh A)\cap \Ad(\pi^{\l})G(\wh A)\to G(\wh A)\surj G(F)$ (given by $g\mapsto \Ad(\pi^{-\l})g$ then mod $\pi$) has image $P(F)$.   This implies that $\cE_{\y_{0}}$ carries a canonical $P$-reduction $\cE_{\y_{0}, P}$, and $\cF_{\y_{0}}$ carries a canonical $\ov P$ reduction $\cF_{\y_{0},\ov P}$, together with a canonical isomorphism of the induced $L$-bundles  $\io_{L}: \cE_{\y_{0}, L}\isom  \cF_{\y_{0}, L}$. We saturate these parabolic reductions to the $P$-reduction $\cE_{0,P}$ of $\cE_{0}=\cE|_{X_{0}}$, and to the $\ov P$-reduction $\cF_{0,\ov P}$ of $\cF_{0}=\cF|_{X_{0}}$. 

After trivializing $\om_X|_{\y_{0}}$ and under the chosen trivializations of $\cE_{\wh A}$ and $\cF_{\wh A}$, the Higgs fields $\ph$ and $\psi$ become elements $\ph_{\wh A}\in \frg^{*}(\wh A)$ and $\psi_{\wh A}\in \frg^{*}(\wh A)$. Since $\io$ is compatible with the Higgs fields, we have that $\Ad^{*}(\pi^{\l})\ph_{\wh A}=\psi_{\wh A}\in \frg^{*}(\wh A)\cap \Ad^{*}(\pi^{\l})\frg^{*}(\wh A)$. Now $\psi|_{\y_{0}}\in \frg^{*}(F)$ is the reduction of $\psi$ mod $\pi$, which then lies in the image of $\frg^{*}(\wh A)\cap \Ad^{*}(\pi^{\l})\frg^{*}(\wh A)\subset \frg^{*}(\wh A)\xr{\mod \pi}\frg^{*}(F)$, which is $\frn_{\ov P}^{\bot}(F)$; similarly, $\ph|_{\y_{0}}$ lies in  $\frn_{P}^{\bot}(F)$. By saturation, we conclude that $\ph_{0}=\ph|_{X_{0}}\in \G(X_{0}, \cE_{0,P, \frn_{P}^{\bot}}\ot \om_{X_{0}}(x))$ (i.e. it preserves the $P$-reduction of $\cE_{0}$), and $\psi_{0}=\psi|_{\y_{0}}$ preserves the $\ov P$-reduction of $\cF_{0}$.

For each representation $V$ of $G$ consider the isomorphism $\io_{V,K}: \cE_{V, K}\isom \cF_{V,K}$ of the induced vector bundles over $X_{K}$. Let $T_{L}=L/L^{\der}$. Let $\{\a_{j}|j\in J\}$ be the set of simple roots of $L$. Now assume $V=V_{\mu}$ is the Weyl module with lowest weight $\mu$ such that $\mu\in \xch(T_{L})\subset \xch(T)$ (i.e., $\j{\mu, \a^{\vee}_{j}}=0$ for all $j\in J$, and $\j{\mu,\a^{\vee}_{i}}<0$ for all $i\notin J$). Note that any other weight $\mu'$ of $V_{\mu}$ is of the form $\mu+\sum_{i\in I}n_{i}\a_{i}$ (where $n_{i}\ge0$) such that $n_{i}>0$ for at least one $i\notin J$. This implies $\j{\l,\mu}<\j{\l,\mu'}$ for all weights $\mu'\ne \mu$ of $V_{\mu}$. 
 
Then $\io_{V,K}$ extends to a map of $\wh A$-modules
\begin{equation*}
\io_{V,\wh A}: \cE_{V, \wh A}\to  \pi^{\j{\l, \mu}}\cF_{V,\wh A}
\end{equation*}
 whose reduction mod $\pi$ (after dividing by $\pi^{\j{\l, \mu}}$) takes the form
\begin{equation*}
\io_{V,\y_{0}}: \cE_{V,\wh A}/\pi=\cE_{V, \y_{0}}\surj \cE_{\y_{0}, P, \mu}\isom \cF_{\y_{0}, \ov P, \mu}\incl \cF_{V, \y_{0}}=\cF_{V,\wh A}/\pi.
\end{equation*}
Here $\cE_{\y_{0}, P, \mu}$ is the $1$-dimensional vector space over $F$ induced from the $P$-torsor $\cE_{\y_{0}}$ over $F$ and the character $\mu: P\to L\to \Gm$; and $\cF_{\y_{0}, \ov P, \mu}$ is defined similarly. Since the complement of $\{\y_{0}\}\cup X_{K}$ in $X_{R}$ has codimension two, the map $\io_{V,\wh A}$ extends to a map of coherent sheaves
\begin{equation*}
\io_{V,R}: \cE_{V}\to  \pi^{\j{\l, \mu}}\cF_{V}=\cF_{V}(-\j{\l,\mu}X_{0})
\end{equation*}
whose restriction to $X_{0}$ 
\begin{equation}\label{ioV0}
\io_{V,0}: \cE_{0,V}=\cE_{V}|_{X_{0}}\surj \cE_{0,P, \mu}\to \cF_{0, \ov P, \mu}\incl \cF_{0,V}
\end{equation}
restricts to $\io_{V,\y_{0}}$ over $\y_{0}$. In particular, we conclude that 
$\deg\cE_{0,P, \mu}\le \deg\cF_{0, \ov P, \mu}$ for all anti-dominant $\mu\in \xch(T_{L})$. This implies
\begin{equation}\label{degL diff lower}
\deg_{L}(\cE_{0,L})-\deg_{L}(\cF_{0,L})\in \frC'_{P}.
\end{equation}
 
On the other hand, we write $\th_{\cE}=\th(\cE_{0,P}, \cE^{B}_{x_{0}})$, $\th_{\cF}=\th(\cF_{0,\ov P}, \cF^{B}_{x_{0}})$. Note that both points are sufficiently close to $0\in \xcoch(T_{L})_{\QQ}$. Since $(\cE_{0}, \cE_{x_{0}}^{B}, \ph_{0})$ is $\th$-stable, we have
\begin{equation*}
\deg_{L}(\cE_{0,L})+\th_{\cE}\in -\frC'^{\c}_{P}.
\end{equation*}
Similarly,  $(\cF_{0}, \cF_{x_{0}}^{B}, \psi_{0})$ is $\th$-stable, so we have
\begin{equation*}
\deg_{L}(\cF_{0,L})+\th_{\cF}\in -\frC'^{\c}_{\ov P}=\frC'^{\c}_{P}.
\end{equation*}
Combining the two, we get
\begin{equation}\label{degL diff upper}
\deg_{L}(\cE_{0,L})-\deg_{L}(\cF_{0,L})\in -\frC'^{\c}_{P}-\th_{\cE}+\th_{\cF}.
\end{equation}
Combining \eqref{degL diff lower} and \eqref{degL diff upper} we get
\begin{equation*}
\deg_{L}(\cE_{0,L})-\deg_{L}(\cF_{0,L})\in\frC'_{P}\cap(-\frC'^{\c}_{P}-\th_{\cE}+\th_{\cF}).
\end{equation*}
Since $-\th_{\cE}+\th_{\cF}$ is sufficiently close to $0$, the only possible integral point on the right side is $0$. If this happens, we must have:
\begin{equation}\label{degL eq}
\deg_{L}(\cE_{0,L})=\deg_{L}(\cF_{0,L})
\end{equation}
and
\begin{equation}\label{thEF}
\th_{\cE}-\th_{\cF}\in -\frC'^{\c}_{P}.
\end{equation}
When \eqref{degL eq} holds, the map \eqref{ioV0} must be everywhere nonzero on $X_{0}$, and in particular at $x_{0}$. For $V=V_{\mu}$ (lowest weight $\mu$ such that $\mu\in \xch(T_{L})\subset \xch(T)$ as before), the $B$-reduction $\cE^{B}_{x_{0}}$ gives a filtration $F^{\chi}\cE_{V,x_{R}}$ on $\cE_{V,x_{R}}$ indexed by the weights $\chi$ of $V$ under the partial order given by the positive coroots, such that lowest weight is the last quotient. Let $\chi\in \xch(T)$ be the smallest step such that $F^{\chi}\cE_{V,x_{0}}$ surjects on the quotient $\cE_{x_{0}, P, \mu}=\cE_{0,P,\mu}|_{x_{0}}$. Then under the canonical map $\pi_{\cE}:=\pi(\cE_{x_{0}, P},\cE^{B}_{x_{0}}): T\to T_{L}$, which induces the dual map $\pi_{\cE}^{*}: \xch(T_{L})\to \xch(T)$, we have $\pi_{\cE}^{*}(\mu)=\chi$. Similarly, $\cF^{B}_{x_{0}}$ induces a filtration $\{F^{\psi}\cF_{V,x_{R}}\}$ on $\cF_{V,x_{R}}$, and if we let $\psi\in\xch(T)$ be the smallest step that contains $\cF_{x_{0}, \ov P, \mu}=\cF_{0,\ov P, \mu}|_{x_{0}}$, then $\pi_{\cF}^{*}(\mu)=\psi$. 

Since $\io|_{x_{K}}$ preserves the $B$-reductions, $\io_{V,K}|_{x_{K}}$ maps $F^{\chi}\cE_{V,x_{K}}$ to $F^{\chi}\cF_{V,x_{K}}$. By saturation,  $\io_{V,R}|_{x_{R}}$ maps $F^{\chi}\cE_{V,x_{R}}$ to $\pi^{\j{\l,\mu}}F^{\chi}\cF_{V,x_{R}}$. After reduction mod $\pi$ we conclude that $\io_{V,0}|_{x_{0}}$ maps $F^{\chi}\cE_{V,x_{0}}$ to $F^{\chi}\cF_{V,x_{0}}$. Since the map \eqref{ioV0} is nonzero at $x_{0}$, $\io_{V,x_{0}}(F^{\chi}\cE_{V,x_{0}})=\cF_{x_{0},\ov P,\mu}$, in particular, $\cF_{x_{0},\ov P,\mu}\subset F^{\chi}\cF_{V,x_{0}}$. This implies $\chi'-\chi$ is a sum of positive roots by the minimality of $\chi'$. Now by definition
\begin{equation*}
\j{\th_{\cE}, \mu}=\j{\pi_{\cE}(\th), \mu}=\j{\th, \pi_{\cE}^{*}\mu}=\j{\th,\chi}
\end{equation*}
and similarly
\begin{equation*}
\j{\th_{\cF}, \mu}=\j{\th,\chi'}.
\end{equation*}
Therefore
\begin{equation*}
\j{\th_{\cE}-\th_{\cF}, \mu}=\j{\th, \chi-\chi'}\le 0
\end{equation*}
since $\th$ is assumed to be dominant. However, \eqref{thEF} implies the opposite inequality 
\begin{equation*}
\j{\th_{\cE}-\th_{\cF}, \mu}<0
\end{equation*}
provided $\mu$ is anti-dominant and nonzero. This gives a contradiction, provided that $T_{L}$ is nontrivial. When $T_{L}=1$, we have $P=G$ which is possible only when $\l=0$, i.e., $\io$ extends to an isomorphism over $X_{K}\cup \{\y_{0}\}$, hence to an isomorphism over the whole $X_{R}$. This proves the uniqueness part of the valuative criterion. 

\subsection{Properness of parabolic Hitchin map}\label{a:proper}
We show that there exists a number $N$ depending only on the genus $g$ of $X$ and the Dynkin diagram of $G$, such that whenever $p>N$, the map $h^{\tst}$ is proper. We use a general argument to reduce the statement to the characteristic zero case where the properness of the parabolic Hitchin map is known by Faltings \cite{F}. Hence our argument does not give an explicit bound for $p$. We hope that one can obtain an explicit and reasonably small bound for $p$ along the lines of the arguments in \cite{AHLH}.

We work in a specific relative setting as in \S\ref{sss:rel setting}. Let $\pi^{\univ}_g: \frC_g\to \frM_g$ be the universal curve over the moduli stack $\frM_g$ of smooth connected genus $g$ curves. Let $S$ be a Deligne-Mumford stack of finite type over $\ZZ$ with a surjective map $\nu: S\to \frC_g$. Let $\pi: X=S\times_{\frM_g}\frC_g\to S$ be the base-changed curve, equipped with a section given by the graph of $\nu: S\to \frC_g$. The relative parabolic Hitchin moduli stack $\cM_{G,\pi}G(\bI_\s)$ is defined. We also have the family version of the Hitchin base $\cA_{G,\pi}(\bI_\s)\to S$ (an affine morphism) and a Hitchin map
\begin{equation*}
    h_\pi: \cM_{G, \pi}(\bI_\s)\to \cA_{G, \pi}(\bI_\s).
\end{equation*}
The $\th$-stable locus $\cM_{G, \pi}(\bI_\s)^{\tst}\subset \cM_{G, \pi}(\bI_\s)$ is defined. Let 
\begin{equation*}
    h_\pi^{\tst}: \cM_{G, \pi}(\bI_\s)^{\tst}\to \cA_{G, \pi}(\bI_\s)
\end{equation*}
be the restriction of the Hitchin map.

\begin{lemma}\label{l:gen proper}
    There exists an open dense subset $U\subset \Spec \ZZ$ such that $h_\pi^{\tst}$ is proper when restricted over $U$.
\end{lemma}
\begin{proof}
    Now $\cM_{G, \pi}(\bI_\s)^{\tst}$ is Deligne-Mumford by \S\ref{a:DM}, and $h_\pi^{\tst}$ is of finite type and separated by \S\ref{a:sep}. By the Keel-Mori theorem \cite{KM} and \cite[Theorem 1.1]{C}, $\cM_{G, \pi}(\bI_\s)^{\tst}$ has a coarse moduli space $M$ together with a map $h_M: M\to A:=\cA_{G, \pi}(\bI_\s)$  that is also separated of finite type, and natural map $\cM_{G, \pi}(\bI_\s)^{\tst}\to M$ is proper. 

    By Nagata compactification for algebraic spaces \cite[Theorem 1.2.1]{CLO}, there exists a proper map $\ov M\to A$ from some algebraic space $\ov M$, together with an open embedding $M\incl \ov M$. We may assume $M$ is dense in $\ov M$. Let $Z=\ov M\bs M$, which is also proper over $A$. 

    Consider the $\QQ$-fiber 
    \begin{equation*}
        h_{M,\QQ}: M_\QQ\to A_\QQ
    \end{equation*}
    of $h_M$. By the properness of the parabolic Hitchin map in characteristic zero,  see \cite[Theorem II.4 and its parabolic version in p.561(i)]{F}, $h_{M,\QQ}$ is proper. Since $\ov M_\QQ$ is separated, $M_\QQ$ must be closed in $\ov M_\QQ$. Since $M_\QQ$ is also dense in $\ov M_\QQ$, we see that $M_\QQ=\ov M_\QQ$, i.e., $Z_\QQ=\vn$. Thus the image of $Z\to \Spec \ZZ$ is a finite set of primes, whose complement we take to be $U$. Now $M|_U=\ov M_U$ is proper over $A|_U$, proving the lemma.
\end{proof}

Finally, if $p=\ch(k)$ is a prime in $U$, and $X_0$ is a smooth projective connected curve over $k$ with a point $x$, we can find a $k$-point $s\in S$ such that $(X_0,x)\cong (X_s,\s(s))$. Lemma \ref{l:gen proper} restricted to the point $s$ then implies that $h^{\tst}:\cM_{G,X_0}(\bI_x)\to \cA_{G,X_0}(\bI_x)$ is proper. This finishes the proof of the properness of $h^{\tst}$ for large $p$.

\end{document}